\documentclass[11pt,oneside,english,reqno]{amsart}
\usepackage[T1]{fontenc}
\usepackage[latin9]{inputenc}
\usepackage[a4paper]{geometry}
\usepackage{mathrsfs}
\usepackage{amstext}
\usepackage{amsthm}
\usepackage{amssymb}
\usepackage{color}
\usepackage{hyperref}
\usepackage{dsfont}
\usepackage{enumerate}
\usepackage{verbatim} 
\usepackage{stmaryrd}
\usepackage{enumitem}
\usepackage[dvipsnames]{xcolor}
\usepackage{bbm}
\usepackage{dsfont}

\makeatletter
\numberwithin{equation}{section}
\numberwithin{figure}{section}
\theoremstyle{plain}
\newtheorem{thm}{\protect\theoremname}
\theoremstyle{definition}
\newtheorem{defn}[thm]{\protect\definitionname}
\theoremstyle{remark}
\newtheorem{rem}[thm]{\protect\remarkname}
\theoremstyle{plain}
\newtheorem{lem}[thm]{\protect\lemmaname}
\theoremstyle{plain}
\newtheorem{prop}[thm]{\protect\propositionname}
\theoremstyle{plain}
\newtheorem{cor}[thm]{\protect\corollaryname}
\theoremstyle{plain}
\newtheorem{ex}[thm]{\protect\examplename}
\theoremstyle{plain}
\newtheorem{ass}[thm]{\protect\assumptionname}
\numberwithin{thm}{section}
\makeatother

\usepackage{babel}
\providecommand{\corollaryname}{Corollary}
\providecommand{\definitionname}{Definition}
\providecommand{\lemmaname}{Lemma}
\providecommand{\propositionname}{Proposition}
\providecommand{\remarkname}{Remark}
\providecommand{\theoremname}{Theorem}
\providecommand{\examplename}{Example}
\providecommand{\assumptionname}{Assumption}

\newcommand{\red}[1]{{\color{red} #1}}
\newcommand{\blue}[1]{{\color{blue} #1}}

\newcommand{\cB}{\mathcal{B}}
\newcommand{\cC}{\mathcal{C}}
\newcommand{\cD}{\mathcal{D}}

\newcommand{\cF}{\mathcal{F}}
\newcommand{\cG}{\mathcal{G}}
\newcommand{\cH}{\mathcal{H}}
\newcommand{\cI}{\mathcal{I}}
\newcommand{\cJ}{\mathcal{J}}

\newcommand{\cL}{\mathcal{L}}

\newcommand{\cP}{\mathcal{P}}

\newcommand{\cS}{\mathcal{S}}

\newcommand{\EE}{\mathbb{E}}

\newcommand{\NN}{\mathbb{N}}

\newcommand{\PP}{\mathbb{P}}
\newcommand{\QQ}{\mathbb{Q}}
\newcommand{\RR}{\mathbb{R}}

\newcommand{\mathd}{\mathrm{d}}

\newcommand{\vertiii}[1]{{\left\vert\kern-0.25ex\left\vert\kern-0.25ex\left\vert #1 
    \right\vert\kern-0.25ex\right\vert\kern-0.25ex\right\vert}}

\newcommand{\eps}{\varepsilon}
\newcommand{\dd}{\mathop{}\!\mathrm{d}}

\newcommand{\ind}{\mathbbm{1}}
\let\div\relax
\DeclareMathOperator{\div}{div}

\long\def\lucio#1{{\color{blue}Lucio:\ #1}}
 
 \hypersetup{
     colorlinks   = true,
     citecolor    = blue,
     linkcolor    = blue
}

\title{Distribution dependent SDEs driven by additive fractional Brownian motion}
\author{ Lucio Galeati \and Fabian A. Harang \and Avi Mayorcas}
\date{\today}

\address{Lucio Galeati: 
Institute of Applied Mathematics, University of Bonn, 53115 Endenicher Allee 60, Bonn, Germany.
}
\email{lucio.galeati@iam.uni-bonn.de}
\address{Fabian A. Harang: 
 Department of Mathematics, University of Oslo, P.O. box 1053, Blindern, 0316, OSLO, Norway.}
\email{fabianah@math.uio.no} 
\address{Avi J. Mayorcas: Mathematical Institute, University of Oxford, Oxford, OX2 6GG, UK.}
\email{mayorcas@math.ox.ac.uk}


\thanks{
\emph{Acknowledgments.} 
FH gratefully acknowledges financial support from the STORM project 274410, funded by the Research Council of Norway.
LG is funded by the DFG under Germany's Excellence Strategy - GZ 2047/1, project-id 390685813.}

\begin{document}

\begin{abstract}
    We study distribution dependent stochastic differential equations with irregular, possibly distributional drift, driven by an additive fractional Brownian motion of Hurst parameter $H\in (0,1)$.
    We establish strong well-posedness under a variety of assumptions on the drift; these include the choice
    \[
    B(\cdot,\mu)=(f\ast \mu)(\cdot) + g(\cdot), \quad f,\,g\in B^\alpha_{\infty,\infty},\quad \alpha>1-\frac{1}{2H},
    \]
    thus extending the results by Catellier and Gubinelli \cite{Catellier2016} to the distribution dependent case.
    The proofs rely on some novel stability estimates for singular SDEs driven by fractional Brownian motion and the use of Wasserstein distances.
\end{abstract}

\maketitle

{
\hypersetup{linkcolor=black}
 \tableofcontents 
}

\section{Introduction}
In this work we consider a distribution dependent SDE (henceforth DDSDE) of the form
\begin{equation}\label{eq:GenMckeanSDE}
    X_t = \xi + \int_0^t B_s(X_s,\cL(X_s))\dd s + W_t
\end{equation}
where $B : \RR_+\times \RR^d \times \cP(\RR^d)\rightarrow \RR^d$, $\xi$ is an $\RR^d$-valued random variable and $W$ is a $\RR^d$-valued stochastic process independent of $\xi$.
The drift $B$ and the law of $(\xi,W)$ are prescribed, while the process $X$ is the unknown and $\cL(X_t)$ denotes the law of its marginal at time $t$.

Usually in the literature $W$ is sampled as a standard Brownian motion; in this case the DDSDE is also called a McKean--Vlasov SDE, after the pioneering work \cite{mckean1966class} where it was first introduced.

The importance of McKean--Vlasov equations is due to their connection to systems of $N$ particles subject to a mean field interaction of the form
\begin{equation}\label{eq:intro-particle-system}
    X^{i,N}_t = \xi^i + \int_0^t B_s\big(X^{i,N}_s, L^N\big(X^{(N)}_s\big)\big) \dd s + W^{i}_t, \quad L^N\big(X^{(N)}_t\big):=\frac{1}{N}\sum_{i=1}^N \delta_{X^{i,N}_t}
\end{equation}
where $(\xi^i,W^i)$ are typically taken to be i.i.d. copies of $(\xi,W)$ and $L^N\big(X^{(N)}_t\big)$ stands for the empirical measure of the system at time $t$.
One expects the DDSDE \eqref{eq:GenMckeanSDE} to be the mean field limit of \eqref{eq:intro-particle-system} in the sense that, as $N$ goes to infinity, $L^N\big(X^{(N)}_t\big)$ converges weakly to $\cL(X_t)$ with probability $1$.

Another feature of DDSDEs in the Brownian noise case is their connection to nonlinear Fokker--Planck PDEs (also called McKean--Vlasov equations) of the form
\begin{equation}\label{eq:intro-PDE}
    \partial_t \rho + \nabla \cdot ((B_t(\,\cdot\,, \rho)\, \rho) = \frac{1}{2} \Delta \rho, \quad \rho_0=\cL(\xi),
\end{equation}
which describe the evolution of the marginal $\rho_t=\cL(X_t)$;
in particular, both \eqref{eq:GenMckeanSDE} and \eqref{eq:intro-PDE} provide a macroscopic, compact description of the system \eqref{eq:intro-particle-system}, allowing one to reduce its complexity.
For this reason, DDSDEs have found applications in numerous fields, see the review \cite{jabin2017mean} and the references therein; let us also mention their connection to mean-field games \cite{lasry2007mean}.

Classical results concerning the well-posedness of the DDSDE \eqref{eq:GenMckeanSDE} and the mean-field limit property go back to Sznitman \cite{sznitman1991topics} and G\"artner \cite{Gartner1988}; in the last years the field has witnessed substantial contributions both from the analytic and probabilistic communities.
On the one hand, new methods based on entropy inequalities \cite{fournier2014propagation, jabin2018quantitative, bresch2019mean} and modulated energy methods \cite{serfaty2017mean, serfaty2020mean} have allowed for the rigorous derivation of mean field limits for fairly singular $B$;
while on the other, DDSDEs with irregular drifts are related to the flourishing field of \textit{regularization by noise} phenomena. The latter topic was initiated by Zvonkin \cite{Zvonkin1974} and Veretennikov \cite{Veretennikov1981} in the case of standard SDEs, see \cite{Flandoli2011} for a general overview; recently many authors have applied similar techniques in the DDSDE case, see for instance \cite{BauerMBrandisProske2018,mishura2016existence,rockner2018well,chaudru2020strong,huang2019distribution}.\\

Contrary to the previously mentioned works, here we will study DDSDEs in which $W$ is sampled as a \textit{fractional Brownian motion} (fBm for short) of Hurst parameter $H\in (0,1)$.
Our main reasons for doing so are the following:
\begin{enumerate}
    \item[1.] It was shown in \cite{coghi2020pathwise}, revisiting the ideas of Tanaka \cite{tanaka1984Limits}, that for Lipschitz $B$ the mean-field limit of \eqref{eq:intro-particle-system} to \eqref{eq:GenMckeanSDE} holds for \textit{any} choice of the process $W$, regardless of it being Markov or a semimartingale.
    In particular the DDSDE has a physical meaning and still provides a compact description of a much more complex system of interacting particles.
    \item[2.] Several regularization by noise results for standard SDE are available for $W$ sampled as an fBm (or similar fractional processes), see \cite{nualart2002regularization, Catellier2016, le2020stochastic, amine2017c, banos2019strong} for a short selection.
\end{enumerate}
In light of Point 2. above, it is natural to expect similar results to hold for DDSDEs with singular (possibly even distributional in space) drifts and $W$ sampled as an fBm;
by Point 1., they are relevant in the study of particle systems with singular interactions (for instance with a discontinuity at the origin, as typical of Coulomb and Riesz-type potentials).

Let us mention that there is a certain degree of arbitrariness in choosing $W$ to be sampled as an fBm, as one could consider other non-Markovian, non-martingale processes.
We believe our choice to be simple enough while at the same time representing what one might expect for a larger class of processes (e.g. Gaussian processes satisfying a local non-determinism condition).
In this sense, this work also serves as a comparison to the results from \cite{GalHarMay_21benchmark}, where we explored in detail the DDSDE \eqref{eq:GenMckeanSDE} in the opposite regime where no assumption whatsoever is imposed on $W$, thus no regularization can be observed.\\

Despite the above motivations, singular DDSDEs driven by fBm (or similar fractional processes) so far have not received the same attention as their Brownian counterparts; to the best of our knowledge, the only other work treating these kind of equations is \cite{bauer2019mckean}. 

One possible reason for this is the substantial new difficulties presented by such equations: fBm with parameter $H\neq 1/2$ is neither a Markov process, nor a semimartingale, so techniques based on It\^o calculus are not applicable.
This includes in particular the connection to parabolic semigroups, the martingale problem formulation and the use of Zvonkin transform (or It\^o--Tanaka trick), all techniques used extensively in the aforementioned works in the Brownian case.
It also prevents the use of standard arguments, which typically rely on establishing uniqueness of the law $\rho_t=\cL(X_t)$ through PDE analysis of \eqref{eq:intro-PDE} and then fixing the law in the DDSDE and treating it as a standard SDE.\\

Treating DDSDEs driven by fBm thus requires a novel set of tools and ideas; our strategy in this paper builds on the work of Catellier and Gubinelli \cite{Catellier2016}, which represented a major breakthrough in the study of standard SDEs driven by fBm of the form
\begin{equation}\label{SDE}
X_t = \xi + \int_0^t b_s(X_s)\mathd s + W_t.
\end{equation}
Therein the authors develop a \textit{pathwise approach} to the equation, based on \textit{nonlinear Young integrals} and \textit{Girsanov transform}, that allows to give meaning to \eqref{SDE} and establish its \textit{path-by-path uniqueness}, for drifts $b$ of poor regularity, possibly even distributional.
Their results and techniques have been revisited in subsequent works \cite{galeati2020noiseless,harang2020cinfinity,galeati2020regularization,harang2020pathwise}; in general it suffices to require
\begin{equation}\label{regularity condition}
b \in
\begin{cases}
L^q_T B^\alpha_{\infty,\infty} & \text{with } \alpha>1-\frac{1}{2H}+\frac{1}{Hq}\quad \ \, \text{if } H\leq 1/2\\
C^{\alpha H}_T C^0_x \cap C^0_T C^\alpha_x \quad &\text{with } \alpha>1-\frac{1}{2H}\qquad\qquad  \text{if } H>1/2
\end{cases}
\end{equation}
see for instance Theorem 15 and Corollary 2 from \cite{galeati2020noiseless}.
Here $B^\alpha_{\infty,\infty}$ denote Besov-H\"older spaces; see Section \ref{subsec:notation} below for the relevant definitions and notations in use throughout the article.
%
%

For the sake of exposition, let us ignore for the moment the additional time regularity required in \eqref{regularity condition} in the case $H>1/2$, since it is mostly of a technical nature;
then condition \eqref{regularity condition} roughly amounts to the drift $b$ enjoying a spatial regularity $B^{\alpha}_{\infty,\infty}$ with $\alpha>1-1/(2H)$.
Observe that for all $H\in (0,1)$ this includes values $\alpha<1/2$, while for $H<1/2$ we are even allowed to take $\alpha<0$, namely distributional $b$.
To the best of our knowledge, no work after \cite{Catellier2016} has improved on the allowed range of $\alpha$.\\


With the above theory at hand, we can interpret the DDSDE \eqref{eq:GenMckeanSDE} by rewriting it as
\begin{align*}
    X_t = \xi + \int_0^t \bar{b}_s(X_s)\dd s + W_t, \quad \bar{b}_t(\cdot):= B_t(\cdot,\cL(X_t));
\end{align*}
namely, $X$ solves the SDE with drift $\bar{b}$, in the Catellier-Gubinelli sense, where $\bar{b}$ depends in a nontrivial way on the law of $X$ itself.
This interpretation comes with a natural fixed point formulation: given a process $X$, we can associate to it a ``flow of measures'' $\mu_t=\cL(X_t)$ and a drift $b^\mu_t := B_t(\cdot, \mu_t)$, then solve the associated SDE, which gives a new process $Y=\cI(X)$;
thus $X$ is a solution to \eqref{eq:GenMckeanSDE} if and only if it is a fixed point for $\cI$.

Alternatively, one could start with the flow of measures $\mu_\cdot=\{\mu_t\}_{t\in [0,T]}$ and set up the fixed point procedure for this object, by defining $\cJ(\mu_\cdot)_t=\cL(X_t)$ for $X$ solution to $b^\mu$.
These two interpretations are in fact equivalent: once $\mu_\cdot$ is completely determined, the DDSDE reduces to a standard SDE with fixed drift $b^\mu$, to which the previous results can be applied; see Lemma \ref{lem:SolDefEquiv} for more details.
Throughout the article we will exploit both interpretations whenever useful.\\

Given the above interpretation, we need two main ingredients to develop a solution theory:
\begin{enumerate}
    \item[1.] Firstly, $B$ must have the properties that $b^\mu$ satisfies \eqref{regularity condition} for any $\mu_\cdot$ of interest and that the solution-to-drift map $X(\mapsto\mu_\cdot)\mapsto b^\mu$ is Lipschitz in suitable topology.
    \item[2.] Secondly, we must develop stability estimates for the drift-to-solution map $b\mapsto Y$, in an appropriate topology that complements the stability of $\mu\mapsto b^\mu$.
\end{enumerate}
Once these points are established, the contractivity of the overall map $X\mapsto b^\mu\mapsto \cI(X)$ follows.

There are however major problems with the program outlined above;
to describe them without too many technicalities, let us consider here the most relevant case $B(\mu)=f\ast \mu+g$ for time homogeneous $f,g\in B^\alpha_{\infty,\infty}$, $\alpha>1-1/(2H)$.
In this case, the map $\mu\mapsto b^{\mu}$ is naturally Lipschitz in the total variation topology, in the sense that
\begin{align*}
    \| B(\mu^1)-B(\mu^2)\|_{B^{\alpha}_{\infty,\infty}} \lesssim \| \mu^1-\mu^2\|_{TV};
\end{align*}
however due to the lack of an underlying parabolic PDE \eqref{eq:intro-PDE} (and the associated maximum principle) in the fBm setting, it is no obvious how to control the drift-to-solution map $b\mapsto Y$ in this topology, i.e. how to bound $\| \cL(Y^1_t)-\cL(Y^2_t)\|_{TV}$ as a function of $\| b^1-b^2\|_{B^\alpha_{\infty,\infty}}$.

One of the main intuitions of the current work, which allows us to overcome this difficulty, is the understanding that although the regularity $B^\alpha_{\infty,\infty}$ is needed in order to solve the SDE \eqref{SDE}, one may establish stability estimates in the weaker norm $B^{\alpha-1}_{\infty,\infty}$.
Roughly speaking, given two solutions $X^1,X^2$ to \eqref{SDE} associated to different initial data and drifts $(\xi^i,b^i)$, for any $p\in [1,\infty)$ we have
\begin{equation}\label{eq:intro-stability}
    \EE\Big[\sup_{t\in [0,T]} |X^1_t-X^2_t|^p\Big]^{1/p} \lesssim \EE\big[\,|\xi^1-\xi^2|^p\big]^{1/p} + \| b^1-b^2\|_{B^{\alpha-1}_{\infty,\infty}}
\end{equation}
see Theorem \ref{thm:fBmSDEWellPosed} and Corollary \ref{cor:RndDataWellPosed} for the rigorous statements.
This property is a natural analogous to standard ODE theory, where solvability requires $b$ Lipschitz, but stability estimates are in the supremum norm.

In our setting, it implies that $B$ only needs to enjoy some \textit{multiscale regularity} of the form
\begin{align*}
    \|B(\mu)\|_{B^\alpha_{\infty,\infty}}\lesssim 1, \quad \| B(\mu^1)-B(\mu^2)\|_{B^{\alpha-1}_{\infty,\infty}} \lesssim d(\mu^1,\mu^2)
\end{align*}
for another notion of distance $d(\mu^1,\mu^2)$, possibly different from the total variation one.
The right choice for $d$ turns out to be the family of $p$-Wasserstein distances $d_p(\mu^1,\mu^2)$, which complements the bound \eqref{eq:intro-stability} thanks to the basic property $d_p(\cL(X^1_t),\cL(X^2_t))\leq \EE[|X^1_t-X^2_t|^p]^{1/p}$.

Overall, the newly found stability estimate \eqref{eq:intro-stability} and the use of Wasserstein distance allow us to fulfill Points 1.-2. outlined above and to solve the DDSDE \eqref{eq:DDSDE} for a large class of drifts $B$, see Theorems \ref{thm:main_thm1} and \ref{thm:main_thm2} for the precise statements; this includes the case $B(\mu)=f\ast \mu +g$ mentioned above.

For the sake of this preliminary discussion we have ignored the time regularity requirement in \eqref{regularity condition}, but it does indeed play a relevant role, making the proofs a bit more technical and requiring us to treat the cases $H>1/2$ and $H\leq 1/2$ slightly differently; see Section \ref{sec:MainResultsProofs} for more details.\\

Let us stress that, since we are not allowed to use the same tools as in the Brownian setting, our results are not optimal for the choice $H=1/2$, sharper ones being available for instance in \cite{rockner2018well,  huang2019distribution}.
Nevertheless, they still provide some new insights, with the stability estimate \eqref{eq:intro-stability} being new in this setting as well.
This also partially answers the ongoing debate from \cite{rockner2018well,huang2019distribution,huang2020mckean} on whether the drift $B$ should be taken Lipschitz continuous in the measure argument $\mu$ w.r.t. the total variation distance, the Wasserstein one or a weighted mix of the two: the use of Wasserstein distance allows the drift to be Lipschitz continuous in the different regularity scale $B^{\alpha-1}_{\infty,\infty}$, which is strictly negative in the regime $\alpha \in (0,1)$, which is admissible in \eqref{regularity condition} for $H=1/2$.\\

A major open problem coming from this work is the mean-field convergence (and associated propagation of chaos property) of the particle system \eqref{eq:intro-particle-system} to  \eqref{eq:GenMckeanSDE}, for the class of singular drifts for which we establish well-posedness of the DDSDE in Theorem \ref{thm:main_thm1}.
Our techniques are currently not enough to give a full answer; recently, several authors have investigated the Brownian setting using alternative tools based on Girsanov theorem and Large Deviations, see \cite{lacker2018strong,jabir2019rate, tomasevic2020propagation, hoeksema2020large}.
Contrary to It\^o calculus, these tools are available for fBm as well, thus we hope they may be of help in future investigations.

Another interesting question posed by the current work is whether our results can be further improved, in the sense of allowing values of $\alpha<1-1/(2H)$, at least in some special cases.
Theorems \ref{thm:main_thm3} and \ref{thm:main_thm4} suggest an affirmative answer for convolutional drifts $B(\mu)=b\ast \mu$, see also the discussion at the beginning of Section \ref{sec:refined-convolution};
this is in analogy with the Brownian case, where standard SDE theory requires roughly $b\in L^\infty_x$, but the nonlinear PDE \eqref{eq:intro-PDE} can be solved for roughly $b\in W^{-1,\infty}_x$.\\

We conclude this introduction with the structure of the paper.
In Section \ref{subsec:notation} we introduce all relevant notations adopted in the paper and recall some well-known facts. Section \ref{sec:main-results} contains all our main results and Section \ref{subsec:examples} relevant examples of drifts $B$ satisfying them.
We present in detail the Catellier--Gubinelli theory of SDEs driven by fBm in Section \ref{sec:recap-SDE}, where we prove our main stability results (Theorem \ref{thm:fBmSDEWellPosed} and Corollary \ref{cor:RndDataWellPosed} from Section \ref{subsec:stability}) as well as some new auxiliary results on the regularity of the law of solutions (Section \ref{subsec:reg-estim}).
Sections \ref{sec:MainResultsProofs} and \ref{sec:refined-convolution} contain the proofs of our main results, respectively Theorems \ref{thm:main_thm1}, \ref{thm:main_thm2}, \ref{thm:main_thm3} and \ref{thm:main_thm4}.
Finally, we have included in Appendix \ref{app:UsefulLemmas} a collection of useful analytic lemmas used throughout the paper.
\subsection{Notations, conventions and well-known facts}\label{subsec:notation}

Throughout the article we will always work on a finite time interval $[0,T]$, although arbitrarily large; we will never deal with estimates on the infinite interval $[0,+\infty)$.
We write $a\lesssim b$ whenever there exists a constant $C>0$ such that $a\leq b$. To stress the dependence $C=C(\lambda)$ on a particular parameter $\lambda$, we will write $a\lesssim_\lambda b$. For $p \in [1,\infty]$ and where it will not cause confusion, we write $p'$ to denote the dual exponent to $p$, that is $1/p+1/p'=1$, with the interpretation $p=1\iff p'=\infty$.

Throughout the article, whenever not mentioned explicitly, we will consider an underlying probability space $(\Omega,\cF,\PP)$; any $\sigma$-algebra appearing is assumed to be $\PP$-complete. If $\Omega$ has a topological structure, then $\cB(\Omega)$ denotes its Borel $\sigma$-algebra (again up to $\PP$-completion).

We denote by $\EE_\PP$, or simply $\EE$, expectation w.r.t. $\PP$;
Given a Banach space $E$ and $p\in [1,\infty]$, we will frequently consider $E$-valued random variables $X$ in the space $L^p_\Omega E:=L^p(\Omega,\cF,\PP;E)$, with norm $\| X\|_{L^p_\Omega} = \EE[ \| X\|_E^p]^{1/p}$ (essential supremum if $p=\infty$).

We denote by $\cL_\PP(X)$, or simply $\cL(X)$, the law of $X$ on $E$, namely the pushforward measure $\PP\circ X^{-1} = X \sharp \PP$; more generally, we adopt the notation $F\sharp \mu$ for the pushforward of a measure $\mu$ under a measurable map $F$. Given a measure $\mu \in \cP(C_T)$,we mention in particular the pushforward $\mu_t := e_t \sharp \mu$ where $e_t (h) = h_t$ denotes the evaluation map, $e_t:C_T\to\RR^d$.

\subsubsection{Function spaces on $[0,T]$}\label{not:function spaces-1}

Given a metric space $(M,d_M)$, we denote by $C_T M = C([0,T];M)$ the space of all continuous functions $f:[0,T]\to M$; for $\gamma\in (0,1)$, we set $C^\gamma_T M = C^\gamma([0,T];M)$ to be the subset of $\gamma$-H\"older continuous functions, namely
\begin{align*}
    \llbracket f\rrbracket_{\gamma,M}:=\sup_{s\neq t\in [0,T]} \frac{d_M(f_t,f_s)}{|t-s|^\gamma}<\infty.
\end{align*}
If $(E,\| \cdot\|_E)$ is a Banach space, then $C_T E$ and $C^\gamma_T E$ are Banach spaces with norms
\begin{align*}
    \| f\|_{C_T E} =\sup_{t\in [0,T]} \| f_t\|_E, \quad
    \| f\|_{C^\gamma_T E} = \| f\|_{C_T E} + \llbracket f\rrbracket_{\gamma,E}.
\end{align*}

In the case $E=\RR^n$ for some $n\in\NN$, whenever it doesn't create confusion we will simply use $C_T$, $C^\gamma_T$ and $\| f\|_\gamma$ in place of $C_T \RR^d$, $C^\gamma_T \RR^d$, $\| f\|_{C^\gamma_T}$; moreover for any $[s,t]\subset [0,T]$ we set
\begin{equation*}
    \llbracket f\rrbracket_{\gamma,[a,b]}:=\sup_{s\neq t\in [a,b]} \frac{|f_t-f_s|}{|t-s|^\gamma}.
\end{equation*}

Given a Banach space $E$ and $q\in [1,\infty]$, we denote by $L^q_T E = L^q(0,T;E)$ the Bochner--Lebesgue space of strongly measurable $f:[0,T]\to E$ such that
\begin{align*}
    \| f\|_{L^q_T E} = \bigg( \int_0^T \| f_t\|_E^q \dd t\bigg)^{\frac{1}{q}} <\infty
\end{align*}
with usual modification for $q=\infty$; as before we write $L^q_T$ for $L^q_T \RR^n$.

\subsubsection{Function spaces on $\RR^d$}\label{not:function spaces-2}

Given $d, m\in\NN$, we denote by $C(\RR^d;\RR^m)$ the space of continuous, bounded functions $f:\RR^d\to \RR^m$, endowed with the supremum norm $\|  f\|_{C^0_x}$; whenever it doesn't create confusion we will simply write $C^0_x$.
$C^\infty_c = C^\infty_c(\RR^d;\RR^m)$, $C^n_x = C^n(\RR^d;\RR^m)$ denote respectively compactly supported smooth functions and $n$-times differentiable functions with continuous, bounded derivatives up to order $n$;
$\cS=\cS(\RR^d;\RR^m)$ denote Schwartz functions, $\cS'$ their dual.
Given $f$, we denote by $Df$ its Jacobian, i.e. the collection of first order derivatives $(\partial_j f_i)_{i,j}$, possibly interpreted in the distributional sense.
For $\alpha\in (0,1)$, $C^\alpha_x=C^\alpha(\RR^d;\RR^m)$ stand for the Banach space of H\"older continuous functions, with norm
\begin{equation*}
   \| f\|_{C^\alpha_x}:=  \| f\|_{C^0_x} + \llbracket f\rrbracket_{C^\alpha_x}, \quad
   \llbracket f\rrbracket_{C^\alpha_x}:=\sup_{x\neq y\in \RR^d} \frac{|f(x)-f(y)|}{|x-y|^\alpha}.
\end{equation*}
The definition of $C^\alpha_x$ extends canonically to $\alpha\in (1,+\infty)$ by imposing that $f\in C^\alpha_x$ if $f\in C^{\lfloor \alpha\rfloor}_x$ and its derivatives of order $\lfloor \alpha\rfloor$ belong to $C^{\alpha-\lfloor \alpha\rfloor}_x$, where $\lfloor \alpha\rfloor$ denotes the integer part of $\alpha$.
We denote by $C^\alpha_{loc}=C^\alpha_{loc}(\RR^d;\RR^n)$ the vector space of all continuous $f:\RR^d\to \RR^m$ such that $\varphi f\in C^\alpha_x$ for all $\varphi\in C^\infty_c$; we say that $f^n\to f$ in $C^\alpha_{loc}$ if $\varphi f^n\to \varphi f$ in $C^\alpha_x$ for all $\varphi\in C^\infty_c$.

Given $\alpha\in \RR$ and $p\in [1,\infty]$, we denote by $B^\alpha_{p,p}=B^\alpha_{p,p}(\RR^d;\RR^m)$ the associated (inhomogeneous) Besov space, given by distributions $f\in \cS'$ such that
\begin{align*}
    \| f\|_{B^\alpha_{p,p}} := \bigg( \sum_{n=-1}^{+\infty} 2^{\alpha n p} \| \Delta_n f\|_{L^p}^p\bigg)^{\frac{1}{p}} <\infty
\end{align*}
where $\Delta_n$ denote the Littlewood-Paley blocks associated to a partition of the unity.
We refer to the monograph \cite{BahCheDan} for details on Besov spaces; throughout the paper we will frequently employ their properties, like Besov embeddings, Bernstein estimates for $\Delta_n f$ or the regularity of $f\ast g$ for $f$, $g$ in different Besov spaces.
For $\alpha \in \RR_+\setminus \NN$, the spaces $C^\alpha_x$ and $B^\alpha_{\infty,\infty}$ coincide; however for clarity we will continue to write $\cC^{\alpha}$ for $\alpha\geq 0$ and $B^\alpha_{\infty,\infty}$ otherwise.

The notations from this section and the previous one can be combined to define $C^\gamma_T C^\alpha_x$, $L^q_T B^\alpha_{p,p}$, etc.; similarly, we define $C^\gamma_T C^\alpha_{loc}$ to be the vector space of all $f:[0,T]\times\RR^d\to\RR^m$ such that $\varphi f\in C^\gamma_T C^\alpha_x$ for all $\varphi\in C^\infty_c$, with convergence $f^n\to f$ in $C^\gamma_T C^\alpha_{loc}$ if $\varphi f^n\to \varphi f$ in $C^\gamma_T C^\alpha_x$.
Given a function $f$ of time and space, $D f$ always denotes its Jacobian in the space variable only.
\subsubsection{Probability measures and Wasserstein distance}\label{not:probability}

Given a separable Banach space $E$, we denote by $\cP(E)$ the set of probability measures over $E$; we write $\mu^n\rightharpoonup \mu$ for weak convergence of measures, in the sense of testing against continuous bounded functions.

Given $\mu,\nu\in \cP(E)$, $\Pi(\mu,\nu)$ stands for the set of all possible couplings of $(\mu,\nu)$, i.e. the subset of $\cP(E\times E)$ with first and second marginals given respectively by $\mu$ and $\nu$. For any $p\in [1,\infty)$, we define
\begin{equation*}
    d_p(\mu,\nu):=\inf_{m\in \Pi(\mu,\nu)} \bigg(\int_{E\times E} \|x-y\|^p_{E}\, m(\dd x,\dd y)\bigg)^{1/p}
\end{equation*}
which is a well defined quantity (possibly taking value $+\infty$). By \cite[Theorem 4.1]{villani2008optimal}, an optimal coupling $\bar{m}\in \Pi(\mu,\nu)$ realizing the above infimum always exists.

Similarly we define $\cP_p(E)$ to be set of $p$-integrable  probability measures; that is, $\mu\in \cP_p(E)$ if $\mu\in \cP(E)$ and
\begin{equation*}
    \| \mu\|_p := \bigg( \int_{E} \|x\|_E^p \, \mu(\dd x)\bigg)^{1/p}<\infty.  
\end{equation*}
It is well known that $d_p(\mu,\nu)<\infty$ for $\mu,\nu\in \cP_p(E)$ and that $(\cP_p(E),d_p)$ is a complete metric space, usually referred to as the $p$-Wasserstein space on $E$; let us stress however that our definition of $d_p(\mu,\nu)$ holds for all $\mu,\nu\in\cP(E)$.
We recall that, given a sequence $\{\mu^n\}_n\subset \cP_p(E)$, $d_p(\mu^n,\mu)\to 0$ is equivalent to $\mu^n\rightharpoonup \mu$ weakly and $\| \mu^n\|_p\to \|\mu\|_p$, see \cite[Theorem 6.9]{villani2008optimal}.

Given $\mu\in \cP(\RR^d)$, with a slight abuse of notation we will write $\mu\in L^q(\RR^d)$ (or simply $L^q_x$) for $q\in [1,\infty]$ to indicate that $\mu$ admits a density $\mu(\dd x)=\rho(x)\dd x$ with respect to the $d$-dim. Lebesgue measure, such that $\rho\in L^q_x$.

\subsubsection{Fractional Brownian motion}

A real valued continuous process $\{W_t,\, t\in [0,T]\}$ is a fractional Brownian motion (fBm) with Hurst parameter $H\in (0,1)$ if it is a centered Gaussian process with covariance function
\begin{align*}
    \EE[W_t W_s] = \frac{1}{2}\big(|t|^{2H}+|s|^{2H}+|t-s|^{2H}\big);
\end{align*}
an $\RR^d$-valued process $W$ is a $d$-dimensional fBm if its components are independent $1$-dimensional fBms.
All the results we are going to recall here are classical and can be found in \cite{nualart2006malliavin, picard2011representation}.

For $H=1/2$, fBm corresponds to classical Brownian motion (Bm), but for $H\neq 1/2$ it is not a semimartingale nor a Markov process; its trajectories are $\PP$-a.s. in $C^{H-\eps}_T$ for any $\eps>0$.

Given an fBm $W$ of parameter $H$ on a probability space $(\Omega,\cF,\PP)$, it's always possible to construct a standard Bm $B$ on it such that the  following \textit{canonical representation} holds:
\begin{align*}
    W_t = \int_0^t K_H(t,s) \dd B_s
\end{align*}
where $K_H$ is a Volterra-type kernel and $B$ and $W$ generate the same filtration.
Given a filtration $\{\cF_t\}_{t\in[0,T]}$, we say that $W$ is an $\cF_t$-fBm if the associated $B$ is an $\cF_t$-Bm in the classical sense.

Closely related to the canonical representation are a version of Girsanov theorem for fBm (see e.g. \cite[Theorem 2]{nualart2002regularization}) and the \textit{strong local non-determinism} (LND) of fBm: for any $H\in (0,1)$ there exists $c_H>0$ such that
\begin{align*}
    Var[W_t\big| \cF_s] \geq c_H |t-s|^{2H} I_d \quad\forall\, t>s.
\end{align*}
The LND property plays a key role in establishing the regularising features of $W$, cf. \cite{harang2020cinfinity, galeati2020prevalence}.

\section{Main results}\label{sec:main-results}
Let us recall that the focus here is an abstract DDSDE of the form
\begin{equation}\label{eq:DDSDE}
    X_t = \xi + \int_0^t B_s(X_s,\mu_s)\mathd s + W_t,\quad \mu_t=\cL(X_t) \quad \forall\, t\in [0,T]
\end{equation}
where $\cL(\xi)=\mu_0$, $\xi$ independent of $W$ and $W$ is sampled as a fBm of parameter $H\in (0,1)$.

We want to identify general conditions for measurable drifts $B:[0,T]\times \cP_p(\RR^d)\rightarrow B^\alpha_{\infty,\infty}$, $\alpha\in \RR$, such that we can develop a solution theory for \eqref{eq:DDSDE}. As explained in the introduction, our strategy consists in setting up a fixed point for $\mu\mapsto b^\mu_t:=B_t(\mu_t)\mapsto X\mapsto \tilde{\mu}_t:=\cL(X_t)$.

To this end, the assumptions on $B$ should enforce two facts: for any flow of measures $\mu\in C_T \cP_p$, the associated drift $b^\mu_t:=B_t(\mu)$ is regular enough to solve \eqref{SDE}, namely $b^\mu$ must satisfy condition \eqref{regularity condition}; the map $\mu\mapsto b^\mu$ should be stable in suitable topologies.
Last but not least, the eligible $B$ should include cases of particular interest (most notably $B(\mu)=b\ast \mu$), see Section \ref{subsec:examples} below.


Corresponding to the above requirements, for $H > 1/2$ we define the following space:

\begin{defn}\label{def: class of functions alpha big}
For $\alpha,\beta\in (0,1)$ and $p\in [1,\infty)$, let $\cH^{\beta,\alpha}_p$ denote the class of continuous functions~$B:[0,T]\times \RR^d\times \cP_p(\RR^d)\rightarrow \RR^d$ satisfying the following condition: there exists $C>0$ such that
\begin{enumerate}[label=\roman*.]
    \item \label{it:alpha-big-1}
    For all $(t,x,\mu)\in [0,T]\times \RR^d\times \cP_p(\RR^d)$, $|B_t(x,\mu)|\leq C$. 
    \item \label{it:alpha-big-2}
    For all $(s,t)\in [0,T]^2$, $(x,y)\in (\RR^d)^2$ and $(\mu,\nu)\in \cP_p(\RR^d)\times \cP_p(\RR^d)$, we have 
        \begin{equation*}
        |B_t(x,\mu)-B_s(y,\nu)| \leq C (|t-s|^{\alpha \beta}+|x-y|^\alpha+d_p(\mu,\nu)^\alpha).
        \end{equation*}
    \item \label{it:alpha-big-3}
    For all $t\in [0,T]$ and $\mu,\nu\in \cP_p(\RR^d)$ 
        \begin{equation*}
        \| B_t(\cdot,\mu)-B_t(\cdot,\nu)\|_{B^{\alpha-1}_{\infty,\infty}} \leq C d_p(\mu,\nu).
        \end{equation*}
\end{enumerate}
Whenever it does not create confusion, we will simply denote by $\| B\|$ the optimal constant $C$.
\end{defn}

Corresponding to the above requirements, for $H \leq 1/2$ we define the following space:

\begin{defn}\label{def: class of functions alpha small}
For $\alpha \in \RR$, $p\in [1,\infty)$ and $q\in [1,\infty]$, let $\cG^{q,\alpha}_{p}$ denote the class of measurable functions $B:[0,T]\times \cP_p(\RR^d)\rightarrow B^\alpha_{\infty,\infty}$ satisfying the following condition: there exists $h\in L^q_T$ such that
\begin{enumerate}[label=\roman*.]
    \item \label{it:alpha-small-1}
    For all $(t,\mu)\in [0,T]\times \cP_p(\RR^d)$, we have $\|B_t(\mu)\|_{B^\alpha_{\infty,\infty}} \leq  h_t$.
    \item \label{it:alpha-small-2}
    For all $(t,\mu,\nu)\in [0,T]\times \cP_p(\RR^d)\times \cP_p(\RR^d)$, we have $\|B_t(\mu)-B_t(\nu)\|_{B^{\alpha-1}_{\infty,\infty}} \leq  h_t d_p(\mu,\nu)$.
\end{enumerate}
Whenever it does not create confusion, we will simply denote by $\| B\|$ the optimal constant $\|h\|_{L^q_T}$.
\end{defn}

\begin{rem}
It is readily checked that for  $\alpha\leq \tilde \alpha$, $p\geq \tilde p$ and $q\leq \tilde q$ we have $\cG^{\tilde q,\tilde \alpha}_{\tilde p}\subset \cG^{q,\alpha}_p$.
Similarly, for $\alpha\leq \tilde \alpha$, $\beta\leq \tilde \beta$ and $p\geq \tilde p$ it holds $\cH^{\tilde \beta,\tilde \alpha}_{\tilde p} \subset \cH^{\beta,\alpha}_p$.
\end{rem}

Roughly speaking, we say that $X$ is a solution to the DDSDE \eqref{eq:DDSDE} if, setting $b^\mu_t:=B_t(\cL(X_t))$, then $X$ is a solution to the standard SDE \eqref{SDE} associated to $b^\mu$, being interpreted in the Catellier--Gubinelli sense whenever $b^\mu$ is singular;
the pathwise theory for singular SDEs will be recalled in detail in Section \ref{sec:recap-SDE}.
All the concepts of strong existence, pathwise uniqueness and uniqueness in law for DDSDEs then follow from the standard ones, see Definition \ref{def:solution-singular-DDSDE} from Section \ref{subsec:H<1/2}.\\

Our first main result is the well-posedness of DDSDE \eqref{eq:DDSDE} under suitable conditions on $B$; it can be seen as an extension of \cite[Theorem 15]{galeati2020noiseless} to the distribution dependent case.

\begin{thm}\label{thm:main_thm1}
Let $H>1/2$ and let $B\in\cH^{H,\alpha}_p$ for parameters
\begin{equation}\label{eq:parameters-H>1/2}
    \alpha>1-\frac{1}{2H}>0,\quad p\in [1,\infty).
\end{equation}
Then for any $\mu_0\in \cP_p(\RR^d)$, strong existence, pathwise uniqueness and uniqueness in law hold for the DDSDE \eqref{eq:DDSDE}.

Similarly, let $H\leq 1/2$ and let $B\in \cG^{q,\alpha}_p$ for parameters
\begin{equation}\label{eq:parameters-H<1/2}
    \alpha>1+\frac{1}{Hq}-\frac{1}{2H},\quad \alpha\in\RR, \quad q\in (2,\infty], \quad p\in [1,\infty).
\end{equation}
Then for any $\mu_0\in \cP_p(\RR^d)$, strong existence, pathwise uniqueness and uniqueness in law hold for the DDSDE \eqref{eq:DDSDE}.
\end{thm}

Given a DDSDE \eqref{eq:DDSDE}, we will consider either $(\xi,B)$ or $(\mu_0,B)$ to be the data of the problem, where we recall that $\cL(\xi)=\mu_0$.
As already mentioned in the introduction, the solution $X$ is entirely determined by the associated flow of measures $\mu\in C_T\cP_p$ given by $\mu_t=\cL(X_t)$: once this is known, the drift $b^\mu_t=B_t(\mu_t)$ is determined as well and so we can reconstruct the strong solution $X$ (or construct another copy of it on any probability space of interest).
For this reason, it is quite useful to regard $\mu\in C_T\cP_p$ to be itself a solution to the DDSDE; the exact equivalence between $\mu$ and $X$ will be discussed rigorously in Lemma \ref{lem:SolDefEquiv} from Section \ref{subsec:H<1/2}.

The next theorem provides stability estimates for the data-to-solution map $(\mu_0,B) \mapsto \mu$ (respectively $(\xi,B)\mapsto X$), showing that it is locally Lipschitz.

%

\begin{thm}\label{thm:main_thm2}
Let  $\mu_0,\nu_0\in \cP_p$ for some $p\in [1,\infty)$. Then the following holds:
\begin{enumerate}[label=\roman*.]
    \item \label{it:main_thm2-1}
    For $H>1/2$, let $B^1,\,B^2$, be drifts in $\cH^{H,\alpha}_p$ with parameters satisfying \eqref{eq:parameters-H>1/2} and let $M>0$ be a constant such that $\| B^i\|\leq M$.
    Then there exists a constant $C=C(\alpha,H,T,M,p)$ such that, for any $\mu_0^i\in \cP_p(\RR^d)$, the associated solutions $\mu^i\in C_T \cP_p$ satisfy
    \begin{equation}\label{eq:stability-estim-DDSDE}
    \sup_{t\in [0,T]} d_p(\mu^1_t,\mu^2_t) \leq C \big(d_p(\mu^1_0,\mu^2_0) + \| B^1-B^2\|_{\infty}\big),
    \end{equation}
    where 
    \[
    \| B^1-B^2\|_{\infty}:=\sup_{(t,\mu)\in [0,T]\times \cP_p} \| B^1(t,\mu)-B^2(t,\mu)\|_{B^{\alpha-1}_{\infty,\infty}}.
    \]
    If $X^1,X^2$ are two associated solutions, in the sense of stochastic processes, defined on the same probability space, then there exists $\gamma>1/2$ such that
    \begin{equation}\label{eq:stability-estim-DDSDE-2}
    \EE\Big[\|X^1-X^2 \|_{\gamma;[0,T]}^p\Big]^{1/p} \leq C\big( \| \xi^1-\xi^2\|_{L^p_\Omega} + \| B^1-B^2\|_\infty\big).
    \end{equation}
    \item \label{it:main_thm2-2} 
    For $H\leq 1/2$, let $B^1,\,B^2$, be drifts in $\cG^{q,\alpha}_p$ with parameters satisfying \eqref{eq:parameters-H<1/2} and let $M>0$ be a constant such that $\| B^i\|\leq M$.
    Then there exists a constant $C=C(\alpha,H,T,M,p,q)$  such that, for any $\mu_0^i\in \cP_p(\RR^d)$, the associated solutions $\mu^i\in C_T \cP_p$ satisfy
    \begin{equation}\label{eq:stability small}
    \sup_{t\in [0,T]} d_p(\mu^1_t,\mu^2_t) \leq C \big(d_p(\mu^1_0,\mu^2_0) + \| B^1-B^2\|_{q,\infty}\big).
    \end{equation}
    where
    \begin{equation*}
    \| B^1-B^2\|_{q,\infty}:=\bigg( \int_0^T \sup_{\mu\in \cP_p} \| B^1(t,\mu)-B^2(t,\mu)\|_{B^{\alpha-1}_{\infty,\infty}}^{q} \dd t\bigg)^{1/q}.
    \end{equation*}
    If $X^1,X^2$ are two associated solutions, in the sense of stochastic processes, defined on the same probability space, then there exists $\gamma>1/2$ such that
    \begin{equation}\label{eq:stability-estim-DDSDE-3}
    \EE\Big[\|X^1-X^2\|_{\gamma;[0,T]}^p\Big]^{1/p} \leq C\big( \| \xi^1-\xi^2\|_{L^p_\Omega} + \| B^1-B^2\|_{q,\infty}\big).
    \end{equation}
\end{enumerate}
\end{thm}

As the settings of Theorems \ref{thm:main_thm1} and \ref{thm:main_thm2} are very general, they do not allow one to exploit any specific structure of the DDSDE in consideration to obtain sharper results. A prototypical example of such structure, which arises in many practical applications, is given by convolutional drifts $B_t(x,\mu):= (b_t\ast \mu)(x)$.
The associated DDSDE takes the form
\begin{equation}\label{eq:ConvDDSDE}
    X_t = \xi + \int_0^t (b_s \ast \cL(X_s))(X_s)\dd s + W_t \quad \forall\, t\in [0,T].
\end{equation}
As before we allow the drift $b$ to be distributional, at least of the form $b \in L^1_T B^\alpha_{p,p}$ for some $\alpha \in \RR$, $p\in [1,\infty]$; at this stage pointwise evaluation of $b_s\ast\cL(X_s)$ is not meaningful, instead we again interpret the equation in the Catellier--Gubinelli sense.

The heuristic idea behind the next results is that we can use the convolutional structure in a recursive way: assuming we are given a solution $X$ with sufficiently regular law $\cL(X_\cdot)$, this in turn leads to an improved regularity for the effective drift $b_\cdot\ast \cL(X_\cdot)$, compared to the original $b$.
The argument can be made rigorous by establishing a priori estimates and working with smooth approximations; as a result, we are able to establish well-posedness for \eqref{eq:ConvDDSDE} in situations where the general Theorem \ref{thm:main_thm1} does not apply.

In both results we are going to present, we will need some additional regularity for the initial data $\mu_0$, in the form of an integrability assumption.
This is because, as explained in the introduction, the lack of an underlying parabolic PDE prevents us from proving a smoothing effect at strictly positive times analogous to that of parabolic equations; rather, in order to develop a priori estimates, we will show that such integrability is propagated by the dynamics.

The next result shows existence and uniqueness of solutions to \eqref{eq:ConvDDSDE} in a suitable class, under an additional condition on $\div b$, which is by now quite standard since the pioneering work \cite{diperna1989ordinary}.

\begin{thm}\label{thm:main_thm3}
Let $H\in (0,1)$, $q \in (2,\infty]$, $p\in [1,\infty]$, $p'$ its conjugate exponent.
Assume that $\div b \in L^1_T L^\infty_x$ and either:
\begin{enumerate}[label=\roman*.]
    \item if $H>1/2$, then $b\in C^{\alpha H}_T L^p_x\cap  C^0_T B^\alpha_{p,p}$ for some $\alpha>1-\frac{1}{2H}$;
    \item if $H \leq 1/2$, then $b \in L^q_T B^\alpha_{p,p}$ with $1>\alpha>1-\frac{1}{2H}+\frac{1}{Hq}$.
\end{enumerate}
Then for any $\mu_0\in L^{p'}_x$ there exists a strong solution to \eqref{eq:ConvDDSDE}, which satisfies
\begin{equation}\label{eq:main_thm3_integrability}
    \sup_{t\in [0,T]} \| \cL(X_t)\|_{L^{p'}_x}<\infty;
\end{equation}
moreover uniqueness holds, both pathwise and in law, in the class of solutions satisfying \eqref{eq:main_thm3_integrability}.
\end{thm}

Our second result in the convolutional case is established under $L^q_T L^p_x$-type assumptions on $b$; here instead of relying on a bound for $\div b$, we exploit Girsanov-based arguments to establish integrability of $\cL(X_t)$.
This technique however only works in the regime $H\leq 1/2$.

\begin{thm}\label{thm:main_thm4}
Let $d\geq 2$, $H \leq 1/2$, $(r,p,q) \in [1,\infty)^2 \times (2,\infty]$ be such that
\begin{equation}\label{eq:main_thm4_parameters}
r>\frac{d}{d-1}, \quad \frac{1}{q}+\frac{Hd}{p}<\frac{1}{2}.
\end{equation}
Then for any $b \in L^q_T L^{p}_x$ and $\mu_0 \in L^r_x$, there exists a strong solution to \eqref{eq:DDSDE}, which satisfies
\begin{equation}\label{eq:main_thm4_integrability}
    \sup_{t\in [0,T]} \| \cL(X_t)\|_{L^{\tilde r}_x}<\infty \quad \forall\, \tilde r<r;
\end{equation}
moreover uniqueness holds, both pathwise and in law, in the class of solutions satisfying \eqref{eq:main_thm4_integrability}.
\end{thm}

\begin{rem}
Condition \eqref{eq:main_thm4_parameters} can be generalized in a way that allows values $r\leq d/(d-1)$ and that applies for $d=1$, see Theorem \ref{thm:main_sec5} in Section \ref{subsec:refined-integrable} for more details.
We warn the reader not to interpret Theorems \ref{thm:main_thm3} and \ref{thm:main_thm4} as full pathwise uniqueness (resp. uniqueness in law) statements: in general they do not exclude the existence of irregular solutions $X$ which do not satisfy condition \eqref{eq:main_thm3_integrability} (resp. \eqref{eq:main_thm4_integrability}).
However, as the proofs show, any solution constructed as the limit of smooth drifts $b^n\to b$ does satisfy \eqref{eq:main_thm3_integrability} (resp. \eqref{eq:main_thm4_integrability}), thus it is the only physical solution to the DDSDE \eqref{eq:DDSDE}.
\end{rem}

\subsection{Examples}\label{subsec:examples}
To illustrate the variety of situations to which Theorems \ref{thm:main_thm1} and \ref{thm:main_thm2} apply, we provide here several examples of functions contained in $\cG^{q,\alpha}_p$ and $\cH^{\beta,\alpha}_p$.

\begin{ex}\label{ex:true-mckeanvlasov-setting-1}
Let $\alpha\in \RR$, and for any $y\in \RR^d$, $b:[0,T]\times \RR^d \rightarrow B^\alpha_{\infty,\infty}$  be a measurable map, $b_t(\cdot,y):=b_t(y)(\cdot)$, and suppose there exists $h\in L^q_T$ for some $q\in [1,\infty]$ such that
\begin{equation*}
    \|b_t(\cdot,y)\|_{B^\alpha_{\infty,\infty}}
    \leq h_t,\quad \|b_t(\cdot,y)-b_t(\cdot,y')\|_{B^{\alpha-1}_{\infty,\infty}} \leq h_t\, |y-y'|\quad
    \forall\, t\in [0,T],\, (y,y')\in \RR^{2d}.
\end{equation*}
Define now another measurable map $B:[0,T]\times \cP(\RR^d)\to B^\alpha_{\infty,\infty}$ by
\begin{align*}
    B_t(\cdot,\mu):= \int_{\RR^d}b_t(\cdot,y)\,\mu (\dd y), \quad \forall\,(t,\mu)\in [0,T]\times \cP_p(\RR^d)
\end{align*}
where the integral is meaningful in the Bochner sense; then $B\in \cG^{q,\alpha}_p$ for any $p\in [1,\infty)$.

Indeed, by the hypothesis on $b$, it is readily checked that
\begin{equation*}
    \|B_t(\cdot,\mu)\|_\alpha \leq \int_{\RR^d}\|b_t(\cdot,y)\|_\alpha\,\mu(\dd y)
    \leq h_t \quad \forall\, t\in [0,T], \mu\in \cP(\RR^d);
\end{equation*}
given $\mu,\nu\in \cP(\RR^d)$, let $m\in \cP(\RR^{2d})$ be an optimal coupling for $d_1(\mu,\nu)$, then
\begin{equation*}
     \|B_t(\cdot,\mu)-B_t(\cdot,\nu)\|_{B^{\alpha-1}_{\infty,\infty}}
     \leq  \int_{\RR^{2d}}\|b_t(\cdot,y)-b_t(\cdot,y')\|_{B^{\alpha-1}_{\infty,\infty}}\, m(\mathd y,\mathd y')
     \leq h_t \int_{\RR^{2d}} |y-y'|\, m(\mathd y,\mathd y')
\end{equation*}
which implies that
\begin{equation*}
     \|B_t(\cdot,\mu)-B_t(\cdot,\nu)\|_{B^{\alpha-1}_{\infty,\infty}}
     \leq h_t\, d_1(\mu,\nu)
     \leq h_t\, d_p(\mu,\nu) \quad \forall\, p\in [1,\infty).
\end{equation*}
\end{ex}

\begin{ex}\label{ex:true-mckeanvlasov-setting-2}
Given $\alpha, \beta\in (0,1)$, assume that $b:[0,T]\times\RR^d\times \RR^d\to \RR^d$ satisfies
\begin{align*}
    |b_t(x,y)| \leq C,\quad |b_t(x,y)-b_s(x',y')| \leq C(|t-s|^{\alpha\beta} + |x-x'|^\alpha + |y-y'|^\alpha)
\end{align*}
for some $C>0$, uniformly over $s,\,t,\, x,\,x',\,y,\,y'$;
we can identify $b$ with the map $b:[0,T]\times \RR^d\to C^\alpha_x=B^{\alpha}_{\infty,\infty}$ given by $(t,y)\mapsto b_t(\cdot,y)$.
Assume additionally that for the same constant $C$ it holds
\begin{align*}
    \| b_t(\cdot,y)-b_t(\cdot,y')\|_{B^{\alpha-1}_{\infty,\infty}} \leq C\, |y-y'|
\end{align*}
and define $B:[0,T]\times \RR^d\times\cP(\RR^d)\to \RR^d$ by
\begin{align*}
    B_t(x,\mu):=\int_{\RR^d} b_t(x,y)\,\mu(\dd y).
\end{align*}
Then $B\in \cH^{\beta,\alpha}_p$ for any $p\in [1,\infty)$.
The verification of Conditions \ref{it:alpha-big-1} and \ref{it:alpha-big-3} of Definition \ref{def: class of functions alpha big} is identical to that of Example \ref{ex:true-mckeanvlasov-setting-1}, so we only need to focus on Condition \ref{it:alpha-big-2} for $p=1$.

Given $\mu,\,\nu\in\cP(\RR^d)$, let $m$ be an optimal coupling for $d_1(\mu,\nu)$, then
\begin{align*}
    |B_t(x,\mu)-B_s(x',\nu)|
    & = \bigg|\int_{\RR^d} b_t(x,y)\mu(\dd y)-\int_{\RR^d} b_s(x',y')\nu(\dd y')\bigg|\\
    & \leq \int_{\RR^{2d}} |b_t(x,y)-b_s(x',y')|\, m(\dd y,\dd y')\\
    & \leq C \bigg(|t-s|^{\alpha\beta} + |x-x'|^\alpha + \int_{\RR^{2d}} |y-y'|^\alpha\, m(\dd y,\dd y')\bigg)\\
    & \leq C \big( |t-s|^{\alpha\beta} + |x-x'|^\alpha + d_1(\mu,\nu)^\alpha\big)
\end{align*}
where in the last step we used Jensen's inequality and the optimality of $m$.
\end{ex}

\begin{ex}\label{ex:convolutional-setting}
Consider now $B_t(\cdot,\mu):=b_t\ast\mu$, where $b\in L_T^q B^\alpha_{\infty,\infty}$ for some $\alpha\in \RR$; then $B\in \cG^{q,\alpha}_p$ for any $p\in [1,\infty)$.

Indeed, the verification of Condition \ref{it:alpha-small-1} from Definition \ref{def: class of functions alpha small} is the same as in Example \ref{ex:true-mckeanvlasov-setting-1}, where now we can take $h_\cdot = \|b_\cdot\|_{B^\alpha_{\infty,\infty}}\in L^q_T$.
Moreover by Lemma \ref{lem:basic-Besov-convolve} in Appendix \ref{app:UsefulLemmas}, for any $\mu,\nu\in \cP(\RR^d)$ and any $p\in [1,\infty)$ it holds
\begin{align*}
    \| B_t(\cdot,\mu)-B_t(\cdot,\nu)\|_{B^{\alpha-1}_{\infty,\infty}}
    = \| b_t \ast (\mu-\nu)\|_{B^{\alpha-1}_{\infty,\infty}}
    \lesssim \| b_t\|_{B^\alpha_{\infty,\infty}}\, d_p(\mu,\nu).
\end{align*}

Similarly, given $\alpha,\beta \in (0,1)$, let $b\in C^{\alpha\beta}_T C^0_x\cap C_T C^\alpha_x$ and set $B_t(\cdot,\mu)=b_t\ast\mu$; then $B\in \cH^{\beta,\alpha}_p$ for any $p\in [1,\infty)$.

The verification of Conditions \ref{it:alpha-big-1} and \ref{it:alpha-big-2} from Definition \ref{def: class of functions alpha big} follows from Example \ref{ex:true-mckeanvlasov-setting-2}, as we can simply set $\tilde b_t(x,y):=b_t(x-y)$ and apply the calculations therein to $\tilde b$.
Condition \ref{it:alpha-big-1} instead follows as above from an application of Lemma \ref{lem:basic-Besov-convolve}.

Finally, let us point out that all the computations carry over to the case $B_t(\cdot,\mu)=b^1_t\ast\mu+b^2_t$ for $b^i\in L^q_T B^\alpha_{\infty,\infty}$ (resp. $b^i\in C^{\alpha\beta}_T C^0_x\cap C_T C^\alpha_x$).
\end{ex}

\begin{ex}\label{ex:functions-statistics}
Let $b:[0,T]\times \RR^d\to B^\alpha_{\infty,\infty}$ be as in Example \ref{ex:true-mckeanvlasov-setting-1} and $\phi:\RR^d\to \RR^d$ be a globally Lipschitz with constant $\llbracket\phi\rrbracket_{Lip}$; define $B:[0,T]\times \cP_1(\RR^d)\to B^\alpha_{\infty,\infty}$ by
\begin{align*}
    B_t(\cdot,\mu)= b_t(\cdot, \langle \phi,\mu\rangle), \quad \text{where } \langle \phi,\mu \rangle:=\int_{\RR^d} \phi(x)\, \mu(\dd x).
\end{align*}
Then $B\in \cG^{q,\alpha}_p$ for any $p\in [1,\infty)$. Similarly, given $b$ as in Example \ref{ex:true-mckeanvlasov-setting-2}, with $B$ defined as above, it is easy to verify that $B\in \cH^{\alpha,\beta}_p$ for any $p\in [1,\infty)$.

As a prototypical example, one may consider $b\in B^\alpha_{\infty,\infty}$ and define
\begin{align*}
    B_t(\cdot,\mu)=B(\cdot,\mu):=b(\cdot -\langle x,\mu\rangle) \quad\text{where }
    \langle x,\mu\rangle:=\int_{\RR^d} x\,\mu(\dd x)
\end{align*}
in which case, similarly to before, it holds $B\in \cG^{q,\alpha}_p$ for any $q\in [1,\infty]$ (resp. $B\in \cH^{\beta,\alpha}_p$ for any $\beta\in (0,1)$) and $p\in [1,\infty)$.

We highlight that this class of examples are quite important since $B$ is only defined on $\cP_1(\RR^d)$ and not on the whole $\cP(\RR^d)$, thus making the use of other notions of distance between measures (e.g. total variation norm) more difficult to handle.
It can be further generalized to the case $\phi:\RR^d\to \RR^m$ for another $m\in \NN$ (namely, $B$ is determined by $m$ statistics associated to $\mu$) or to dependence on $p$-moments like $B_t(\cdot,\mu)=b_t(\cdot,\| \mu\|_p)$ for $\mu\in \cP_p(\RR^d)$; for $p>1$ we can also allow $\phi$ to grow more than linearly at infinity.
\end{ex}

\section{SDEs driven by fBm}\label{sec:recap-SDE}

In this section we revisit the theory of singular SDEs driven by fBm, in order to derive useful estimates to apply later to the DDSDE setting. Sections \ref{subsec:pathwise-SDEs} and \ref{subsec:sde-averaged-girsanov} serve as a recap of key facts, respectively the pathwise meaning of singular SDEs and the regularising properties of fractional Brownian motion. Sections \ref{subsec:stability} and \ref{subsec:reg-estim} instead provide novel results, Theorem \ref{thm:fBmSDEWellPosed} being the most important for our purposes.

Although the material of Sections \ref{subsec:pathwise-SDEs}-\ref{subsec:sde-averaged-girsanov} is strongly based on the works previous \cite{Catellier2016,galeati2020noiseless,galeati2020nonlinear,harang2020cinfinity}, we felt obliged to provide the proofs of several key results for technical but rather important reasons.
On the one hand, the aforementioned works are focused entirely on a pathwise setting, never establishing clear probabilistic concepts of solutions (cf. Definitions \ref{def:SDE-as-NLYE}-\ref{def:pathwise-strong-weak-sol} below);
on the other hand, previously singular drifts $b\in L^q_T B^\alpha_{\infty,\infty}$ were treated in \cite{Catellier2016} only in the autonomous case, while in \cite{galeati2020noiseless} when they are compactly supported in space. As neither option fits our setting nicely (consider drifts of the form $b=\tilde{b}\ast \mu_t$) we extend the results therein to suit our analysis of DDSDEs.

\subsection{Pathwise SDEs as nonlinear Young equations}\label{subsec:pathwise-SDEs}

Consider a standard SDE of the form
\begin{equation}\label{eq:simple sde}
    X_t=X_0+\int_0^t b(s,X_s)\dd s + W_t,\quad \forall\, t\in [0,T],
\end{equation}
where $b\in L^1_T B^\alpha_{\infty,\infty}$ with $\alpha \in \RR$ and $W$ is an $\RR^d$-valued fractional Brownian motion.

When $\alpha>0$, the SDE has a classical meaning; it can be solved pathwise by standard ODE theory if $b$ is regular enough, e.g. $\alpha>1$.
We will say that $b$ is a \textit{distributional drift} (sometimes \textit{distributional field}) if instead $\alpha<0$, in which case pointwise evaluation is not allowed, and we cannot give meaning to the integral appearing in \eqref{eq:simple theta} in the classical Lebesgue sense.

To deal with distributional drifts, we will employ the \textit{nonlinear Young integral} framework, first developed in \cite{Catellier2016}; to present it, we first need the concept of \textit{averaged field}.

Let us give an heuristic motivation before going into technical details. In the regular regime $\alpha>0$, if $X$ is a solution to \eqref{eq:simple sde}, by the change of variables $\theta_t:=X_t-W_t$ we find that $\theta$ solves
\begin{equation}\label{eq:simple theta}
    \theta_t=\theta_0+\int_0^t b(s,\theta_s+W_s)\dd s. 
\end{equation}
Closely related to the above integral is the averaging of the field $b$ along the curve $W$, namely the space-time function
\begin{equation}\label{eq:avg field}
    T^W b(t,x):=\int_0^t b(s,x+W_s)\dd s
\end{equation}
which we call an \textit{averaged field}; we will write $T^W_{s,t}b(x):=T^W b(t,x)-T^W b(s,x)$.

As long as $b$ is at least measurable and bounded, both integrals appearing in \eqref{eq:simple theta} and \eqref{eq:avg field} are well defined.
However, for distributional $b$, while equation \eqref{eq:simple theta} breaks down, the averaged field $T^W b$ is still meaningful in the distributional sense, see \cite[Section 3.1]{galeati2020noiseless}; moreover, depending on the properties of $W$, $T^W b$ might even be continuous or (higher order) differentiable in the spatial variable.

The fundamental intuition of \cite{Catellier2016} is that the regularity of $T^W b$ can be used to give meaning to \eqref{eq:simple theta}, thus also to \eqref{eq:simple sde}, by reformulating the SDE as a {\em nonlinear Young equation}.

As the next statement shows, given any space-time function $A:[0,T]\times \RR^d\to \RR^d$ and path $\theta:[0,T]\to\RR^d$ of suitable regularity, it's possible to give meaning to $\int_0^t \partial_t A(s,\theta_s)\dd s$ also when $\partial_t A$ is not well defined anymore.

\begin{prop}\label{prop:NLY integral}
Let $\gamma>1/2$ and consider a function $A\in C^\gamma_T C^1_{loc}$ and a path $\theta\in C^\gamma_T$.
Then for any  interval $[s,t]\subset [0,T]$ and any sequence of partitions $\cD_n$ of $[s,t]$ with mesh converging to zero, the following limit exists and is independent of the chosen sequence:
\begin{equation*}
   \int_s^t A(\dd s,\theta_s):= \lim_{n\to\infty} \sum_{[u,v]\in \cD_n} A_{u,v}(\theta_u).
\end{equation*}
We will refer to it as a nonlinear Young integral. Furthermore:
\begin{enumerate}[label=\roman*.]
    \item \label{it:NLY-integral-1}
    The integral is additive: $\int_s^tA(\dd u,\theta_u)=\int_s^rA(\dd u,\theta_u)+\int_r^t A(\dd u,\theta_u)$ for any $r\in [s,t]$. 
    \item \label{it:NLY-integral-2}
    If $\partial_t A$ exists and is continuous, then $\int_0^t A(\dd u,\theta_u)=\int_0^t \partial_u A(u,\theta_u)\dd u$.
    \item \label{it:NLY-integral-3}
    The map from $C^\gamma_T C^1_{loc}\times C^\gamma_T$ to $C^\gamma_T$ given by $(A,\theta)\mapsto \int_0^\cdot A(\dd u,\theta_u)$ is linear in $A$ and continuous in both variables. Namely, if $A^n\to A$ in $C^\gamma_T C^1_{loc}$ and $\theta^n\to \theta$ in $C^\gamma_T$, then $\int_0^\cdot A^n(\dd s,\theta^n_s)\to \int_0^\cdot A(\dd s,\theta_s)$ in $C^\gamma_T$.
    %
\end{enumerate}
\end{prop}

\begin{proof}
The statement is a particular subcase of \cite[Theorem 2.7]{galeati2020nonlinear}.
\end{proof}

We provided the statement only  for $A\in C^\gamma_T C^1_{loc}$ as this setting is sufficient for our purposes, but let us mention that the theory is more general and allows to consider $A\in C^\gamma_T C^\nu_{loc}$, $\theta\in C^\rho_T$ for $\gamma+\nu\rho>1$.
With the above result at hand, we can now define nonlinear Young equations.

\begin{defn}
Let $A\in C^\gamma_T C^1_{loc}$ with $\gamma>1/2 $, $\theta_0\in \RR^d$; we say that $\theta$ is a solution to the nonlinear Young equation associated to $(\theta_0,A)$ if $\theta\in C^\gamma_T$ and
\begin{equation}\label{eq:abs NLYE}
    \theta_t=\theta_0+\int_0^t A(\dd s,\theta_s)\quad \forall\, t\in [0,T]. 
\end{equation}
\end{defn}
For later use, we provide the following technical lemma; loosely speaking it shows that solutions to nonlinear Young equations have a closure property. 
\begin{lem}\label{lem:closure-YDE}
Let $\gamma>1/2$, $A\in C^\gamma_T C^1_{loc}$ and $\{A^n\}_{n\in \NN}$ be a sequence converging to $A$ in $C^\gamma_T C^1_{loc}$; suppose that for each $n$ there exists a solution $\theta^n$ associated to $(\theta_0,A^n)$ and that $\theta^n\to \theta$ in $C^\gamma_T$.
Then $\theta$ solves the nonlinear Young equation associated to $(\theta_0,A)$.  
\end{lem}
\begin{proof}
This is a direct consequence of Point \ref{it:NLY-integral-3} of Proposition \ref{prop:NLY integral}.
By assumption
\begin{align*}
   \theta^n_t = \theta_0 + \int_0^t A^n(\dd s,\theta^n_s) \quad \forall\, t\in [0,T],\,n\in\NN
\end{align*}
and we can pass to the limit on both sides thanks to the continuity of $(\theta,A)\mapsto \int_0^\cdot A(\dd s,\theta_s)$.
\end{proof}

We are now ready to explain what it means for $X$ to be a solution to \eqref{eq:simple sde} when $b$ is distributional but $T^W b$ is regular enough: roughly speaking, we impose the condition $X=\theta + W$, where $\theta$ solves the nonlinear YDE associated to $A=T^W b$, which is the natural extension of \eqref{eq:simple theta}.
Although so far we have always dealt with a stochastic process $W$, this is a \textit{pathwise notion} of solution, in the sense that for any fixed realization of $W(\omega)$ such that $T^{W(\omega)} b\in C^\gamma_T C^1_{loc}$ we have an analytically well-defined equation of the form \eqref{eq:abs NLYE}. This is encoded in the next definition, inspired by \cite[Section 4.3]{galeati2020regularization}, which contains a more in-depth discussion of various related concepts.

\begin{defn}\label{def:SDE-as-NLYE}
Let $(\Omega,\cF,\PP)$ be a probability space, $(\xi,W)$ an $\RR^d\times C_T$-valued random variable defined on it and let $b$ be a distributional field.
We say that another $C_T$-valued random variable $X$ on $(\Omega,\cF,\PP)$ is a \textit{pathwise solution} to the SDE \eqref{eq:simple sde} associated to $(b,\xi, W)$ if there exists $\Omega'\subset \Omega$ with $\PP(\Omega')=1$ and a deterministic $\gamma>1/2$ such that for all $\omega\in \Omega'$ the following hold:
\begin{enumerate}[label=\roman*.]
    \item $T^{W(\omega)} b \in C^\gamma_T C^1_{loc}$;
    \item $\theta(\omega):= X(\omega)-W(\omega)\in C^\gamma_T$;
    \item $\theta(\omega)$ satisfies the nonlinear Young equation
    \begin{align*}
        \theta_t(\omega) = \xi(\omega) - W_0(\omega) + \int_0^t T^{W(\omega)} b(\dd s, \theta_s(\omega)) \quad \forall\, t\in [0,T].
    \end{align*}
\end{enumerate}
\end{defn}

The following definition relates standard probabilistic notions of weak and strong solutions and of uniqueness to the notion of pathwise existence given in Definition \ref{def:SDE-as-NLYE}.

\begin{defn}\label{def:pathwise-strong-weak-sol}
Let $b$ be a distributional field, $\nu\in \cP(\RR^d\times C_T)$. A tuple $(\Omega,\cF,\PP; X,\xi,W)$ given by a probability space $(\Omega,\cF,\PP)$ and a $C_T\times \RR^d\times C_T$-valued random variable is a \textit{weak solution} to the SDE \eqref{eq:simple sde} associated to $(b,\nu)$ if $\cL_\PP(\xi,W)=\nu$ and $X$ is a pathwise solution associated to $(b,\xi,W)$ in the sense of Definition \ref{def:SDE-as-NLYE}.
We say that $X$ is a \textit{strong solution} if it is adapted to the filtration $\cF_t=\sigma\{\xi, W_s\,|\, s\leq t\}$.
\textit{Weak uniqueness} holds for the SDE associated to $(b,\nu)$ if any given weak solutions $(\Omega^i ,\cF^i,\PP^i;X^i,\xi^i,W^i)$, $i=1,2$, associated to the same data $(b,\nu)$, satisfy $\cL_{\PP^1}(X^1)=\cL_{\PP^2}(X^2)$.
Similarly, \textit{pathwise uniqueness} holds if any two given solutions $(X^i,\xi,W)$ defined on the same probability space, w.r.t. the same $(b,\xi,W)$, satisfy $X^1=X^2$ $\PP$-a.s.
\end{defn}

In line with the above definition, we will use the standard terminology that weak (resp. strong) existence holds for the SDE associated to $(b,\nu)$ to mean that we can construct a weak (resp. strong) solution $(\Omega,\cF,\PP;X,\xi,W)$. In particular, if strong existence holds, then $(\Omega,\cF,\PP)$ can be chosen to be the canonical space, namely with $\Omega=\RR^d\times C_T$, $\PP=\nu$ and $\cF$ the completion of $\cB(\RR^d\times C_T)$ under $\nu$.

\begin{rem}\label{rem:SDE-classical-case}
If $b\in C_T C^1_{loc}$, then any classical solution to \eqref{eq:simple sde} is of the form $X=W+\theta$ for $\theta\in C^1_T$; moreover in this case $T^W b\in C^1_T C^1_{loc}$ and $\partial_t T^W b(t,x) = b(t,x+W_t)$.
It then follows from Point \ref{it:NLY-integral-2} of Proposition \ref{prop:NLY integral} that in this setting the concept of pathwise solution from Definition \ref{def:SDE-as-NLYE} is equivalent to the standard one.
Moreover for $b\in C_T C^1_{loc}$ standard ODE theory guarantees pathwise uniqueness, uniqueness in law and strong existence of solutions for the SDE associated to $(b,\nu)$ for any choice of $\nu\in \cP(\RR^d\times C_T)$.
\end{rem}

The next lemma provides a simple condition to establish uniqueness of solutions to \eqref{eq:simple sde}.

\begin{lem}\label{lem:criteria-uniqueness-YDE}
Let $(\Omega,\cF,\PP)$ be a probability space, $(X,\xi,W)$ be a triple defined on it such that $X$ solves the SDE associated to $(b,\xi,W)$ in the sense of Definition \ref{def:SDE-as-NLYE}.
If $T^{X(\omega)} b\in C^\gamma_T C^1_{loc}$ for $\PP$-a.e. $\omega$, then any other solution $\tilde X$ defined on the same probability space and associated to $(b,\xi,W)$ must coincide with it, in the sense that $X=\tilde X$ $\PP$-a.s.
\end{lem}

\begin{proof}
The statement is a useful rewriting of \cite[Remark 15]{galeati2020noiseless}. 
\end{proof}

Let us stress that, even when the assumptions of Lemma \ref{lem:criteria-uniqueness-YDE} are met, pathwise uniqueness doesn't immediately follow, unless one can additionally show that  $X$ is a strong solution.

\subsection{Regularity of averaged fields and Girsanov transform for fBm}\label{subsec:sde-averaged-girsanov}

In Section \ref{subsec:pathwise-SDEs} we have treated the SDE \eqref{eq:simple sde} in full generality, but in the remainder of Section \ref{sec:recap-SDE} we will deal with a slightly more specific setting.
We will always take $W$ to be an $\RR^d$-valued fBm of parameter $H\in (0,1)$ and $\xi$ to be random initial data independent of it; in particular $W_0\equiv 0$ and $\nu = \cL(\xi,W)=\cL(\xi)\otimes \cL(W)=\mu_0\otimes \mu^H$ for some $\mu_0\in \cP(\RR^d)$, where $\mu^H\in \cP(C_T)$ denotes the law of fBm of parameter $H\in (0,1)$.
Therefore for fixed $H$ we can regard the data of the problem to be the pair $(\mu_0,b)$; if the initial data $\xi=x_0\in \RR^d$ is deterministic, with a slight abuse we will write $(x_0,b)$ in place of $(\delta_{x_0},b)$.

We begin by showing the $\PP$-a.s. regularity of averaged fields $T^W b$ for $W$ sampled as an fBm. We continue to make use of the intuitive notation
\[
\int_0^t b(r,x+W_r)\dd r = T^W b,
\]
despite the fact that in general these objects will not be defined as Lebesgue integrals; rather they are random variables defined on $(\Omega,\cF,\PP)$ constructed as the unique limits of $\int_0^t b^n(r,x+W_r)\dd r$ for any sequence $b^n\to b$ in appropriate topologies.

\begin{prop}\label{prop:main_avg_estimates-1}
Let $b\in L^q_T B^\alpha_{\infty,\infty}$ with $\alpha<0$, $q\in (2,\infty]$, $W$ be a fBm of parameter $H\in (0,1)$; suppose $(\alpha,q)$ satisfy
\begin{equation}\label{eq:simple-cond-1} \gamma:=1-\frac{1}{q}+\alpha H>\frac{1}{2}.
\end{equation}
Then for any $\tilde\gamma<\gamma$ there exists an increasing function $K$ (depending on $d,T$ and the above parameters) such that
\begin{equation}\label{eq:averaging-stoch-estimate-1}
\EE \Bigg[\exp \bigg( \frac{\eta}{\| b\|_{L^q_T B^\alpha_{\infty,\infty}}^2} \Big\llbracket \int_0^\cdot b(r,x+W_r)\mathd r\Big\rrbracket_{\tilde\gamma}^2 \bigg)\Bigg]\leq K(\eta) \quad \forall\, \eta>0
\end{equation}
uniformly over $x\in \RR^d$ and $b\in L^q_T B^\alpha_{\infty,\infty}$, $b\neq 0$.
\end{prop}

\begin{proof}
As the proof follows quite closely the ones given in \cite[Section 3.3]{galeati2020noiseless}, we only provide a sketch. Let $b$ be smooth and compactly supported, otherwise one can argue by density; up to reasoning componentwise, scaling and shifting, we can assume $x=0$, $b\in C^\infty_c(\RR^d)$ and $\| b\|_{L^q_T B^\alpha_{\infty,\infty}} = 1$, so it will never appear in the computations in the sequel. 

Set $W^{(2)}_{s,t}=\EE[W_t|\cF_s]$ for $\cF_s=\sigma\{W_r: r\leq s\}$, then by \cite[Lemma 5]{galeati2020noiseless} there exist $c_H,\,\tilde{c}_H>0$ such that
\begin{equation}\label{eq:itotanaka}\begin{split}
    \int_s^t b(r,W_r)\dd r
    & = \int_s^t P_{\tilde{c}_H |r-s|^{2H}} b(r,W^{(2)}_{s,r})\dd r
    \\ &\quad
    + c_H\int_s^t \int_u^t P_{\tilde{c}_H |r-u|^{2H}} \nabla b(r,W^{(2)}_{u,r})\, |r-u|^{H-1/2} \dd r \cdot \dd B_u
\end{split}\end{equation}
where $P_t$ denotes the Gaussian heat kernel, and $B_t$ is a standard Brownian motion in $\RR^d$. In the following we will drop the constants $c_H,\tilde{c}_H$, as they don't play any significant role. Thus for any fixed $s<t$, it holds
\[
\int_s^t b(r,W_r)\dd r 
    = I^1_{s,t} + I^2_{s,t} = I^1_{s,t} + \int_s^t J_{u,t}\cdot \dd B_u
\]
where
\begin{align*}
I^1_{s,t} := \int_s^t P_{|r-s|^{2H}} b(r,W^{(2)}_{s,r})\dd r, \quad
J_{u,t} :=\int_u^t P_{|r-u|^{2H}} \nabla b(r,W^{(2)}_{u,r}) |r-u|^{H-1/2} \dd r.
\end{align*}
Let us show how to obtain exponential estimates for $I^2$, the ones for $I^1$ being similar. Going through analogous computations to \cite[Theorem 4]{galeati2020noiseless}, invoking heat kernel type estimates, it holds
\begin{align*}
|J_{u,t}|
& \lesssim \int_u^t \| P_{|r-u|^{2H}} \nabla b_r \|_{L^\infty_x}\, |r-u|^{H-1/2} \dd r\\
& \lesssim \int_u^t \| b_r\|_{B^\alpha_{\infty,\infty}} |r-u|^{H\alpha-1/2} \dd r\\
& \lesssim  \| b\|_{L^q_T B^\alpha_{\infty,\infty}} |t-u|^{1/2 - 1/q + H\alpha}
= |t-u|^{\gamma-1/2}.
\end{align*}
Applying Burkholder-Davis-Gundy inequality with optimal asymptotic behaviour for large $p$, we deduce that
\begin{align*}
\EE [|I^2_{s,t}|^p]^{1/p}
\lesssim \sqrt{p}\, \EE\bigg[ \bigg(\int_s^t |J_{u,t}|^2\dd u\bigg)^{p/2}\bigg]^{1/p}
\lesssim \sqrt{p} |t-s|^\gamma
\end{align*}
Putting everything together, there exists a constant $C>0$ such that for any $\eta>0$
\begin{align*}
\EE\bigg[\exp\bigg(\eta\Big|\frac{I^2_{s,t}}{|t-s|^\gamma}\Big|^2\bigg)\bigg]
= \sum_{n\in\NN} \frac{\eta^n}{n!}\, \EE\bigg[\Big|\frac{I^2_{s,t}}{|t-s|^\gamma}\Big|^{2n}\bigg]
\leq \sum_{n\in\NN} (C\eta)^n \frac{n^n}{n!}
\end{align*}
and by Stirling's approximation the last series is convergent for any $\eta< (C e)^{-1}$. Together with similar estimates for $I^1$, we conclude that there exists $\tilde{\eta}>0$ sufficiently small and $C>0$ such that
\begin{equation*}
    \EE \Bigg[\exp \bigg(\eta\, \bigg| \frac{\int_s^t b(r,W_r)\mathd r}{|t-s|^\gamma }\bigg|^2\bigg)\Bigg]\leq C \quad \forall\, \eta\leq \tilde{\eta},\, s<t.
\end{equation*}
The above estimate together with \cite[Lemma 18]{galeati2020noiseless} implies that for any $\tilde{\gamma}<\gamma$ there exist $\bar{\eta}>0$ and $\kappa>0$ such that
\begin{equation}\label{eq:intermediate-stochaverag1}
\EE \Bigg[\exp \bigg(\eta \bigg\llbracket \int_0^\cdot b(r,W_r)\mathd r \bigg\rrbracket^2_{\tilde{\gamma}}\bigg)\Bigg]\leq \kappa \quad \forall\, \eta\leq \bar{\eta}.
\end{equation}
It remains to show that we can improve the above inequality by allowing any value $\eta>0$, so that we reach \eqref{eq:averaging-stoch-estimate-1}. To do so, we will resort to an interpolation trick, similar in style to techniques already applied in \cite[Theorem 15]{galeati2020noiseless}, \cite[Corollary 4.6]{Catellier2016}.

First, observe that if $\alpha,q,H$ satisfy \eqref{eq:simple-cond-1} and we fix $\tilde\gamma<\gamma$, then we can find $\eps$ sufficiently small so that $\gamma^\eps = 1-1/q-(\alpha-\eps)H>1/2$ and $\tilde\gamma < \gamma^\eps$; then by estimate \eqref{eq:intermediate-stochaverag1} (for $\alpha-\eps$ in place of $\alpha$) and linearity, there exist $\bar{\eta}>0$ and $\kappa>0$ such that
\begin{equation}\label{eq:intermediate-stochaverag2}
\EE \Bigg[\exp \bigg(\frac{\bar{\eta}}{\| \tilde b\|_{L^q_T B^{\alpha-\eps}_{\infty,\infty}}^2} \bigg\llbracket \int_0^\cdot \tilde b(r,W_r)\mathd r \bigg\rrbracket^2_{\tilde{\gamma}}\bigg)\Bigg]\leq \kappa \quad \forall\, \tilde b\in L^q_T B^{\alpha-\eps}_{\infty,\infty},\, \tilde b\neq 0.
\end{equation}
As before we can assume $\| b\|_{L^q_T B^\alpha_{\infty,\infty}}=1$ and we fix $\eps>0$ as above. Then for any $N\in\NN$ we can decompose $b$ as
\[ b_t = b^{1,N}_t + b^{2,N}_t,\quad b^{1,N}_t = \sum_{j\leq N} \Delta_j b_t,\quad b^{2,N}_t =\sum_{j>N} \Delta_j b_t
\]
where $\Delta_j$ denote Littlewood-Paley blocks. There exists $C>0$ such that
\[
\| b^{1,N}\|_{L^q_T C^0_x} \leq C\, 2^{-N\alpha}, \quad \| b^{2,N}\|_{L^q_T B^{\alpha-\eps}_{\infty,\infty}} \leq C\, 2^{-N\eps}. 
\]
Now for a given $\eta>0$, choose $N=N(\eta)\in\NN$ such that $\eta \leq C^{-2} 2^{2N\eps -1} \bar{\eta}$ and decompose $b$ as above; w.l.o.g. we may assume that $b^{2,N}\neq 0$, otherwise the stated estimate is trivial. Clearly under \eqref{eq:simple-cond-1} it holds that $\tilde\gamma\leq 1-1/q$, therefore setting $\beta=1-1/q-\tilde\gamma$ we have
\[
\bigg\llbracket \int_0^\cdot b^{1,N}(r,W_r)\dd r\bigg\rrbracket_{\tilde{\gamma}}
\leq (1+T)^\beta\, \| b^{1,N} \|_{L^q_T C^0_x}
\leq C (1+T)^\beta 2^{-N\alpha}=: C_{N(\eta)},
\] 
where the estimate is deterministic; combining it with \eqref{eq:intermediate-stochaverag2} applied to $\tilde b=b^{2,N}$, we get
\begin{align*}
\EE \Bigg[\exp \bigg(\eta \bigg\llbracket \int_0^\cdot b(r,W_r)\mathd r \bigg\rrbracket^2_{\tilde\gamma}\bigg)\Bigg]
& \leq \EE \Bigg[\exp \bigg(2\eta \bigg\llbracket \int_0^\cdot b^{1,N}(r,W_r)\mathd r \bigg\rrbracket^2_{\tilde\gamma}+2\eta \bigg\llbracket \int_0^\cdot b^{2,N}(r,W_r)\mathd r \bigg\rrbracket_{\tilde{\gamma}}^2\bigg)\Bigg]\\
& \leq \exp \big( 2\eta C_{N(\eta)}^2\big)\,
\EE \Bigg[\exp \bigg(\frac{\bar{\eta}}{\| b^{2,N}\|^2_{L^q_T B^{\alpha-\eps}_{\infty,\infty}}} \bigg\llbracket \int_0^\cdot b^{2,N}(r,W_r)\mathd r \bigg\rrbracket^2_{\tilde{\gamma}}\bigg)\Bigg]\\
& \leq \exp \big( 2\eta C_{N(\eta)}^2\big)\, \kappa
\end{align*}
where the estimate now holds for all $\eta\geq 0$.
\end{proof}

\begin{cor}\label{cor:main_avg_estimates2}
Let $b\in L^q_T B^\alpha_{\infty,\infty}$ with $\alpha<0$, $q\in (2,\infty]$, $W$ be a fBm of parameter $H\in (0,1)$ and let $\rho\in (0,1]$; suppose $(\alpha,\rho,q)$ satisfy
\begin{equation}\label{eq:simple-cond-2} \gamma:=1-\frac{1}{q}+(\alpha-\rho) H>\frac{1}{2}.
\end{equation}
Then for any $\tilde\gamma<\gamma$ there exists an increasing function $K$ (depending on $d,T$ and the previous parameters) such that
\begin{equation}\label{eq:averaging-stoch-estimate-2}
\EE \Bigg[\exp \bigg(\eta\, \bigg| \frac{\llbracket \int_0^\cdot b(r,x+W_r)\mathd r-\int_0^\cdot b(r,y+W_r)\mathd r\rrbracket_{\tilde{\gamma}}}{\| b\|_{L^q_T B^\alpha_{\infty,\infty}} |x-y|^{\rho}}\bigg|^2\bigg)\Bigg]\leq K(\eta)\quad \forall\, \eta\geq 0
\end{equation}
uniformly over $x\neq y\in \RR^d$ and $b\in L^q_T B^\alpha_{\infty,\infty}$, $b\neq 0$; as a consequence, for any $\eps>0$, $\PP$-a.s. $T^W b \in C^{\gamma-\eps}_T C^{\rho-\eps}_{loc}$.
Suppose now $b\in L^q_T B^\alpha_{\infty,\infty}$ with $\alpha<1$, $q\in (2,\infty]$ satisfying
\begin{equation}\label{eq:good parameters}
\alpha-\frac{1}{Hq}>1-\frac{1}{2H},
\end{equation}
then the following hold:
\begin{enumerate}[label=\roman*.]
\item \label{it:main-avg1} There exists $\tilde\gamma>1/2$ such that $\PP$-a.s. $T^W b\in C^{\tilde \gamma}_T C^1_{loc}$.
\item \label{it:main-avg2} There exists $\tilde\gamma>1/2$ such that for any $b^1,b^2\in L^q_T B^\alpha_{\infty,\infty}$ and any $n\in \NN$
\begin{equation*}
\EE \Bigg[ \Big\llbracket \int_0^\cdot b^1(r,W_r)\mathd r - \int_0^\cdot b^2(r,W_r)\mathd r\Big\rrbracket_{\gamma;[0,\tau]}^n\Bigg]^{1/n}
\lesssim_n
\bigg( \int_0^\tau \| b^1_r-b^2_r\|^q_{B^{\alpha-1}_{\infty,\infty}} \dd r\bigg)^{1/q}\quad \forall\,\tau\in [0,T].
\end{equation*}
\item \label{it:main-avg3} If $H<1/2$, $\alpha<0$, there exists $\tilde\gamma>H+1/2$ and an increasing function $K$ such that
\begin{equation*}
\EE \Bigg[\exp \bigg(\frac{\eta}{\| b\|^2_{L^q_T B^\alpha_{\infty,\infty}}} \bigg\llbracket \int_0^\cdot b(r,W_r)\mathd r \bigg\rrbracket^2_{\tilde\gamma}\bigg)\Bigg] \leq K(\eta)
\quad \forall\, \eta\geq 0,\, b\in L^q_T B^\alpha_{\infty,\infty}, b\neq 0.
\end{equation*}
\end{enumerate}
\end{cor}

\begin{proof}
Given $b$ as above, $x\neq y$ fixed, define $\tilde b(t,\cdot)= |x-y|^{-\rho}\, [b(t,x+\cdot)-b(t,y+\cdot)]$; by properties of Besov spaces
\[
\| \tilde{b}\|_{L^q_T B^{\alpha-\rho}_{\infty,\infty}} \lesssim \| b \|_{L^q_T B^\alpha_{\infty,\infty}}.
\]
Inequality \eqref{eq:averaging-stoch-estimate-2} follows from \eqref{eq:averaging-stoch-estimate-1} applied to $\tilde b$, since by assumption \eqref{eq:simple-cond-2} $\tilde\alpha=\alpha-\rho$ satisfies \eqref{eq:simple-cond-1}; $T^W b$ belonging to $C^{\gamma-\eps}_T C^{\rho-\eps}_{loc}$ is a consequence of Garsia-Rodemich-Rumsay lemma.

We now assume \eqref{eq:good parameters} holds and prove points \ref{it:main-avg1}-\ref{it:main-avg3}

If $b\in L^q_T B^\alpha_{\infty,\infty}$, then $D_x b\in L^q_T B^{\alpha-1}_{\infty,\infty}$ with $\tilde\alpha=\alpha-1$ satisfying \eqref{eq:simple-cond-1}, so we can find $\rho>0$ small enough such that $(\tilde\alpha,\rho)$ satisfy \eqref{eq:simple-cond-2} as well. It follows that $T^W D_x b=D_x T^W b\in C^{\gamma-\eps}_T C^0_{loc}$, namely $T^W b\in C^{\gamma-\eps}_T C^1_{loc}$, for any $\eps>0$, showing \ref{it:main-avg1}.

For $\tau=T$, the statement in part \ref{it:main-avg2} is again a consequence of \eqref{eq:averaging-stoch-estimate-1} (for $x=0$ and $\tilde\alpha=\alpha-1$) and the linearity of $b\mapsto T^W b$. For general $\tau\in [0,T]$, define $\tilde{b}^i_t= b^i_t\, \mathds{1}_{[0,\tau]}(t)$ and observe that
\[
\Big\llbracket \int_0^\cdot (b^1-b^2)(r,W_r)\mathd r \Big\rrbracket_{\gamma;[0,\tau]} = \Big\llbracket \int_0^\cdot (\tilde b^1-\tilde b^2)(r,W_r)\mathd r \Big\rrbracket_{\gamma;[0,T]};
\]
the estimate for general $\tau$ thus follows applying the one for $\tau=T$ to $\tilde b^i$.

Finally, in order to prove \ref{it:main-avg3}  it is enough to show that
\[
\gamma = 1-\frac{1}{q}+\alpha H > H + 1/2,
\]
as in that case we can find $\tilde{\gamma}\in (H+1/2,\gamma)$ such that \eqref{eq:averaging-stoch-estimate-1} holds. But the above condition on $\gamma$ is exactly \eqref{eq:good parameters}.
\end{proof}

In order to apply Lemma \ref{lem:criteria-uniqueness-YDE}, we need some information on the pathwise properties of weak solutions $X$.
From this perspective, techniques based on Girsanov theorem are vey natural, as they suggest that $T^X b$ may have the same regularity as  $T^W b$.
As already mentioned, Girsanov transform holds for fBm, see \cite{nualart2002regularization}; sufficient conditions in order to apply it in our context (in particular to check that Novikov condition is satisfied) can be found in \cite[Section 4.2.2]{galeati2020noiseless}, to which we also refer for more details on the explicit formula for $\dd \PP/ \dd \QQ$.

\begin{prop}\label{prop:Girsanov sufficient}
Let $(\Omega,\mathcal{F},\{\mathcal{F}_t\}_{t\geq 0},\PP)$ be a filtered probability space, $W$ be an $\mathcal{F}_t$-fBm of parameter $H\in (0,1)$ and $h$ be an $\mathcal{F}_t$-adapted process with trajectories in $C^\gamma_T$, $\gamma>H+1/2$, such that $h_0=0$ and
\[
\EE_\PP[\exp(\eta \llbracket h\rrbracket_{\gamma}^2)] \leq K(\eta)<\infty \quad\forall\, \eta\in\RR.
\]
Then there exists another probability measure $\QQ$, given by Girsanov theorem, such that $h+W$ is distributed as an $\mathcal{F}_t$-fBm under $\QQ$. Moreover $\PP$ and $\QQ$ are equivalent and it holds
\begin{equation}\label{eq:controll of RN derivative}
 \EE_\QQ\Big[ \Big( \frac{\dd \PP}{\dd \QQ}\Big)^n + \Big(\frac{\dd \QQ}{\dd \PP} \Big)^n\Big] <\infty \quad \forall\, n\in \NN   
\end{equation}
where the above estimate only depends on the function $K$.
\end{prop}

\begin{proof}
Follows almost exactly as the proof of \cite[Theorem 14]{galeati2020noiseless}.
\end{proof}

\begin{rem}\label{rem: reg for a>0}
For $H\leq 1/2$ and $b\in L^q_T B^\alpha_{\infty,\infty}$ with $(\alpha,q,H)$ satisfying \eqref{eq:good parameters}, it follows from Corollary \ref{cor:main_avg_estimates2} and Proposition \ref{prop:Girsanov sufficient} that we can construct a weak solution $(\Omega,\mathcal{F},\PP;X,W)$ to the SDE associated to $(x_0,b)$, with the property that there exists a measure $\QQ$ equivalent to $\PP$ such that $\cL_\QQ(X)=\cL_\PP(x_0+W)$;
moreover all the moments of $\dd \PP/\dd \QQ$ and $\dd \QQ/\dd \PP$ can be controlled in a way that depends on $\| b\|_{L^q_T B^\alpha_{\infty,\infty}}$ but not on the specific $(x_0,b)$.
In particular, the estimates can be performed uniformly over $x_0\in\RR^d$ and $\| b\|_{L^q_T B^\alpha_{\infty,\infty}}\leq M$ for a fixed parameter $M>0$.

Similarly, in the case $H>1/2$, given $b\in E=C^{\alpha H}_T C^0_x\cap C^0_T C^\alpha_x$ for some $\alpha>1-1/(2H)$, using the regularity of fBm trajectories it's easy to check that the map $t\mapsto b(t,x_0+W_t)$ belongs $\PP$-a.s. to $C^{\alpha H-\eps}_T$ for any $\eps>0$. Furthermore, reasoning as in the proof of \cite[Theorem 15]{galeati2020noiseless}, it can be shown that there exists $\gamma>H+1/2$ and an increasing function $K$ such that
\[
\EE \bigg[\exp\bigg(\eta \Big\| \int_0^\cdot b(r,x_0+W_r)\dd r\Big\|_{\gamma}^2\bigg)\bigg]\leq K(\eta)<\infty\quad \forall\, \eta\geq 0.
\]
Therefore also in this case we can apply Proposition \ref{prop:Girsanov sufficient} to construct weak solutions to the SDE.
Moreover the function $K$ only depends on $\| b\|_E$, therefore as before all estimates are uniform over $x_0\in \RR^d$ and $b\in E$ with $\| b\|_E\leq M$, $M$ fixed parameter.

If both cases, if in addition $b$ is smooth, then the weak solution constructed in this way necessarily coincides with the unique strong one; thus the above reasoning also provide uniform estimates for the solutions associated to smooth drifts.
\end{rem}

\subsection{Stability estimates for SDEs}\label{subsec:stability}

In light of the above results, in the remainder of Section \ref{sec:recap-SDE} we will always impose the following assumption on the drift $b$.

\begin{ass}\label{ass:drift-SDE}
Given $H\in (0,1)$, $b$ satisfies one of the following:
\begin{itemize}
    \item If $H>1/2$, then $b\in C^{\alpha H}_T C^0_x\cap b\in C^0_T C^\alpha_x$ for some
    \[\alpha>1-\frac{1}{2H};\]
    equivalently, there exists a constant $C>0$ s.t., for all $(s,t,x,y)\in [0,T]^2\times \RR^{2d}$, such that
    \begin{equation*}
        |b(t,x)|\leq C, \quad |b(t,x)-b(s,y)|\leq C(|t-s|^{\alpha H} + |x-y|^\alpha).
    \end{equation*}
    \item If $H\leq 1/2$, then $b\in L^q_T B^\alpha_{\infty,\infty}$ for some $(\alpha,q)$ satisfying \eqref{eq:good parameters}.
\end{itemize}
In both cases we will use the notation $\| b\|_E$ for $E=C^{\alpha H}_T C^0_x\cap C^0_T C^\alpha_x$ when $H>1/2$, respectively $E=L^q_T B^\alpha_{\infty,\infty}$ when $H\leq 1/2$.
\end{ass}

We are now ready to present the main result of this section.

\begin{thm}\label{thm:fBmSDEWellPosed}
Let $W$ be an fBm of parameter $H\in (0,1)$ and let $b$ satisfy Assumption \ref{ass:drift-SDE}. Then for any $x_0\in \RR^d$ strong existence, pathwise uniqueness and uniqueness in law hold for the SDE
\begin{equation}\label{eq:SDE thm}
X_t = x_0 + \int_0^t b(r,X_r)\mathd r + W_t
\end{equation}
in the sense of Definition \ref{def:pathwise-strong-weak-sol}.
Given $x_0^i\in\RR^d$ and $b^i$ satisfying Assumption \ref{ass:drift-SDE}, $i=1,2$, denote by $X^i$ the solutions associated to $(x^i_0,b^i)$ and let $M>0$ be a constant such that $\|b^i\|_E\leq M$ for $i=1,2$.
Let $(\alpha,\tilde q)$ be another pair satisfying \eqref{eq:good parameters} with the same $\alpha$ as in Assumption \ref{ass:drift-SDE} and $\tilde q\leq q$.
Then there exists $\gamma>1/2$ with the following property: for any $p\in [1,\infty)$ there exists a constant $C>0$ (depending on $\gamma, p, M, T, d, \tilde q$ and the parameters appearing in Assumption \ref{ass:drift-SDE}) such that
\begin{equation}\label{eq:SDE main est}
\EE \Big[\| X^1-X^2\|^p_{\gamma;[0,\tau]} \Big]^{1/p} \leq C \Big(|x^1_0-x^2_0|+\|b^1-b^2\|_{L^{\tilde q}(0,\tau; B^{\alpha-1}_{\infty,\infty})}\Big)\quad \forall\, \tau\in [0,T].
\end{equation}
\end{thm}

\begin{proof}
We will only treat the case $H\leq 1/2$, the other one being almost identical.

Let us first assume $b^i$ to be smooth functions and show that \eqref{eq:SDE main est} holds; in this case by Remark \ref{rem:SDE-classical-case} strong existence and uniqueness hold automatically. Moreover by Remark \ref{rem: reg for a>0}, there exist probability measures $\QQ^i$ equivalent to $\PP$ such that $\cL_{\QQ^i}(X^i)=\cL_\PP(x^i+W)$, with moment estimates depending on $M$ but not on $(x^i_0,b^i)$; the solutions decompose as $X^i=x^i_0+h^i+W^i$ with $h^i_0=0$, $h^i\in C^{\tilde \gamma}_T$ with $\tilde{\gamma}>H+1/2$ and such that
\[
\EE \big[\exp\big(\eta \| h^i\|_{\gamma}^2\big)\big]\leq K(\eta)<\infty\quad \forall\, \eta\geq 0
\]
where again $K$ depends on $M$ but not on the specific $(x^i_0,b^i)$.

For any $\lambda\in [0,1]$, let us define $x^\lambda_0:=x^2_0+\lambda(x^1_0-x^2_0)$, $h^\lambda:=h^2+\lambda(h^1-h^2)$, so that $X^2+\lambda(X^1-X^2) = x^\lambda_0+h^\lambda+W$. By Taylor expansion and elementary addition and subtraction, the difference $Y=X^1-X^2$ satisfies
\begin{equation*}
    Y_t= x^1_0-x^2_0 + \int_0^t \bigg(\int_0^1 D b^1 (r, x^\lambda+h^\lambda_r+W_r)\mathd \lambda\bigg) \cdot Y_r\,\mathd r + \int_0^t (b^1-b^2)(r,X^2_r)\mathd r;
\end{equation*}
let us define
\begin{align*}
A_t := \int_0^t \int_0^1 D b^1 (r, x^\lambda_0+h^\lambda_r+W_r)\mathd \lambda\mathd r,\quad
\psi_t := \int_0^t (b^1-b^2)(r,X^2_r)\mathd r.
\end{align*}
In order to get estimates for $Y$, it turns out to be useful to reinterpret the above equation as a linear Young differential equation of the form
\begin{equation}\label{eq:LYE}
    Y_t=Y_0+\int_0^t A_{\dd s} Y_s+\psi_t. 
\end{equation}
Indeed for any $\gamma>1/2$, we can apply \cite[estimate (3.16), Theorem 3.9]{galeati2020nonlinear} to obtain the existence of a constant $C=C(\gamma)$ such that for any $\tau\leq T$ it holds
\begin{align*}
\|Y\|_{\gamma;[0,\tau]}
& \leq C \exp(C\tau (1+\llbracket A\rrbracket_{\gamma;[0,\tau]}^2)) (|Y_0|+(1+\tau^\alpha) \llbracket \psi\rrbracket_{\gamma;[0,\tau]})\\
& \lesssim_T \exp(C T \llbracket A\rrbracket_{\gamma}^2) (|x_0^1-x_0^2|+\llbracket \psi\rrbracket_{\gamma;[0,\tau]})
\end{align*}
and so our task reduces to finding estimates for quantities of the form
\[
\EE_\PP\big[ \exp(\eta \llbracket A\rrbracket_\gamma^2)\big],\quad \EE_\PP\big[\| \psi\|_{\gamma;[0,\tau]}^p\big].
\]
We start by estimating $\psi$, which is the simplest term. Recalling that $\cL_{\QQ^2}(X^2)=\cL_\PP(x^2_0+W)$, by Point \ref{it:main-avg2} of Corollary  \ref{cor:main_avg_estimates2} and Cauchy inequality we can find $\gamma>1/2$ such that, for any $p\geq 1$,
\begin{align*}
  \EE_\PP \big[\| \psi\|_{\gamma;[0,\tau]}^p\big]
  & =\EE_\PP\bigg[ \Big\llbracket\int_0^\cdot (b^1-b^2)(r,X^2_r)\dd r\Big\rrbracket_{\gamma;[0,\tau]}^p\bigg]
  = \EE_{\QQ^2}\bigg[ \Big\llbracket\int_0^\cdot (b^1-b^2)(r,X^2_r)\dd r\Big\rrbracket_{\gamma;[0,\tau]}^p \, \frac{\dd \PP}{\dd \QQ^2}\bigg]\\
  & \leq \EE_{\QQ^2}\bigg[ \Big(\frac{\dd \PP}{\dd \QQ^2}\Big)^2\bigg]^{1/2}\, \EE_{\QQ^2}\bigg[ \Big\llbracket\int_0^\cdot (b^1-b^2)(r,X^2_r)\dd r\Big\rrbracket_{\gamma;[0,\tau]}^{2p}\bigg]^{1/2}\\
  & \lesssim_M \EE_{\PP}\bigg[ \Big\llbracket\int_0^\cdot (b^1-b^2)(r,x^2_0+W_r)\dd r\Big\rrbracket_{\gamma;[0,\tau]}^{2p}\bigg]^{1/2}
  \lesssim_p \|b^1-b^2\|_{L^{\tilde q}(0,\tau;B^{\alpha-1}_{\infty,\infty})}.  
\end{align*}
In order to get estimates for $A$, observe first of all that by convexity of $z\mapsto \exp(\eta z^2)$, it holds
\begin{align*}
    \EE \big[\exp\big(\eta \| h^\lambda\|_{\gamma}^2\big)\big]
    \leq \lambda \EE \big[\exp\big(\eta \| h^1\|_{\gamma}^2\big)\big] + (1-\lambda) \EE \big[\exp\big(\eta \| h^2\|_{\gamma}^2\big)\big]
    \leq K(\eta)<\infty
\end{align*}
where the estimate is uniform in $\lambda$ and $\eta$; therefore by Proposition \ref{prop:Girsanov sufficient}, for any $\lambda$ there exists a probability $\QQ^\lambda$ equivalent to $\PP$ such that $\cL_{\QQ^\lambda}(h^\lambda+W)=\cL_\PP(W)$; moreover estimates of the form \eqref{eq:controll of RN derivative} only depend on $K$ and thus on $M$, but not $(x^i_0,b^i)$. Therefore by Jensen's inequality and Proposition \ref{prop:main_avg_estimates-1}, we can find $\gamma>1/2$ such that, for any $\eta\geq 0$, it holds that
\begin{align*}
    \EE_{\PP}\big[\exp\left(\eta\llbracket A\rrbracket^2_\gamma\right)\big] 
    & = \EE_\PP\bigg[\exp \bigg(\eta \Big\llbracket \int_0^1 \int_0^\cdot Db^1(r,x^\lambda_0 + h^\lambda_r + W_r)\dd r\Big\rrbracket^2_\gamma\bigg)\bigg]\\
    & \leq \int_0^1 \EE_\PP\bigg[\exp \bigg(\eta \Big\llbracket \int_0^\cdot Db^1(r,x^\lambda_0 + h^\lambda_r + W_r)\Big\rrbracket^2_\gamma\bigg)\bigg]\\
    & = \int_0^1 \EE_{\QQ^\lambda}\bigg[\exp \bigg(\eta \Big\llbracket \int_0^\cdot Db^1(r,x^\lambda_0 + h^\lambda_r + W_r)\Big\rrbracket_\gamma\bigg) \frac{\dd \PP}{\dd \QQ^\lambda}\bigg]\\
    & \leq \int_0^1  \EE_{\PP}\bigg[\exp\bigg(2\eta\Big\llbracket \int_0^\cdot D b^1(r,x^\lambda_0 + W_r)\dd r \Big\rrbracket_\gamma^2\bigg)\bigg]^{1/2} \EE_{\QQ^\lambda} \bigg[\Big(\frac{\dd \PP}{\dd \QQ^\lambda}\Big)^2\bigg]^{1/2}\dd \lambda\\
    & \lesssim_M \EE_{\PP}\left[\exp\left(2\eta\Big\llbracket \int_0^\cdot D b^1(r, W_r)\dd r \Big\rrbracket_\gamma^2\right)\right]^{1/2}
    \leq K(2\eta)^{1/2}.
\end{align*}
Putting everything together, we have obtained
\begin{align*}
    \EE_\PP \big[\| Y\|_{\gamma;[0,\tau]}^p\big]
    & \lesssim \EE_\PP\big[ \exp(p C T\llbracket A\rrbracket_\gamma^2) (|Y_0|^p+\llbracket \psi\rrbracket_{\gamma;[0,\tau]}^p)\big]\\
    & \lesssim \EE_\PP\big[ \exp(2pCT\llbracket A\rrbracket_\gamma^2)\big]^{1/2} \left(|Y_0|^p+\EE_\PP\big[ \llbracket \psi\rrbracket_{\gamma;[0,\tau]}^{2p}\big]^{1/2}\right)\\
    & \lesssim_{p,T,M} |x_0^1-x_0^2|^p + \|b^1-b^2 \|_{L^{\tilde q}(0,\tau;B^\alpha_{\infty,\infty})}^p
\end{align*}
which proves \eqref{eq:SDE main est} for smooth $b^i$.

Assume now we are given $x_0\in \RR^d$ and $b$ satisfying Assumption \ref{ass:drift-SDE}; we can find $\tilde{q}\leq q$, $\tilde{q}<\infty$ such that $(\alpha,\tilde q)$ satisfy \eqref{eq:good parameters} and a sequence $\{b^n\}_n$ be smooth drifts s.t. $\| b^n\|_E\leq \|b\|_E$ for all $n\geq 1$ and $b^n\to b$ in $L^{\tilde q}_T B^{\alpha-\eps}_{\infty,\infty}$ for any $\eps>0$ (for instance set $b^n = b\ast \psi^n$ with $\{\psi^n\}_{n\geq 1}$ a standard family of mollifiers). Let $X^n$ be the unique solutions to \eqref{eq:SDE thm} associated to $(x_0,b^n)$, then by \eqref{eq:SDE main est} it holds
\[
\EE\big[\| X^n-X^m\|_\gamma^p\big]^{1/p} \lesssim \| b^n-b^m\|_{L^{\tilde q}_T B^{\alpha-1}_{\infty,\infty}}
\]
showing that the random variables $\theta^n=X^n-W$ are a Cauchy sequence in $L^p_\Omega C^\gamma_T$.
Therefore they converge to a unique limit $\theta$, which is adapted to the filtration $\cF_t=\sigma\{W_s:s\leq t\}$ since $\theta^n$ are so. Similarly the $X^n$ converge to $X=\theta+W$ which is adapted.

The estimates from Corollary \ref{cor:main_avg_estimates2}, the linearity of $b\mapsto T^w b$ and the property $b^n\to b$ in $L^{\tilde q}_T B^{\alpha-\eps}_{p,p}$ together imply that $\PP$-a.s. $T^W b^n\to T^W b$ in $C^\gamma_T C^1_{loc}$. Since we have $\PP$-a.s. $\theta^n\to \theta$ in $C^\gamma_T$ as well, we can invoke the closure property of nonlinear Young equations (Lemma \ref{lem:closure-YDE}) to deduce that $X=\theta + W$ is a pathwise solution to \eqref{eq:SDE thm} in the sense of Definition \ref{def:SDE-as-NLYE}.

Furthermore, by Fatou's lemma
\[
\EE_\PP\big[ \exp(\eta \llbracket \theta\rrbracket_{\tilde{\gamma}}^2)\big]
\leq \liminf_{n\to\infty} \EE_\PP\big[ \exp(\eta \llbracket \theta^n\rrbracket_{\tilde{\gamma}}^2)\big] \leq K(\eta)<\infty \quad \forall \eta\geq 0;
\]
it follows that Girsanov can be applied to $X=\theta+W=x_0+h+W$  to deduce that $X$ is distributed as $x_0+ W$ under another probability measure equivalent to $\PP$.
In particular, $\PP$-a.s. it must hold $T^X b\in C^\gamma_T C^1_{loc}$.
To summarise, $X$ is a strong solution (so that a copy of it can be constructed on \textit{any} probability space supporting the measure $\mu^H$) such that $T^X b\in C^\gamma_T C^1_{loc}$, which implies by Lemma \ref{lem:criteria-uniqueness-YDE} that pathwise uniqueness must hold. This also implies that the law of any solution coincides with the one constructed by Girsanov theorem, from which uniqueness in law follows.

The extension of inequality \eqref{eq:SDE thm} to any pair of solutions $X^i$ associated to distributional drifts $b^i$ is now a direct consequence of the approximation argument.
\end{proof}
\begin{rem}
At the price of making the statement of Theorem \ref{thm:fBmSDEWellPosed} slightly more technical, we have allowed the presence of the additional parameter $\tilde q\leq q$ to handle $q=\infty$.
Indeed finding approximation sequences in $L^\infty_T B^{\alpha-1}_{p,p}$ can be a hard task since this is not a separable space; the use of $L^{\tilde{q}}_T B^{\alpha-1}_{p,p}$ with $\tilde q<\infty$ will also be useful later in the proofs in Section \ref{subsec:stability-DDSDE}.
\end{rem}

\begin{rem}
Theorem \ref{thm:fBmSDEWellPosed} gives us the information that, for drifts $b$ satisfying Assumption \ref{ass:drift-SDE}, the nonlinear Young interpretation of the SDE is the only physical one. Namely, any other solution concept sharing the fundamental property of being the limit of solutions associated to smooth drifts $b^n\to b$ will coincide with ours. The statement of Theorem \ref{thm:fBmSDEWellPosed} can be further strengthened to establish \textit{path-by-path uniqueness}, see \cite{Catellier2016}, however we will not need this for our purposes.
\end{rem}

\begin{rem}
Although we have proved the stability estimate \eqref{eq:SDE main est} in order to apply to DDSDEs, it is of interest on its own. Indeed it can be applied to construct the stochastic flow associated to SDE \eqref{eq:SDE thm}, or to develop numerical schemes for distributional drifts $b$ by first approximating them by smoother $b^n$. We leave both applications for future research.
\end{rem}

The next lemmas extend the previous results to the case of random initial data.

\begin{cor}\label{cor:RndDataWellPosed}
Given $H\in (0,1)$ and $b$ satisfying Assumption \ref{ass:drift-SDE}, strong existence, uniqueness in law and pathwise uniqueness also hold for random initial data $X_0=\xi$ independent of $W$. Assume $b^i$ are drifts satisfying the assumptions of Theorem \ref{thm:fBmSDEWellPosed} 
and $(\xi^1,\xi^2)\in L^p(\Omega;\RR^{2d})$ is independent of $W$, then the solutions $X^i$ associated to $(\xi^i,b^i)$ satisfy
\begin{equation}\label{eq:SDEGammaNormEstim}
\EE_\PP \big[\| X^1-X^2\|^p_{\gamma;[0,\tau]} \big]^{1/p} \leq C \Big(\EE_\PP[|\xi^1-\xi^2|^p]^{1/p}+\|b^1-b^2\|_{L^{\tilde q}(0,\tau; B^{\alpha-1}_{\infty,\infty})}\Big)\quad \forall\,\tau\in [0,T].
\end{equation}
where the constant $C$ and the parameters $\gamma,\,\alpha\,\tilde q$ are the same as in \eqref{eq:SDE main est}.
Moreover, denoting by $\mu^i_t = \cL(X^i_t)$ the laws of the unique solutions $X^i$ associated to $(\mu_0^i,b^i)$ with $\cL(\xi^i,W)=\mu_0^i\otimes \cL(W)$, it holds
\begin{equation}\label{eq:SDE-estim-wasserstein}
\sup_{t\in [0,\tau]} d_p (\mu^1_t,\mu^2_t) \leq C \Big(d_p(\mu^1_0,\mu^2_0) + \|b^1-b^2 \|_{L^{\tilde q}(0,\tau; B^{\alpha-1}_{\infty,\infty})}\Big)\quad \forall\, \tau\in [0,T].
\end{equation}
\end{cor}

\begin{proof}
Strong existence and pathwise uniqueness for random initial data follows from that for deterministic ones by classical arguments.
Given a probability space $(\Omega,\cF,\PP)$ with $(W,\xi^1,\xi^2)$ defined on it and drifts $(b^1,b^2)$, we can condition on the variables $(\xi^1,\xi^2)$ independent of $W$ and apply estimate \eqref{eq:SDE main est} to deduce that
\begin{align*}
    \EE_\PP \big[\| X^1-X^2\|^p_{\gamma;[0,\tau]} \big\vert\, \xi^1,\xi^2 \big]^{1/p} \leq C \Big( |\xi^1-\xi^2|+\|b^1-b^2\|_{L^{\tilde q}(0,\tau; B^{\alpha-1}_{\infty,\infty})}\Big);
\end{align*}
inequality \eqref{eq:SDEGammaNormEstim} follows taking the $L^p_\Omega$-norm on both sides, using the tower property of conditional expectation.

Now assume we are given a pair $(\mu_0^1,\mu_0^2)\in \cP(\RR^d)\times \cP(\RR^d)$ and let $m\in \Pi(\mu^1_0,\mu^2_0)$ be an optimal coupling for them. On the canonical space $\Omega=\RR^{2d}\times C_T$, endowed with $\PP=m\otimes \mu^H$, we can construct random variables $(\xi^1,\xi^2,W)$ and solutions $X^i$ associated to $(\xi^i,b^i)$, in such a way that $\cL_\PP(\xi^1,\xi^2)=m$, $\EE[|\xi^1-\xi^2|^p]^{1/p}=d_p(\mu^1_0,\mu^2_0)$. But then by the definition of $d_p$ it must hold $d_p(\mu^1_t,\mu^2_t) \leq \| X^1_t-X^2_t\|_{L^p_\Omega}$ and so estimate \eqref{eq:SDE-estim-wasserstein} follows from \eqref{eq:SDEGammaNormEstim} applied in this setting.
\end{proof}

\begin{cor}\label{cor:SDE-random-data-girsanov}
Let $H\in (0,1)$, $b$ satisfying Assumption \ref{ass:drift-SDE}, $\xi$ random initial data independent of $W$ and $X$ be the solution associated to $(\xi,b)$. Then there exists another probability measure $\QQ$ equivalent to $\PP$ such that $\cL_\QQ(X_\cdot)=\cL_\PP(\xi+W_\cdot)$; moreover
\begin{equation*}
\EE_\QQ\Big[ \Big(\frac{\mathd \PP}{\mathd \QQ}\Big)^n + \Big(\frac{\mathd \QQ}{\mathd \PP}\Big)^n\Big] <\infty \quad \forall\, n\in\NN .
\end{equation*}
\end{cor}

\begin{proof}
It suffices to work on the canonical space $(\Omega,\cF,\PP)$ with $\Omega=\RR^d\times C_T\ni(x,\omega)$, $\PP=\mu_0\otimes \mu^H$ where $\mu_0:=\cL(\xi)$. For any $x\in\RR^d$, denote by $\omega\mapsto X^x(\omega)$ the unique strong solution associated to $(x,b)$, so that $(x,\omega)\mapsto X^x(\omega)$ gives the solution to the SDE with initial distribution $\mu_0$.
Recall that for any $x\in \RR^d$, there exists a probability measure on $C_T$ denoted by $\QQ^x$, equivalent to $\mu^H$, such that $\cL_{\QQ^x} (X^x)=\cL_{\mu^H}(x+W)$; therefore for any measurable $F:\Omega\to\RR$
\begin{align*}
    \EE_\PP[F(\xi + W)]
    & = \int_{\RR^d} \int_{C_T} F(x+\omega)\, \mu^H(\dd \omega) \mu_0(\dd x)\\
    & = \int_{\RR^d} \int_{C_T} F(X^x(\omega)) \frac{\dd \mu^H}{\dd \QQ^x}(\omega)\, \QQ^x(\dd \omega) \mu_0(\dd x).
\end{align*}
Thus if we define a probability measure $\QQ$ on $\Omega=\RR^d\times C_T$ by
\begin{align*}
    \QQ(E_1\times E_2)=\int_{E_1} \QQ^x(E_2)  \mu_0(\dd x) \quad \forall\, E_1\in \cB(\RR^d),\, E_2\in \cB(C_T),
\end{align*}
it must hold that $\cL_\QQ(X)=\cL_\PP(\xi+W)$; since $\mu^H\ll \QQ^x$ for every $x$, $\PP=\mu^H\otimes \mu^\xi\ll \QQ$ with Radon-Nikodym derivative given by
\[
\frac{\dd \PP}{\dd \QQ}(x,\omega) = \frac{\dd \mu^H}{\dd \QQ^x} (\omega)\quad \text{for }\PP\text{-a.e. } (x,\omega).
\]
Exploiting the bounds from Proposition \ref{prop:Girsanov sufficient} (which for given $b$ are uniform in $x\in\RR^d$) we find
\begin{align*}
    \EE_\QQ\Big[ \Big(\frac{\mathd \PP}{\mathd \QQ}\Big)^n + \Big(\frac{\mathd \QQ}{\mathd \PP}\Big)^n\Big]
    & = \int_{\RR^d} \int_{C_T} \bigg[\Big( \frac{\dd \mu^H}{\dd \QQ^x}(\omega) \Big)^n + \Big( \frac{\dd \QQ^x}{\dd \mu^H}(\omega) \Big)^n\bigg]\, \QQ^x(\dd \omega) \mu_0(\dd x)\\
    & \lesssim_{n,b} \int_{\RR^d} 1\, \mu_0(\dd x) <\infty
\end{align*}
providing the conclusion.
\end{proof}

\begin{rem}\label{rem:holder-bound-solution}
It follows from the above that for any $p\in [1,\infty)$ and any $\eps>0$
\begin{align*}
\EE_\PP\big[\llbracket X\rrbracket_{H-\eps}^p\big]
\leq \EE_\QQ\big[\llbracket X\rrbracket_{H-\eps}^{2p}\big]^{1/2}
\,\EE_\QQ\bigg[ \Big(\frac{\mathd \PP}{\mathd \QQ}\Big)^{2p}\bigg]^{1/2}
\lesssim \EE_\PP\big[\llbracket W\rrbracket_{H-\eps}^{2p}\big]^{1/2}<\infty.
\end{align*}
In particular, if $\xi\in \cP_{\tilde p}$ for another $\tilde p\in [1,\infty)$, then $\EE_\PP[\| X\|_{H-\eps}^{\tilde p}]<\infty$. As in the case of Remark \ref{rem: reg for a>0}, for fixed $\xi$ the estimate can be performed uniformly over $\|b\|_E \leq M$.
\end{rem}

\subsection{Regularity of the solution laws}\label{subsec:reg-estim}

Although our main interest is the study of DDSDEs, our analysis also yields results on the regularity of the law $\cL(X_t)$ for the solution to a standard SDE with singular drift. The method is quite simple but appears to be new and does not rely on PDE techniques nor Malliavin calculus; rather we exploit Girsanov transform and the averaging estimates for fBm, in combination with duality arguments.

\begin{prop}\label{prop:regularity-density1}
Let $b$ satisfy Assumption \ref{ass:drift-SDE}, $X$ be the solution associated to $(\xi,b)$ for random initial $\xi$ independent of $W$.
Then $\cL(X_\cdot)\in L^q_T B^\alpha_{1,1}$ for all $(\alpha,q)\in (0,\infty)\times (1,2)$ satisfying
\begin{equation}\label{eq:parameters-regularity-density1}
\alpha < \frac{1}{H} \Big(\frac{1}{q}-\frac{1}{2}\Big).
\end{equation}
\end{prop}

\begin{proof}
Observe that if $(\alpha, q, H)$ satisfy \eqref{eq:parameters-regularity-density1}, then we can find $\eps>0$ small enough so that $(-\alpha-2\eps,q',H)$ satisfy the assumptions of Proposition \ref{prop:main_avg_estimates-1}, where $q'$ denotes the conjugate of $q$.
By Corollary \ref{cor:SDE-random-data-girsanov}, there exists an equivalent measure $\QQ$ such that $\cL_\QQ(X)=\cL_\PP(\xi+W)$, therefore for any $f\in L^{q'}_T B^{-\alpha-2\eps}_{\infty,\infty}$ it holds
\begin{align*}
\bigg|\int_0^T \langle f_s, \cL_\PP(X_s)\rangle \dd s\bigg|
& \leq \EE_\PP \bigg[\Big|\int_0^T f(s,X_s)\dd s\Big|\bigg]\\
& \leq \EE_\QQ \bigg[\Big(\frac{\dd \PP}{\dd \QQ}\Big)^2\bigg]^{1/2}\, \EE_\QQ \bigg[\Big|\int_0^T f(s,X_s)\dd s\Big|^2\bigg]^{1/2}\\
& \lesssim \bigg(\int_{\RR^d} \EE_{\mu^H}\Big[\Big|\int_0^T f(s,x+\omega_s)\dd s\Big|^2\Big] \mu_0(\dd x) \bigg)^{1/2}
\lesssim \| f\|_{L^{q'}_T B^{-\alpha-2\eps}_{\infty,\infty}}.
\end{align*}
where in the last passage we used the fact that estimate \eqref{eq:averaging-stoch-estimate-1} is uniform in $x\in\RR^d$.

Using the embedding $B^{-\alpha-\eps}_{p,p}\hookrightarrow B^{-\alpha-\eps-d/p}_{\infty,\infty}\hookrightarrow B^{-\alpha-2\eps}_{\infty,\infty}$ for $p<\infty$ big enough, by the duality $(L^{q'}_T B^{-s}_{p,p})^\ast \simeq L^q_T B^s_{p',p'}$ we deduce that $\cL(X_\cdot)\in L^q_T B^{\alpha+\eps}_{p',p'}$. Thus $h:=(I-\Delta)^{\alpha/2} \cL(X_\cdot)\in L^q_T B^\eps_{p',p'} \hookrightarrow L^q_T L^{p'}_x$; in order to conclude, it's enough to show that $h\in L^q_T L^1_x$. Observe that $h\in L^1_{loc}([0,T]\times \RR^d)$ and for any $\varphi\in L^{q'}_T L^\infty_x$ it holds
\begin{align*}
\int_0^T \langle \varphi_s, h_s\rangle \dd s
= \int_0^T \langle (I-\Delta)^{\alpha/2} \varphi_s, \cL(X_s)\rangle \dd s
\lesssim \| (I-\Delta)^{\alpha/2} \varphi\|_{L^{q'}_T B^{-\alpha}_{\infty,\infty}}
\lesssim \| \varphi\|_{L^{q'}_T L^\infty_x};
\end{align*}
the conclusion then follows from an application of Lemma \ref{lem:duality1} from Appendix \ref{app:UsefulLemmas}.
\end{proof}

\begin{prop}\label{prop:regularity-density2}
Let $X,\,b,\,\xi$ be as in Proposition \ref{prop:regularity-density1}. Then $\cL(X_\cdot)\in L^q_T L^p_x$ for all $(q,p)\in (1,\infty)^2$ satisfying
\begin{equation}\label{eq:parameters-regularity-density2}
\frac{1}{q}+\frac{Hd}{p}>Hd.
\end{equation}
If in addition $Hd<1$, then $\cL(X_\cdot)\in L^q_T L^\infty_x$ for all $q\in (1,\infty)$ satisfying $q<(Hd)^{-1}$.
\end{prop}

\begin{proof}
Observe that $(q,p)\in (1,\infty)^2$ satisfy \eqref{eq:parameters-regularity-density2} if and only if the conjugates $(q',p')$ satisfy
\[
\frac{1}{q'} + \frac{H d}{p'}<1.
\]
By \cite[Lemma 6.4]{le2020stochastic} (more precisely equation (6.11) right after the proof therein) and estimates based on Girsanov theorem analogous to the proof of Proposition \ref{prop:regularity-density1}, we deduce that
\[
\bigg|\int_0^T \langle f_s,\cL(X_s)\rangle \dd s\bigg| \leq \EE\bigg[\Big|\int_0^T f(s,X_s)\dd s\Big|\bigg] \lesssim \| f\|_{L^{q'}_T L^{p'}_x};
\]
therefore by duality $\cL(X_\cdot)\in L^q_T L^p_x$ if $(q,p)\in (1,\infty)^2$. Taking $p=\infty,\, q<(Hd)^{-1}$ (correspondingly $p'=1, 1/q'<1-Hd$), we obtain
\[
\bigg|\int_0^T \langle f_s,\cL(X_s)\rangle \dd s\bigg| \lesssim \| f\|_{L^{q'}_T L^1_x} \quad\forall\, f\in L^{q'}_T L^1_x
\]
and so we can conclude by Lemma \ref{lem:duality2} in the Appendix that in this case $\cL(X_\cdot)\in L^q_T L^\infty_x$.
\end{proof}

\section{Proofs of the main results}\label{sec:MainResultsProofs}

We split the proof of Theorem \ref{thm:main_thm1} into sections, which deal respectively with the cases $H>1/2$ and $H\leq 1/2$; the proof of Theorem \ref{thm:main_thm2} is presented in Section \ref{subsec:stability-DDSDE} instead.

We recall to the reader that in this section we will be dealing with DDSDEs of the form
\begin{equation}\label{eq:DDSDE-sec4}
    X_t = \xi + \int_0^t B_s(X_s,\cL(X_s))\,\dd s + W_t \quad \forall\, t\in [0,T]
\end{equation}
with the drift $B$ belonging to either $\cG^{q,\alpha}_p$ or $\cH^{\beta,\alpha}_p$ (cf. Definitions \ref{def: class of functions alpha small}-\ref{def: class of functions alpha big}) depending on the value of $H\in (0,1)$.
The variable $\xi$ is independent of $W$ and with prescribed law $\mu_0\in \cP(\RR^d)$, thus depending on the context we will treat both $(\xi,B)$ and $(\mu_0,B)$ as the data of the problem.

\subsection{The case $H>1/2$}\label{subsec:H>1/2}

In this regime we will always consider drifts $B\in \cH^{\beta,\alpha}_p$ with $\alpha,\beta>0$ and $p\in [1,\infty)$.
In particular here $B:[0,T]\times \RR^d\times \cP_p(\RR^d)\to \RR^d$ is bounded and uniformly continuous in all of its arguments; in this sense, although the concept of solution introduced in Section \ref{sec:recap-SDE} does include the standard one by Remark \ref{rem:SDE-classical-case}, we do not employ it here.

Rather, we will simply say that a tuple $(X,\xi,W)$, defined on a probability space $(\Omega,\cF,\PP)$, such that $\cL_\PP(\xi,W)=\mu_0\otimes\mu^H$, is a solution to \eqref{eq:DDSDE-sec4} if $\cL(X_t)\in \cP_p(\RR^d)$ for all $t\in [0,T]$ and the integral equation \eqref{eq:DDSDE-sec4} holds $\PP$-a.s. The concepts of strong existence, pathwise uniqueness and uniqueness in law immediately carry over from the usual ones for SDEs.

\begin{prop}\label{prop:uniqueness-H>1/2}
Suppose $B\in \cH^{H,\alpha}_{p}$ with
\begin{align*}
    H >\frac{1}{2}, \quad \alpha>1-\frac{1}{2H},\quad p\in [1,\infty).
\end{align*}
Then for any $\mu_0\in\cP_p(\RR^d)$ strong existence, pathwise uniqueness and uniqueness in law hold for the DDSDE \eqref{eq:DDSDE-sec4} with data $(\mu_0,B)$.
\end{prop}

\begin{proof}
We divide the proof in several steps.

\textit{Step 1: weak existence.}
By hypothesis $B:[0,T]\times \RR^d\times \cP_p(\RR^d)\to \RR^d$ is a uniformly continuous, bounded map; existence of weak solutions on $[0,T]$ then follows from \cite[Proposition 3.10]{GalHarMay_21benchmark}.

\textit{Step 2: any weak solution is a strong one.}
Let $X$ be a weak solution of the DDSDE w.r.t. $(\xi,W)$ on a probability space $(\Omega,\cF,\PP)$.
Then setting $\mu_t=\cL(X_t)$, $b^\mu(t,x)=B_t(x,\mu_t)$, $X$ solves the SDE associated to $b^\mu$, which satisfies $|b^\mu(t,x)|\leq \| B\|$ uniformly over $(t,x)$.
As a consequence
\begin{align}\label{eq:MeasureHolderCty_H>1/2}
d_p(\mu_s,\mu_t) \leq \|X_t-X_s\|_{L^p_\Omega} \leq \int_s^t \| b^\mu(r,X_r)\|_{L^p_\Omega} \dd r + \| W_t-W_s\|_{L^p_\Omega} \lesssim_{T,p} (1+\| B\|) |t-s|^H 
\end{align} 
where we repeatedly applied Minkowski's inequality; by assumption
\begin{align*}
|b^\mu(t,x)-b^\mu(s,y)|
& = |B_t(x,\mu_t)-B_s(y,\mu_s)|\\
& \leq \| B\| \,\big(|t-s|^{\alpha H} + |x-y|^\alpha + d_p(\mu_t,\mu_s)^\alpha\big)\\
& \lesssim_{T,p} (1+\| B\|^2)\,\big( |t-s|^{\alpha H} + |x-y|^\alpha\big),
\end{align*}
where we applied \eqref{eq:MeasureHolderCty_H>1/2} to obtain the last inequality. Namely, $b^\mu$ satisfies Assumption \ref{ass:drift-SDE}, implying that strong existence and uniqueness in law holds for the associated SDE; therefore $X$ is adapted to $(\xi,W)$.

\textit{Step 3: reduction to the canonical space.}
As we are dealing with a strong solution $X$, we can regard it as a random variable on the canonical space $(\Omega, \cF, \PP)$ with $\Omega=\RR^d\times C_T$, $\PP=\mu_0\otimes \mu^H$, 
$\cF$ the $\PP$-completion of $\cB(\RR^d\times C_T)$.
%
%
Applying the same reasoning to any pair of weak (thus strong) solutions $X^1,\, X^2$, possibly defined on different probability spaces, we can construct a coupling $(\tilde{X}^1,\tilde{X}^2)$ of solutions defined on the canonical space and w.r.t. the same random variables $(\xi,W)$.
If we show that $\tilde{X}^1\equiv \tilde{X}^2$, then the equality $\cL(X^1)=\cL(X^2)$ follows.

\textit{Step 4: pathwise uniqueness on the canonical space.}
Let us drop the tilde and adopt the notations $\mu^i_t = \cL(X^i_t)$, $b^i(t,x):=B_t(x,\mu^i_t)$.
It follows from the computations of Step 2 that we can find $M\sim 1+\| B\|^2$ such that $b^i$ satisfy Assumption \ref{ass:drift-SDE} with $\| b^i\|_E \leq M$.
We can therefore apply estimate \eqref{eq:SDEGammaNormEstim} for the choice $\tilde q=q=\infty$, together with $X^1_0=X^2_0=\xi$, to find constants $\gamma>1/2$, $C>0$ such that
\[
\EE \Big[\| X^1-X^2\|_{\gamma;[0,\tau]}^p\Big]^{1/p} \leq C \sup_{t\in [0,\tau]} \| b^1(t,\cdot)-b^2(t,\cdot)\|_{B^{\alpha-1}_{\infty,\infty}}\quad \forall\, \tau\in [0,T].
\]
By assumption $B\in \cH^{H,\alpha}_p$, therefore 
\[
\| b^1(t,\cdot)-b^2(t,\cdot)\|_{B^{\alpha-1}_{\infty,\infty}} \leq \| B\|\, d_p(\mu^1_t,\mu^2_t);
\]
combining everything, using again $X^1_0=X^2_0$, for any $\tau\in [0,T]$ it holds
\begin{align*}
\sup_{t\in [0,\tau]} d_p(\mu^1_t,\mu^2_t)
& \leq \tau^\gamma\, \EE\Big[ \| X^1-X^2\|_{\gamma; [0,\tau]}^p\Big]^{1/p}
\leq C \| B\|\, \tau^\gamma \sup_{t\in [0,\tau]} d_p(\mu^1_t,\mu^2_t).
\end{align*}
Choosing $\bar\tau$ small enough so that $C \| B\|\, \bar\tau^\gamma<1$, we conclude that $\mu^1_t=\mu^2_t$ for all $t\in [0,\bar\tau]$ and so that $\EE[ \| X^1-X^2\|_{\gamma;[0,\bar\tau]}]=0$, i.e. $\PP$-a.s. $X^1\equiv X^2$ on $[0,\bar\tau]$. In light of this, choosing now $\tau=2\bar\tau$, going through similar computations we have
\begin{align*}
\sup_{t\in [0,\tau]} d_p(\mu^1_t,\mu^2_t)
= \sup_{t\in [\bar\tau,2\bar\tau]} d_p(\mu^1_t,\mu^2_t)
\leq \bar\tau^\gamma\, \EE\Big[ \| X^1-X^2\|_{\gamma; [\bar\tau,2\bar\tau]}^p\Big]^{1/p}
\leq C \| B\| \bar\tau^\gamma \sup_{t\in [0,\tau]} d_p(\mu^1_t,\mu^2_t)
\end{align*}
implying that the solutions also coincide on $[0,2\bar\tau]$. Iterating the reasoning for $\tau=n\bar\tau$ until we cover $[0,T]$ gives the conclusion.
\end{proof}

\subsection{The case $H\leq 1/2$}\label{subsec:H<1/2}
In this case we can allow the drift to be singular, i.e. take values in $B^\alpha_{\infty,\infty}$ with $\alpha<0$.
We start by defining what we mean by solution to the DDSDE in this case.

\begin{defn}\label{def:solution-singular-DDSDE}
Let $(\Omega,\cF,\PP)$ be a probability space, $(X,\xi,W)$ be a $C_T\times \RR^d\times C_T$-valued random variable defined on it with $\cL_\PP(\xi,W)=\cL(\xi)\otimes \mu^H$; let $B:[0,T]\times \cP_p\to\cS'$ be a measurable map for some $p\in [1,\infty)$.
We say that $X$ is a solution to the DDSDE \eqref{eq:DDSDE-sec4} associated to $(\xi,B)$ if $\mu_t:=\cL_\PP(X_t)\in \cP_p$ for all $t\in [0,T]$ and setting $b^\mu(t,\cdot)=B_t(\mu_t)(\cdot)$, $X$ is a pathwise solution to the SDE
\begin{equation*}
    X_t= \xi + \int_0^t b^\mu(s,X_s)\,\dd s + W_t,
\end{equation*}
associated to $(b^\mu,\xi,W)$ in the sense of Definition \ref{def:SDE-as-NLYE}.
All the concepts of strong solution, pathwise uniqueness and uniqueness in law are similarly readapted from those of Definition \ref{def:pathwise-strong-weak-sol}.
\end{defn}

As before, we will consider both $(\xi,B)$ and $(\mu_0,B)$ to be the data of the problem, depending on whether we are focusing on solutions on a prescribed probability space or on their laws.

Assume now $B\in \cG^{q,\alpha}_p$ with
\begin{equation}\label{eq:parameters-sec4.2}
    H\leq \frac{1}{2}, \quad \alpha>1+\frac{1}{qH}-\frac{1}{2H}, \quad q\in (2,\infty], \quad p\in [1,\infty);
\end{equation}
then to any $\mu_\cdot\in C_T \cP_p$ we can associated a singular drift $b^\mu(t,\cdot)=B(t,\mu_t)(\cdot)$ such that
\[
 \| b^\mu(t,\cdot)\|_{B^\alpha_{\infty,\infty}} \leq h_t,
\]
where $h\in L^q_T$ is the function associated to $B$ from Definition \ref{def: class of functions alpha small}.
Thus $b^\mu$ satisfies Assumption \ref{ass:drift-SDE} and the associated SDE has a unique solution $X$ by Corollary \ref{cor:RndDataWellPosed};
if in addition $\mu_0\in \cP_p$, then by Remark \ref{rem:holder-bound-solution} the map $t\mapsto\cL(X_t)$ belongs to $C_T \cP_p$.

Thus for fixed $\mu_0$, setting $\cI^{\mu_0}(\mu)_\cdot=\cL(X_\cdot)$, we can define a map $\cI^{\mu_0}$ from $C_T \cP_p$ to itself; this map comes with an alternative notion of solution to the DDSDE.

\begin{defn}\label{def:solution-law-DDSDE}
Assume $B\in \cG^{q,\alpha}_p$ with parameters satisfying \eqref{eq:parameters-sec4.2}, $\mu_0\in \cP_p$; we say that a flow of measures $\mu\in C_T \cP_p$ is a solution to the DDSDE  associated to $(\mu_0,B)$ if it satisfies $\cI^{\mu_0}(\mu)=\mu$.
\end{defn}

The next lemma clarifies the relation between Definitions \ref{def:solution-singular-DDSDE} and \ref{def:solution-law-DDSDE}.

\begin{lem}\label{lem:SolDefEquiv}
Let $B\in \cG^{q,\alpha}_p$ with parameters satisfying \eqref{eq:parameters-sec4.2}, $\mu_0\in \cP_p$. The following hold:
\begin{enumerate}[label=\roman*.]
\item \label{it:sol-equiv-1} if $X$ is a weak solution to \eqref{eq:DDSDE-sec4}, then $\mu_t=\cL(X_t)$ is a fixed point for $\cI^{\mu_0}$;
\item \label{it:sol-equiv-2} if $\mu$ is a fixed point for $\cI^{\mu_0}$, then there exists a strong solution $X$ to \eqref{eq:DDSDE-sec4};
\item \label{it:sol-equiv-3} if there exists at most one fixed point for $\cI^{\mu_0}$, then pathwise uniqueness and uniqueness in law hold for \eqref{eq:DDSDE-sec4}.
\end{enumerate}
\end{lem}

\begin{proof}
Point \ref{it:sol-equiv-1} immediately follows from the definitions.
To see Point \ref{it:sol-equiv-2}, assume $\cI^{\mu_0}(\mu)_\cdot=\mu_\cdot$ and set $b^\mu(t,\cdot)= B_t(\mu_t)(\cdot)$; then $b^\mu\in L^q_T B^\alpha_{\infty,\infty}$, so by the results of Section \ref{sec:recap-SDE}, we can construct a strong solution $X$ to the SDE associated to $(\mu_0,b^\mu)$. But then by definition of $\cI^{\mu_0}$ it holds $\cL(X_t)=\mu_t$ and so $X$ solves the DDSDE.
It remains to show Point \ref{it:sol-equiv-3}; assume $X^i$ are two solutions and set $\mu^i_\cdot=\cL(X^i_\cdot)$. Then by Point \ref{it:sol-equiv-2}, $\mu^i$ are both fixed points for $\cI^{\mu_0}$, so $\mu^1=\mu^2$ and $b^{\mu^1}=b^{\mu^2}$. But then $X^i$ both solve the SDE associated to $b^{\mu^1}$, for which uniqueness holds both pathwise and in law, so the conclusion follows.
\end{proof}
It follows from the above that, in order to show strong existence, pathwise uniqueness and uniqueness in law for the DDSDE \eqref{eq:DDSDE-sec4} in the sense of Definition \ref{def:solution-singular-DDSDE}, it's enough to show that there exists exactly one solution $\mu \in C_T \cP_p$ in the sense of Definition \ref{def:solution-law-DDSDE}.

\begin{prop}\label{prop:uniqueness-H<1/2}
Let $B\in\cG^{q,\alpha}_p$ with parameters satisfying \eqref{eq:parameters-sec4.2};
then for any $\mu_0\in\cP_p(\RR^d)$ strong existence, pathwise uniqueness and uniqueness in law hold for the DDSDE \eqref{eq:DDSDE-sec4} associated to $(\mu_0,B)$.
\end{prop}

\begin{proof}
Define the map $\cI^{\mu_0}:C_T\cP_p\to C_T\cP_p$ associated to $(\mu_0,B)$ as before; in order to show that there exists exactly one fixed point to $\cI^{\mu_0}$, it's enough to establish its contractivity.

Given $\mu^i \in C_T\cP_p$, $i=1,2$, set $b^i: = b^{\mu^i} = B(t,\mu^i_t)$; denote by $X^i$ two solutions, defined on the same probability space and with respect to the same data $(\xi,W)$, to the SDEs associated to $(\xi,b^i)$, where $\cL(\xi)=\mu_0$.
By definition of $\cG^{q,\alpha}_p$, there exists $h\in L^q_T$, such that for any $\tau \in (0,T]$ we have
$$\|b^1_t-b^2_t\|_{B^{\alpha-1}_{\infty,\infty}}\leq h_t\, \sup_{t \in [0,\tau]} d_p(\mu^1_t,\mu^2_t).$$
Applying Corollary \ref{cor:RndDataWellPosed}, using the fact that $X^1_0=X^2_0=\xi$, we can find $\gamma>1/2$ and $C>0$ such that for any $\tau \in (0,T]$, we have
\begin{align*}
    \sup_{t\in [0,\tau]} d_p(\cI^{\mu_0}(\mu^1)_t,\cI^{\mu_0}(\mu^2)_t)
    & \leq\sup_{t\in [0,\tau]} \EE_{\PP}[|X^1_t-X^2_t|^p]^{\frac{1}{p}}
    \leq \tau^\gamma\, \EE_{\PP}\big[\|X^1-X^2\|_{\gamma;[0,\tau]}^p\big]^{\frac{1}{p}}\\
    & \leq C\tau^\gamma \left( \int_0^\tau \|b^1_t-b^2_t\|_{B^{\alpha-1}_{\infty,\infty}}^{q}\dd t\right)^{\frac{1}{q}}
    \leq C \|h\|_{L^q_T}\, \tau^\gamma\, \sup_{t\in[0,\tau]}d_p(\mu^1_t,\mu^2_t).
\end{align*}
Choosing $\bar\tau>0$ sufficiently small such that $C\|h\|_{L^q_T}\,\tau^\gamma<1$, we find that $\cI^{\mu_0}$ is a contraction from $C([0,\bar\tau];\cP_p)$ to itself, so therein there exists a unique fixed point $\bar{\mu} = \cI^{\mu_0}(\bar{\mu})$; it remains to show we can extend uniquely this fixed point to the whole interval $[0,T]$.

To do this, the classical argument for SDEs would require to restart the equation at $t=\bar\tau$; however we can't perform this, as the fractional Brownian motion is not a Markov process. We can exploit the fact that $\bar\tau$ only depends on $C\|h\|_{L^q_T}$ and not the history of the paths $X^i$ nor $\mu^i$ to give the following alternative reasoning.

Given $\bar\tau$, $\bar{\mu}\in C([0,\bar\tau];\cP_p)$ as above, consider $E := \{ \mu_{\,\cdot\,} \in C([0,2\bar\tau];\cP_p) \,:\, \mu|_{[0,\tau]} =\bar{\mu} \}$, which is a closed subset of $C([0,2\bar\tau];\cP_p)$ and thus a complete metric space with the same norm.
Since $\bar{\mu}$ is a fixed point on $[0,\bar \tau]$, $\cI^{\mu_0}$ leaves $E$ invariant; for any $\mu^i \in E$, arguing as above it holds
\begin{align*}
    \sup_{t \in [0,2\bar\tau]} d_p(\cI^{\mu_0}(\mu^1)_t,\cI^{\mu_0}(\mu^2)_t)
    = \sup_{t\in [\bar\tau,2\bar\tau]} d_p(\cI^{\mu_0}(\mu^1)_t,\cI^{\mu_0}(\mu^2)_t)
    &\leq \sup_{t \in [\bar\tau,2\bar\tau]} \bar\tau^\gamma \, \EE_{\PP}\big[\|X^1-X^2\|_{\gamma;[\bar\tau;2\bar\tau]}^p\big]^{\frac{1}{p}}\\
    &\leq \bar\tau^\gamma C\|h\|_{L^q_T}\, \sup_{t \in [0,2\bar\tau]}d_p(\mu^1_t,\mu^2_t).
\end{align*}
It follows that $\cI^{\mu_0}$ is a contraction on $E$ and admits a unique fixed point on it, which is necessarily the only possible extension of $\bar\mu$ on $[0,2\bar\tau]$. Repeating the argument on $[0,n\bar\tau]$ as many times as necessary to cover $[0,T]$ concludes the proof.
\end{proof}

\subsection{Stability estimates for DDSDEs}\label{subsec:stability-DDSDE}

The purpose of this section is to provide the proof of Theorem \ref{thm:main_thm2}, which loosely speaking establishes Lipschitz dependence of the solutions $\mu^i\in C_T \cP_p$ in terms of the data $(\mu^i_0,B^i)$ for $i=1,2$.

We assume that we are given drifts $B^i$ belonging to $\cH^{H,\alpha}_p$ for parameters satisfying \eqref{eq:parameters-H>1/2} when $H>1/2$, respectively $B^i\in \cG^{q,\alpha}_p$ for parameters satisfying \eqref{eq:parameters-H<1/2} when $H\leq 1/2$; in both cases we denote the optimal constants by $\| B^i\|$.
Given $\mu_0^i\in \cP_p$, we denote by $\mu^i\in C_T \cP_p$ the unique solutions associated to $(\mu^i_0,B^i)$, whose existence is granted by Theorem \ref{thm:main_thm1}.

Finally, for $\alpha,\,q$ given as above, let us recall the notation introduced in Theorem \ref{thm:main_thm2}:
\begin{equation*}
    \| B^1-B^2\|_{\infty}:=\sup_{(t,\nu)\in [0,T]\times \cP_p} \| B^1_t(\nu)-B^2_t(\nu)\|_{B^{\alpha-1}_{\infty,\infty}}
\end{equation*}
and 
\begin{equation*}
\| B^1-B^2\|_{q,\infty}:=\bigg( \int_0^T \sup_{\nu\in \cP_p} \| B^1_t(\nu)-B^2_t(\nu)\|_{B^{\alpha-1}_{\infty,\infty}}^{q} \dd t\bigg)^{1/q}.
\end{equation*}

\begin{proof}[Proof of Theorem \ref{thm:main_thm2}]
Let $\mu^i$ be the solutions as above and set $b^i_t:=B^i(t,\mu^i_t)$.
Recall from the proofs of Propositions \ref{prop:uniqueness-H>1/2} and \ref{prop:uniqueness-H<1/2} that if $\|B^i\|\leq M$, then $\|b^i\|_E \leq C(M)$, $E$ being suitable spaces for which Assumption \ref{ass:drift-SDE} is met; so we are in a position to apply estimates from Section \ref{subsec:stability}.
First observe that, by addition and subtraction of $B^1_t(\mu^2)$, we have
\begin{equation}\label{eq:bound diff b}
    \| b^1_t-b^2_t\|_{B^{\alpha-1}_{\infty,\infty}}
    \leq \| B^1_t(\mu^1_t)-B^1_t(\mu^2_t)\|_{B^{\alpha-1}_{\infty,\infty}}
    + \sup_{\nu\in \cP_p} \| B^1_t(\nu)-B^2_t(\nu)\|_{B^{\alpha-1}_{\infty,\infty}}.
\end{equation}
The argument slightly differs in the $H>\frac{1}{2}$ and $H\leq \frac{1}{2}$ cases, so we will handle them separately.

We begin with $H>\frac{1}{2}$. Let us choose $\tilde{q}<\infty$ big enough so that $(\alpha,\tilde q,H) $  satisfies \eqref{eq:good parameters}; then we can apply estimate \eqref{eq:SDE-estim-wasserstein} to obtain
\begin{equation}\label{eq:sup ineq}
\sup_{t\in [0,\tau]} d_p(\mu^1_t,\mu^2_t)^{\tilde q} \lesssim_{M,\tilde q,T} d_p(\mu^1_0,\mu^2_0)^{\tilde q} + \left(\int_0^\tau \| b^1_s-b^2_s\|_{B^{\alpha-1}_{\infty,\infty}}^{\tilde q} \dd s\right);
\end{equation}
on the other hand, by estimate \eqref{eq:bound diff b} and the assumption $B^1\in \cH^{H,\alpha}_p$ with $\|B^1\|\leq M$, it holds
\begin{equation*}
    \| b^1_s-b^2_s\|_{B^{\alpha-1}_{\infty,\infty}} \leq M \sup_{r\in [0,s]} d_p(\mu^1_r,\mu^2_r)+ \| B^1-B^2\|_{\infty}.
\end{equation*}
Putting everything together, setting $f_t:=\sup_{s\in [0,t]} d_p(\mu^1_s,\mu^2_s)^{\tilde{q}}$, we obtain
\begin{equation*}
f_t \lesssim_{\tilde q} f_0 + \int_0^t M^{\tilde q} f_s \dd s + T \| B^1-B^2\|_{\infty}^{\tilde q} \quad \forall\, t\in [0,T];
\end{equation*}
applying Gr\"onwall to $f$ and taking the power $1/\tilde q$ no both sides readily gives
\begin{equation*}
    \sup_{t\in [0,T]} d_p(\mu^1_t,\mu^2_t) \lesssim d_p(\mu^1_0,\mu^2_0) + \| B^1-B^2\|_\infty
\end{equation*}
which is exactly the desired estimate \eqref{eq:stability-estim-DDSDE}.

Suppose now $X^i$ are solutions defined on the same probability space, then combining estimate \eqref{eq:SDEGammaNormEstim} with the ones above we find
\begin{align*}
    \EE[\| X^1-X^2\|_\gamma^p]^{1/p}
    & \lesssim \| X^1_0-X^2_0\|_{L^p_\Omega} + \sup_{t\in [0,T]} \| b^1_t-b^2_t\|\\
    & \lesssim_M \| X^1_0-X^2_0\|_{L^p_\Omega} + \sup_{t\in [0,T]} d_p(\mu^1_t,\mu^2_t) + \| B^1-B^2\|_\infty\\
    & \lesssim \| X^1_0-X^2_0\|_{L^p_\Omega} + d_p(\mu^1_0,\mu^2_0) + \| B^1-B^2\|_\infty
\end{align*}
and the conclusion readily follows from $d_p(\mu^1_0,\mu^2_0)\leq \| X^1_0-X^2_0\|_{L^p_\Omega}$.

We now move on to the case $H\leq \frac{1}{2}$; for $q=\infty$ the proof is the same as above, so we can assume w.l.o.g. $q<\infty$ here.
For $B^i\in \cG^{q,\alpha}_p$, it follows again by \eqref{eq:bound diff b} that 
\begin{equation}\label{eq:b small ineq}
    \| b^1_t-b^2_t\|_{B^{\alpha-1}_{\infty,\infty}}
    \leq h^1_t \sup_{r\in [0,t]} d_p(\mu^1_r,\mu^2_r)+\sup_{\nu\in \cP_p}\| B^1_t(\nu)-B^2_t(\nu)\|_{B^{\alpha-1}_{\infty,\infty}},
\end{equation}
where we recall that $h^i\in L^q_T$ are the functions associated to $B^i$ given in Definition \ref{def: class of functions alpha small}.
Following the same strategy as before, by \eqref{eq:b small ineq} and \eqref{eq:sup ineq}, $f_t:=\sup_{s\in [0,t]} d_p(\mu^1_s,\mu^2_s)^q$ satisfies
\begin{equation*}
    f_t \leq f_0+\int_0^t |h^1_s|^q f_s\dd s + \| B^1-B^2\|_{q,\infty} \quad \forall t\in [0,T];
\end{equation*}
by this inequality and the assumption $\| h^1\|_{L^q_T} = \| B^1\| \leq M$, we conclude again by Gr\"onwall's inequality that
\begin{align*}
    \sup_{t\in [0,T]} d_p(\mu^1_t,\mu^2_t) \lesssim_M d_p(\mu^1_0,\mu^2_0) + \| B^1-B^2\|_{q,\infty}
\end{align*}
which gives estimate \eqref{eq:stability small}. The statement for $\EE[\| X^1-X^2\|^p_\gamma]$ now follows exactly as in the case $H>1/2$.
\end{proof}

We conclude this section with the application of Theorem \ref{thm:main_thm2} to a particularly relevant case.

\begin{ex}\label{ex:stability-estimate}
Let $H\in (0,1)$, $\alpha>1-1/(2H)$ and set
\begin{equation*}
    E = \begin{cases}
    C^{\alpha H}_T C^0_x\cap C^0_T C^\alpha_x \quad & \text{for } H>\frac{1}{2}\\
    L^\infty_T B^\alpha_{\infty,\infty} & \text{for } H\leq\frac{1}{2}
    \end{cases}.
\end{equation*}
Let $B^i$ be of the form $B^i(t,\mu)=f^i_t + g^i_t \ast \mu$ for $f^i,g^i\in E$ with $\| f^i\|_E,\|g^i\|_E\leq M$. Then $B^i$ satisfy the assumptions of Theorem \ref{thm:main_thm2} in $\cP_p$ for any $p\in [1,\infty)$ and estimate \eqref{eq:stability-estim-DDSDE} becomes
\[
\sup_{t\in [0,T]} d_p(\mu^1_t,\mu^2_t) \lesssim d_p(\mu^1_0,\mu^2_0) + \|f^1-f^2 \|_{L^\infty_T B^{\alpha-1}_{\infty,\infty}} + \| g^1-g^2\|_{L^\infty_T B^{\alpha-1}_{\infty,\infty}}.
\]
\end{ex}


\section{Refined results in the convolutional case}\label{sec:refined-convolution}

In this section we focus on the case of DDSDEs with convolutional structure, namely
\begin{equation}\label{eq:ConvMckeanSDE}
X_t = \xi + \int_0^t (b_s\ast \mu_s)(X_s)\,\mathd s+W_t,\quad \mu_t=\mathcal{L}(X_t)\quad \forall\, t\in [0,T].
\end{equation}
They correspond to the case $B_t(\mu)=b_t\ast\mu$ and can therefore be solved under suitable assumptions on $b$ (e.g. $b\in E$ as in Example \ref{ex:stability-estimate}). Due to their specific structure however, as soon as the associated solution $X$ has a regular law $\mu$, its regularity immediately transfers to the drift $b^\mu_t=b_t\ast \mu_t$, as the next simple lemma shows.

\begin{lem}\label{lem:regularity-convolution}
Let $H\in (0,1)$, $b\in B^\alpha_{\infty,\infty}$ for $\alpha>1-1/(2H)$, $\mu_0\in \cP_1$ and $X$ denote the unique solution to the DDSDE \eqref{eq:ConvMckeanSDE} with $\cL(\xi,W)=\mu_0\otimes \mu^H$.
Then $X$ also solves an SDE with drift $b^\mu$ which belongs to $L^1_T C^1_x$.
\end{lem}

\begin{proof}
Let $X$ be the aforementioned solution, then by the proof of Theorem \ref{thm:main_thm1} we know that it is solves an SDE with drift $b^\mu$ satisfying Assumption \ref{ass:drift-SDE};
applying Proposition \ref{prop:regularity-density1} for the choice $q=1$ in \eqref{eq:parameters-regularity-density1}, we deduce that $\mu\in L^1_T B^{\tilde \alpha}_{1,1}$ for any $\tilde \alpha <1/(2H)$. 
Therefore by the hypothesis and Young's inequality it holds $b^\mu=b_\cdot\ast \mu_\cdot\in L^1_T B^{\alpha+\tilde \alpha}_{\infty,\infty}$ for some $\alpha>1-1/(2H)$ and all $\tilde \alpha<1/(2H)$; choosing $\tilde\alpha$ appropriately gives the conclusion.
%
\end{proof}

\begin{rem}
Up to technicalities, the proof readapts to the case of time-dependent drifts $B_t(\mu)=b_t\ast \mu$ with $b$ satisfying Assumption \ref{ass:drift-SDE}, with the same conclusion that $b^\mu\in L^1_T C^1_x$.
\end{rem}

Lemma \ref{lem:regularity-convolution} shows that in this setting the effective drift $b^\mu$ is much more regular than the original $b$, to the point that the SDE associated to $b^\mu$ can be solved classically.
However in order to give meaning to the DDSDE, it suffices to know that $b^\mu$ satisfies the weaker Assumption \ref{ass:drift-SDE};
for this reason we expect the criteria coming from Theorem \ref{thm:main_thm1} to be suboptimal for convolutional DDSDEs \eqref{eq:ConvMckeanSDE}, as they don't take in account the different regularity of $b$ and $b^\mu$.

A partial improvement of those results is given by Theorems \ref{thm:main_thm3} and \ref{thm:main_thm4}, whose proofs are presented respectively in Sections \ref{subsec:refined-divergence} and \ref{subsec:refined-integrable}.
Before moving further, let us define rigorously what we mean by solutions here, although the concept is very similar to that of Definition \ref{def:solution-singular-DDSDE}.

\begin{defn}\label{def:solution-convolution-DDSDE}

Fix $H\in (0,1)$; let $(\Omega,\cF,\PP)$ be a probability space, $(X,\xi,W)$ be a $C_T\times \RR^d\times C_T$-valued random variable defined on it with $\cL_\PP(\xi,W)=\mu_0\otimes \mu^H$ and $b$ be a distributional drift.
We say that $X$ is a solution to the DDSDE \eqref{eq:ConvMckeanSDE} associated to $(\mu_0,b)$ if setting $\mu_t:=\cL_\PP(X_t)$, $b^\mu_t:=b_t\ast \mu_t$, $X$ satisfies the SDE
\begin{equation*}
    X_t= \xi + \int_0^t b^\mu_s(X_s)\,\dd s + W_t \quad \forall\, t\in [0,T],
\end{equation*}
where we additionally require that either:
\begin{enumerate}[label=\roman*.]
    \item $b^\mu$ satisfies Assumption \ref{ass:drift-SDE} and the SDE is interpreted in the sense of Definition \ref{def:SDE-as-NLYE}, or
    \item $b^\mu \in L^1_T C^0_x$ and the SDE is interpreted in the standard integral sense.
\end{enumerate}
All the concepts of weak solution, strong solution, pathwise uniqueness and uniqueness in law are readapted similarly.
\end{defn}

A major role in the proofs of Theorems \ref{thm:main_thm3} and \ref{thm:main_thm4} is given by the following conditional uniqueness result.

\begin{prop}\label{prop:RefinedAPrioriUniqueness}
Let $H\in (0,1)$, $p\in [1,\infty)$, $p'$ its conjugate exponent; let $b$ be a distributional drift satisfying one of the following conditions:
\begin{itemize}
    \item[i.] If $H>1/2$, then $b\in C^{\alpha H}_T L^p_x\cap b\in C^0_T B^\alpha_{p,p}$ with 
    $\alpha>1-\frac{1}{2H}$.
    \item[ii.] If $H\leq 1/2$ then $b \in L^q_T B^\alpha_{p,p}$,
    with
    $\alpha>1-\frac{1}{2H}+\frac{1}{Hq}$.
\end{itemize}
Assume furthermore that for a given $\mu_0\in L^{p'}_x$ there exists a weak solution $X$ to the DDSDE \eqref{eq:ConvMckeanSDE} associated to $(\mu_0,b)$,
satisfying
\begin{equation}\label{eq:integrability-law-sec5}
    \sup_{t\in [0,T]} \| \cL(X_\cdot)\|_{L^{p'}_x} <\infty.
\end{equation}
Then $X$ is a strong solution; moreover it is the unique one (both pathwise and in law) in the class of solutions satisfying condition \eqref{eq:integrability-law-sec5}
\end{prop}

\begin{proof}
We handle the cases $H\leq 1/2$ and $H>1/2$ slightly differently.

\textit{The case $H\leq 1/2$.} First observe that, if $X$  satisfies \eqref{eq:integrability-law-sec5}, then by Young's inequality $b^\mu_\cdot=b_\cdot\ast\cL(X_\cdot)\in L^q_T B^\alpha_{\infty,\infty}$; in particular $b^\mu$ satisfies Assumption \ref{ass:drift-SDE}, Definition \ref{def:solution-convolution-DDSDE} is meaningful and $X$ is necessarily a strong solution.

Now let $X^i$, $i=1,2$, be two solutions to \eqref{eq:ConvMckeanSDE} satisfying \eqref{eq:integrability-law-sec5};
as they are both strong solutions, by the usual arguments, we can assume them to be defined on the same probability space, w.r.t. the same $(\xi,W)$, and we only need to check that $X^1=X^2$ $\PP$-a.s. Moreover thanks to the strict inequality $\alpha>1-\frac{1}{2H}+\frac{1}{Hq}$ here we can assume w.l.o.g. $q<\infty$. 

For $i=1,2$, set $\mu^i_t = \cL(X^i_t)$;
%
%
as $b \ast \mu^i$ both satisfy Assumption \ref{ass:drift-SDE}, we may apply Corollary \ref{cor:RndDataWellPosed} to find
%
\begin{align*}
\sup_{t \in [0,T]} d_{r'}(\mu^1_t,\mu^2_t)^{r'} &\leq \sup_{t\in [0,T]}\EE[| X^1_t-X^2_t|^{r'}]
 \leq \EE[\| X^1-X^2\|_\gamma^{r'}]
\lesssim \| b\ast( \mu^1-\mu^2)\|_{L^q_T B^{\alpha-1}_{\infty,\infty}}^{r'}\\
&\lesssim \| b\|_{L^q_T B^\alpha_{p,p}}^{r'} (\| \mu^1\|_{L^\infty_T L^{p'}_x}^{r'} + \| \mu^2\|_{L^\infty_T L^q_x}^{r'})<\infty.
\end{align*}
%
In particular the quantity $d_{r'}(\mu^1_t,\mu^2_t)$ is finite for any $r'\in [1,\infty)$ and any $t\in [0,T]$.

We now wish to apply Corollary \ref{cor: maximal function2} from Appendix \ref{app:UsefulLemmas} to obtain better control on the difference of the drifts $b\ast \mu^1-b\ast \mu^2$.
To do so observe that, under our assumptions on the parameters $(\alpha,q,p)$, we can find new parameters $(s,r)\in (1,\infty)^2$ with $s$ large and $r$ close to $1$ such that
\begin{equation}\label{eq:proof_new_coeff}
   \alpha - \frac{d}{s} > 1-\frac{1}{2H} + \frac{1}{Hq}, \quad 1+\frac{r}{s}\leq \frac{r}{p}+\frac{1}{p'}.
\end{equation}
For this choice, set $\tilde\alpha := \alpha- d/s$;
by construction the parameters $(p,p',s,r)$ satisfy the assumptions of Corollary \ref{cor: maximal function2} from Appendix \ref{app:UsefulLemmas}; its application, together with standard Besov embeddings, yields
\begin{align*}
    \| b_t\ast (\mu^1_t-\mu^2_t) \|_{B^{\tilde\alpha-1}_{\infty,\infty}}
    & \lesssim \| b_t\ast (\mu^1_t-\mu^2_t) \|_{B^{\alpha-1}_{s,s}}\\
    & \lesssim \| b_t\|_{B^\alpha_{p,p}} (\| \mu^1_t\|_{L^{p'}}^{1/r} + \| \mu^2_t\|_{L^{p'}}^{1/r}) d_{r'}(\mu^1_t,\mu^2_t).
\end{align*}
Now since under $\eqref{eq:proof_new_coeff}$ the triple $(\tilde{\alpha},q,H)$ also satisfies \eqref{eq:good parameters}, we can again apply estimate \eqref{eq:SDEGammaNormEstim} from Corollary \ref{cor:RndDataWellPosed} to find
\begin{align*}
    d_{r'}(\mu^1_t,\mu^2_t)^q
    \leq \|X^1_t-X^2_t\|_{L^{r'}_\Omega}^q
    \lesssim \int_0^t \| b_u\ast (\mu^1_u-\mu^2_u)\|_{B^{\tilde\alpha-1}_{\infty,\infty}}^q\, \dd u
    \lesssim  \int_0^t \| b_u\|^q_{B^\alpha_{p,p}} d_{r'}(\mu^1_u,\mu^2_u)^q \dd u.
\end{align*}
Applying Gr\"onwall's lemma we conclude that $d_{r'}(\mu^1_t,\mu^2_t)=0$ and so $\mu^1_t=\mu^2_t$ for all $t\in [0,T]$. Thus, $X^i$ are solutions solutions to the same SDE and therefore $X^1_\cdot=X^2_\cdot$ $\PP$-a.s.

\textit{The case $H>1/2$.}
We argue essentially in the same way, only this time checking that $X$ is a strong solution starting from the available information on $b$ and $\cL(X_\cdot)$ is less straightforward.

First observe that Young's inequality still provides $b^\mu\in C^0_T C^\alpha_x$, so that by Definition \ref{def:solution-convolution-DDSDE} the DDSDE is meaningful in the classical integral sense.
In order to check Assumption \ref{ass:drift-SDE} for $b^\mu$ (which implies $X$ being strong), it remains to show that $b\ast \mu\in C^{\tilde \alpha H}_T C^0_x$ for some $\tilde \alpha$ such that
\begin{align*}
    1-\frac{1}{2H} < \tilde\alpha\leq \alpha.
\end{align*}
%
%
%
By addition and subtraction, $b^\mu_t -b^\mu_s = (b_t-b_s)\ast \mu_t + b_s\ast(\mu_t-\mu_s)$; by the hypothesis on $b,\,\mu,$ we can estimate the first term by
\begin{equation}\label{eq:bTimeIncrement}
    \|(b_t-b_s)\ast \mu_s\|_{C^0_x}
    \leq \|b_t-b_s\|_{L^p_x}\|\mu\|_{L^\infty_T L^{p'}_x} \lesssim |t-s|^{\alpha H}.
\end{equation}
Since $X$ is solution to the SDE $X_t=\xi + \int_0^t b^\mu_s(X_s)\mathd s+ W_t$, for any $r'\in (1,\infty)$ it holds
\begin{equation*}
    d_{r'}(\mu_t,\mu_s)
    \leq \|X_t-X_s\|_{L^{r'}_\Omega}
    \leq \| b\ast \mu^i\|_{L^\infty_T C^0_x}\, |t-s| +\|W_t-W_s\|_{L^{r'}_\Omega}
    \lesssim_{r',T} |t-s|^H.
\end{equation*}
Now similarly to the case $H\leq 1/2$, choose $(s,r) \in (1,\infty)^2$ such that
\begin{equation*}
    \tilde \alpha:=\alpha -\frac{d}{s} > 1-\frac{1}{2H}, \quad 1+\frac{r}{s}\leq \frac{r}{p}+ \frac{1}{p'};
\end{equation*}
%
%
Applying Corollary \ref{cor: maximal function2} and the previous estimate for $d_{r'}(\mu_t,\mu_s)$, we then find
\begin{align*}
   \| b_t\ast (\mu^i_t- \mu^i_s)\|_{B^{\tilde\alpha-1}_{\infty,\infty}} \lesssim d_{r'}(\mu_t,\mu_s) \lesssim |t-s|^H.
\end{align*}
On the other hand, since $b\in C^0_T B^\alpha_{p,p}$ and $\mu \in L^\infty_T L^{p'}_x$, by Young's inequality, we have $\|b_t \ast (\mu^i_t - \mu^i_s)\|_{B^{\alpha}_{\infty,\infty}}\lesssim 1$.
%
We can now interpolate between the two estimates: choose $\theta=\tilde\alpha\in (0,1)$, so that
$\theta (\tilde\alpha -1) + (1-\theta)\alpha = (1-\tilde\alpha)(\alpha-\tilde\alpha)>0$, then by the embedding $B^\eps_{\infty,\infty}\hookrightarrow C^0_x$ for $\eps>0$, we obtain
\begin{align*}
    \| b_t\ast (\mu^i_t-\mu^i_s)\|_{C^0_x}
    \lesssim  \| b\ast (\mu^i_t-\mu^i_s)\|_{B^\alpha_{\infty,\infty}}^{1-\tilde\alpha} \,\| b_t\ast (\mu^i_t-\mu^i_s)\|_{B^{\tilde\alpha-1}_{\infty,\infty}}^{\tilde\alpha}
    \lesssim  |t-s|^{\tilde\alpha H}.
\end{align*}
So we conclude that $b\ast \mu^i\in C^0_T C^{\tilde\alpha}_x\cap C^{\tilde{\alpha} H}_T C^0_x$, where by construction $\tilde\alpha>1-1/(2H)$.

The second part of the argument, concerning the comparison of two solutions $X^i$ satisfying \eqref{eq:integrability-law-sec5}, now proceeds identically as in the case $H\leq 1/2$.
\end{proof}


\subsection{Distributional kernels with bounded divergence}\label{subsec:refined-divergence}

Proposition \ref{prop:RefinedAPrioriUniqueness} reduces the problem of uniqueness of solutions (in a suitable class) to that of establishing their regularity, in the sense of equation \eqref{eq:integrability-law-sec5}.

One classical way to show that the condition $\mu_0\in L^{p'}_x$ is propagated at positive times, which has been exploited systematically after \cite{diperna1989ordinary}, is to impose boundedness of $\div b$; in the setting of DDSDEs with general additive noise and regular drift $b$, an analogous statement can be found in \cite[Proposition 4.3]{GalHarMay_21benchmark}.

\begin{prop}\label{prop:DivStrongExist}
Let $H\in (0,1)$, $q \in (2,\infty]$, $p\in [1,\infty)$, $p'$ its conjugate exponent.
Let $b$ be a distributional drift such that $\div b \in L^1_{T} L^\infty_x$ and either:
\begin{enumerate}[label=\roman*.]
    \item[i.] If $H>1/2$, then $b\in C^{\alpha H}_T L^p_x\cap  C^0_T B^\alpha_{p,p}$ for some $\alpha>1-\frac{1}{2H}$.
    \item[ii.] If $H\leq 1/2$, then $b \in L^q_T B^\alpha_{p,p}$ for some $\alpha>1-\frac{1}{2H}+\frac{1}{Hq}$.
\end{enumerate}
Then for any $\mu_0\in L^{p'}_x$ there exists a strong solution to the DDSDE \eqref{eq:ConvMckeanSDE} associated to $(\mu_0,b)$, which moreover satisfies
$\mathcal{L}(X_\cdot)\in L^\infty_T L^{p'}_x$.
\end{prop}

\begin{proof}
We start by dealing with the case $H\leq 1/2$; at the end of the proof we explain how the reasoning needs to be modified for $H>1/2$.

\textit{The case $H\leq 1/2$.}
In this case we can assume w.l.o.g. $q<\infty$;
recall that if $f^n$ is a bounded sequence in $L^q_T B^\alpha_{\infty,\infty}$, for $(\alpha,q)$ satisfying Assumption \ref{ass:drift-SDE} such that $f^n\to f$ in $L^q_T B^{\alpha-1}_{\infty,\infty}$, then by Corollary \ref{cor:RndDataWellPosed} the associated solutions
\[
X^n_t = \xi + \int_0^t f^n(X^n_s)\,\mathd s + W_t
\]
converge to the unique strong solution $X$ of the SDE associated to $(\xi,W,f)$ (we can assume $\{X^n\}_{n\geq 1}$ and $X$ to be defined on the same probability space for the same $(\xi,W)$).

Given $b$ as in the hypothesis, consider a sequence of smooth, bounded functions $b^n$ such that $b^n\to b$ in $L^q_T B^\alpha_{p,p}$ with $\| \div b^n\|_{L^1_T L^\infty_x} \leq \| \div b\|_{L^1_T L^\infty_x}$;
let $X^n$ be the solutions to
\begin{equation*}
X^n_t = \xi + \int_0^t b^n\ast\mathcal{L}(X^n_s)(X^n_s)\mathd s + W_t,
\end{equation*}
whose existence is granted by classical results (see e.g. \cite[Theorem 7]{coghi2020pathwise}) and set $\mu^n_t = \cL(X^n_t)$.
By \cite[Proposition 4.3]{GalHarMay_21benchmark}, there exists $C=C(\| \div b\|_{L^1_T L^\infty_x})>0$ such that
\begin{align*}
   \sup_{n\in\NN} \sup_{t\in [0,T]} \| \mu^n_t\|_{L^{p'}_x} \leq C \| \mu_0\|_{L^{p'}_x}<\infty.
\end{align*}
As a consequence, each $X^n$ solves an SDE with drift $f^n=b^n\ast\mu^n$ satisfying
\[
\sup_{n\in\NN} \|f^n \|_{L^q_T B^\alpha_{\infty,\infty}}
\lesssim  \sup_{n\in\NN}\|b^n\|_{L^q_T B^\alpha_{p,p}}\, \sup_{n\in\NN}\sup_{t\in [0,T]} \| \mu^n_t\|_{L^{p'}_x}
\lesssim C \| b\|_{L^q_T B^\alpha_{p,p}} \| \mu_0\|_{L^{p'}_x}<\infty.
\]
In turn this implies by Remark \ref{rem:holder-bound-solution} that for any fixed $\eps>0$ we have the uniform estimate $\sup_n \EE[\llbracket X^n\rrbracket_{C^{H-\eps}}]<\infty$; since moreover $X^n_0=\xi$ for all $n\in\NN$, we can conclude by Ascoli--Arzel\`a that the sequence $\{X^n\}_{n\geq 1}$ is tight in $C_T$.
We can then extract a (not relabelled) subsequence such that $\cL(X^n)$ converge weakly to some $\mu\in \cP(C_T)$; consequently $\mu^n_t\rightharpoonup \mu_t$ in $\cP(\RR^d)$ for any $t\in [0,T]$, where $\mu_t = e_t \sharp \mu$ and $e_t:C_T\rightarrow \RR^d$ is the evaluation map. 
It follows from the uniform estimates that $\| \mu_t\|_{L^{p'}_x} \leq C \| \mu_0\|_{L^{p'}}$ as well.

We claim that the drifts $f^n_t=b^n_t\ast \mu^n_t$ converge to $f_t:=b_t\ast \mu_t$ in $L^q_T B^{\alpha-1}_{\infty,\infty}$.
Once this is shown, by the initial observation the solutions $X^n$ must converge to the unique solution $X$ associated to $(\xi,W,f)$; then it must hold $\cL(X_t)=\mu_t$, $f_t=b_t\ast \cL(X_t)$ and so we can conclude that $X$ is a solution to \eqref{eq:ConvMckeanSDE} with the desired regularity.

It remains to show the claim;
to this end, we set
\[
f^n_t-f_t
=b_t^n\ast (\mu^n_t-\mu_t)+(b_t^n-b_t)\ast\mu_t
=: g^n_t + h^n_t.\]
By Corollary \ref{cor:BesovConvolve} in Appendix \ref{app:UsefulLemmas}, $g^n_t\to 0$ in $B^{\tilde \alpha}_{\infty,\infty}$ for all $\tilde \alpha<\alpha$ and a.e. $t\in [0,T]$;
the bound
$|g^n_t|\leq \|b_t\|_{B^\alpha_{p,p}} (\|\mu^n_t\|_{L^{p'}_x}+\|\mu_t\|_{L^{p'}_x})$ and dominated convergence imply that $g^n \rightarrow g$ in $L^q_T B^{\tilde \alpha}_{\infty,\infty}$ for all $\tilde \alpha<\alpha$.
For $h^n$ we have the estimate
\[
\lim_{n\to\infty} \|h^n\|_{L^q_T B^\alpha_{\infty.\infty}} \leq  \sup_{t\in [0,T]} \| \mu_t\|_{L^{p'}_x}\,\lim_{n\to\infty}\| b^n-b\|_{L^q_T B^\alpha_{p,p}} = 0.
\]
Hence we have shown the claim and thus the conclusion in this case.


\textit{The case $H>1/2$.} 
As in the proof of Proposition \ref{prop:RefinedAPrioriUniqueness}, in this regime $\cL(X_\cdot)\in L^\infty_T L^{p'}_x$ is not enough to deduce straightaway that $b\ast \cL(X_\cdot)$ satisfies Assumption \ref{ass:drift-SDE}; however up to technical details, the proof is almost the same as above.

Specifically, we can consider a sequence $\{b^n\}_n$ of smooth functions, uniformly bounded in $C^{\alpha H}_T L^p_x \cap C^0_T B^\alpha_{p,p}$, with $\div b^n$ uniformly bounded in $L^1_T L^\infty_x$ and such that $b^n\to b$ in $L^q_T B^\alpha_{p,p}$ for any $q<\infty$.
Then exploiting the a priori bound from \cite[Proposition 4.3]{GalHarMay_21benchmark} and the argument from Proposition \ref{prop:RefinedAPrioriUniqueness}, one can derive uniform estimates for the solutions $X^n$ associated to $X$ and finally pass to the limit with the help of Corollary \ref{cor:RndDataWellPosed}.

Alternatively, let us mention that the existence of a weak solution $X$ satisfying $\cL(X_\cdot)\in L^\infty_T L^{p'}_x$ in this setting can be obtained by an application of \cite[Proposition 4.4]{GalHarMay_21benchmark}.
%
%
%
\end{proof}

\begin{proof}[Proof of Theorem \ref{thm:main_thm3}]
It is now an immediate consequence of Propositions \ref{prop:RefinedAPrioriUniqueness} and \ref{prop:DivStrongExist}.
\end{proof}

\subsection{Integrable kernels}\label{subsec:refined-integrable}

We now restrict ourselves to the case $H \leq 1/2$ and drifts $b\in L^q_T L^p_x$; in this setting we can present a second route to establishing existence of a solution with sufficiently regular law, to which we can apply Proposition \ref{prop:RefinedAPrioriUniqueness}.

Before proceeding further, let us explain why it is reasonable to expect so. By the Besov embedding $L^p_t \hookrightarrow B^{-d/p}_{\infty,\infty}$, drifts $b\in L^q_T L^p_x$ satisfy Assumption \ref{ass:drift-SDE} if and only if
\begin{equation}\label{eq:parameters-qp-strong}
     \frac{1}{q}+\frac{Hd}{p}<\frac{1}{2}-H;
\end{equation}
however, differently from the class $L^q_T B^\alpha_{p,p}$, for $b\in L^q_T L^p_x$ it is known after the works \cite{nualart03stochastic,le2020stochastic} that Girsanov transform (and thus weak existence and uniqueness in law for associated SDEs) is available as soon as
\begin{equation}\label{eq:parameters-qp-weak}
    \frac{1}{q}+\frac{Hd}{p}<\frac{1}{2}.
\end{equation}
As already seen in Section \ref{subsec:reg-estim}, Girsanov transform allows to deduce information on the regularity of $\cL(X_t)$, which in turn provides higher regularity of the effective drift $b^\mu$ for the convolutional DDSDE.
In particular, we may hope that starting from $b\in L^q_T L^p_x$ for $(q,p)$ satisfying \eqref{eq:parameters-qp-weak}, we end up with $b^\mu\in L^{\tilde q}_T L^{\tilde p}_x$ with $(\tilde q,\tilde p)$ satisfying \eqref{eq:parameters-qp-strong}.


At a technical level, we will proceed similarly as in Section \ref{subsec:refined-divergence}, first establishing uniform a priori estimates for regular $b$ and then running an approximation procedure.
We start by establishing the recalling and improving the available results on Girsanov transform; as we are only interested in smooth approximations, for simplicity we restrict to regular drifts.

\begin{lem}\label{lem:girsanov-sec5}
Let $(\Omega,\cF,\PP)$ be a probability space, $(\xi,W)$ a $\RR^d\times C_T$-valued r.v. on it with $\cL_\PP(\xi,W)=\mu_0\otimes \mu^H$ for some $H \leq 1/2$ and let $f:[0,T]\times\RR^d\to \RR^d$ be a globally Lipschitz drift, $f\in L^q_T L^p_x$ for parameters $(p,q)\in [1,\infty]^2$ satisfying \eqref{eq:parameters-qp-weak}; let $X$ be the unique strong solution to
\begin{equation*}
    X_t= \xi + \int_0^t f_s (X_s)\dd s+ W_t \quad \forall\, t\in [0,T].
\end{equation*}
Then there exists a measure $\QQ$ equivalent to $\PP$ such that $\cL_\QQ(X)=\cL_\PP(\xi+W)$ and there exists an increasing function $F$, depending on $H,T,p,q$, such that
\begin{equation*}
    \EE_\QQ\bigg[\Big(\frac{\dd \PP}{\dd \QQ}\Big)^n\bigg] + \EE_\QQ\bigg[\Big(\frac{\dd \QQ}{\dd \PP}\Big)^n\bigg] \leq F(n, \| f\|_{L^q_T L^p_x})<\infty \quad \forall\, n\in \NN
\end{equation*}
where the estimate does not depend on $\mu_0$ nor the specific function $f$.
\end{lem}

\begin{proof}
For deterministic initial data $\xi=x_0\in \RR^d$ (equiv. $\mu_0=\delta_{x_0}$), the statement is a direct consequence of \cite[Lemma 6.7]{le2020stochastic}, where it is already stressed that the estimates only depend on $\| f\|_{L^q_T L^p_x}$ but not on $x_0$ nor the specific $f$.
The proof for random initial data $\xi$ independent of $W$ is now identical to that of Corollary \ref{cor:SDE-random-data-girsanov}; the estimate not depending on $\xi$ follows from the property that $\| f(x_0+\cdot)\|_{L^q_T L^p_x} = \| f\|_{L^q_T L^p_x}$ for all $x_0\in \RR^d$.
\end{proof}

The next lemma shows that the initial regularity of $\mu_0$ is propagated at positive times, establishing useful a priori estimates; the proof is similar to that of Proposition \ref{prop:regularity-density2}.

\begin{lem}\label{lem:girsanov implies integrability}
Let $\xi,\,W,\,X,\,f,\,(p,q)$ be as in Lemma \ref{lem:girsanov-sec5} and assume $\mu_0\in L^r_x$ for some $r\in (1,\infty)$; then
\begin{align*}
    \sup_{t\in [0,T]} \| \cL_\PP(X_t)\|_{L^{\tilde r}_x} <\infty \quad \forall\, \tilde r\in (1,r).
\end{align*}
\end{lem}

\begin{proof}
Fix $\tilde r<r$ and denote by $\tilde r'$ the conjugate exponent of $\tilde r$; take $\eps>0$ such that $r'(1+\eps)=\tilde r'$.
Let $\QQ$ be the measure given by Lemma \ref{lem:girsanov-sec5} such that $\cL_\QQ(X)=\cL_\PP(\xi+W)$; since $\mathd \PP/\mathd \QQ$ admits moments of any order, for any $g\in C^\infty_c(\RR^d)$, by H\"older
\begin{align*}
|\langle g,\mathcal{L}(X_t)\rangle|
& \leq \EE_\PP[|g|(X_t)]
= \EE_\QQ\bigg[|g|(\xi+W_t)\, \frac{\mathd \PP}{\mathd\QQ} \bigg]\\
& \leq \EE_\QQ\big[|g|^{1+\eps} (\xi+W_t)\big]^{\frac{1}{1+\eps}} \, \EE_\QQ\bigg[\Big( \frac{\mathd \PP}{\mathd \QQ}\Big)^{1+\frac{1}{\eps}}\bigg]^{\frac{\eps}{1+\eps}}\\
& \lesssim_\eps \langle |g|^{1+\eps}, \mu_0\ast \mathcal{L}(W_t)\rangle^{\frac{1}{1+\eps}}
\end{align*}
where in the last passage we used the fact that $\xi$ and $W_t$ are independent.
Recalling that $\mathcal{L}(W_t)$ is a probability measure, by H\"older's and then Young's inequality we arrive at
\begin{equation*}
    |\langle g,\mathcal{L}(X_t)\rangle|
    \lesssim_\eps \| |g|^{1+\eps}\|_{L^{r'}_x}^{\frac{1}{1+\eps}}\, \| \rho\ast \cL(W_t) \|_{L^r_x}^{\frac{1}{1+\eps}}
    \lesssim \| g\|_{L^{r'(1+\eps)}_x}\, \| \mu_0 \|_{L^r_x}^{\frac{1}{1+\eps}}
    = \| g\|_{L^{\tilde r'}_x}\, \| \mu_0 \|_{L^r_x}^{\frac{1}{1+\eps}}
\end{equation*}
%
%
As the estimate is uniform over all $g\in C^\infty_c(\RR^d)$ and $t\in [0,T]$, by duality we deduce that
\begin{equation*}
  \sup_{t\in [0,T]}  \|\cL(X_t)\|_{L^{\tilde r}_x}\lesssim_{\varepsilon} \|\mu_0\|_{L^r_x}^{\frac{1}{1+\eps}};
\end{equation*}
%
as the reasoning holds for all $\tilde r<r$, the conclusion follows.
\end{proof}


We are now ready to prove the existence of solutions to the DDSDE \eqref{eq:ConvMckeanSDE} for $b\in L^q_T L^p_x$ and sufficiently integrable $\mu_0$.

\begin{prop}\label{prop:GirsanovLqLpExist}
Let $H \leq 1/2$,  $(p,r,q) \in [1,\infty)\times (1,\infty) \times (2,\infty]$ such that
\begin{equation}\label{eq:GirsanovBoodstrap}
\frac{1}{q}+Hd \bigg(\frac{1}{p}+\frac{1}{r}-1\bigg) < \frac{1}{2} -H, \quad \frac{1}{q}+\frac{Hd}{p}<\frac{1}{2}.
\end{equation}
Then for any $b \in L^q_T L^{p}_x$ and any $\mu_0\in L^r_x$ there exists a strong solution $X$ to the associated DDSDE \eqref{eq:ConvMckeanSDE}, which moreover satisfies $\mathcal{L}(X_\cdot)\in L^\infty_T L^{\tilde r}_x$ for any $\tilde r\in [1,r)$.
\end{prop}


\begin{proof}
We pursue the same general strategy as in the proof of Proposition \ref{prop:DivStrongExist}.

As condition \eqref{eq:GirsanovBoodstrap} only contains strict inequalities, w.l.o.g. we can assume $q<\infty$; consider a sequence $\{b^n\}_n$ of Lipschitz, compactly supported functions such that $b^n\to b$ in $L^q_T L^p_x$ and $\| b^n\|_{L^q_T L^p_x}\leq \| b\|_{L^q_T L^p_x}$.
It follows from \cite[Theorem 7]{coghi2020pathwise} that for every $n$ there exists a unique solution $X^n$ to the approximating DDSDE
\begin{equation*}
    X^n_t = \xi + \int_0^t (b^n_s\ast \mu^n_s)(X^n_s)\,\dd s+ W_t, \quad \mu^n_t=\cL(X^n_t).
\end{equation*}
In particular, each $X^n$ is also a solution to an SDE with drift $f^n_t:=b^n_t\ast\cL(X^n_t)$ and by Young's inequality
\begin{align*}
    \sup_n \| f^n\|_{L^q_T L^p_x} \leq \sup_n \| b^n\|_{L^q_T L^p_x} \leq \| b\|_{L^q_T L^p_x}.
\end{align*}
Therefore we may apply Lemmas \ref{lem:girsanov-sec5} and \ref{lem:girsanov implies integrability} to obtain the uniform bound
\begin{equation*}
    \sup_n \|\mu^n_\cdot\|_{L^\infty_T L^{\tilde r}_x} <\infty
\end{equation*}
for all $\tilde r\in [1,r)$.
Applying H\"older's inequality to the integral in time and and Young's inequalities to the convolution in space, we find
\begin{equation*}
    \sup_n \|b^n \ast \mu^n\|_{L^q_T L^{\tilde p}_x} <\infty
\end{equation*}
for any $\tilde p< \bar{p}$, where
\begin{equation*}
    1+\frac{1}{\bar{p}} = \frac{1}{r} + \frac{1}{p}.
\end{equation*}
%
%
%
Using the fact that $\tilde p$ can be chosen arbitratrily close to $\bar{p}$ and that the first inequality in \eqref{eq:GirsanovBoodstrap} is strict, we see that the family $\{b^n\ast \mu^n\}$ is bounded in $L^q_T L^{\tilde p}_x$ for parameters $(q,\tilde p)$ satisfying
\begin{equation*}
    \frac{1}{q}+ \frac{Hd}{\tilde p} <\frac{1}{2}-H;
\end{equation*}
but this is exactly condition \eqref{eq:parameters-qp-strong}, i.e. the regularity regime in which we know how to solve the SDE in a strong sense.
On the other hand, the uniform bound for $\| b^n\|_{L^q_T L^p_x}$ and the use of Girsanov transform allows to derive a uniform bound for $\EE[\llbracket X^n\rrbracket_{H-\eps}]$ for any $\eps>0$;
together with $X^n_0=\xi$ for all $n$ this implies tightness of $\{X^n\}_n$, so that we can extract a (not relabelled) subsequence such that $\cL(X^n)\rightharpoonup \mu$ in $\cP(C_T)$, $\mu^n_t\rightharpoonup \mu_t=e_t\sharp \mu$ for all $t\in [0,T]$.

From here, the argument is almost identical to that of Proposition \ref{prop:DivStrongExist}: once we show that $b^n\ast \mu^n\to b\ast \mu$ in a sufficiently strong topology, then by Corollary \ref{cor:RndDataWellPosed} the solutions $X^n$ will converge to the unique strong solution $X$ associated to $(b\ast \mu,\xi)$, which must therefore be a solution to the DDSDE associated to $\mu_0=\cL(\xi)$ and $b$.
By the uniform bounds on $\{\mu^n\}_n$ and weak convergence $\mu^n_t\rightharpoonup\mu_t=\cL(X_t)$, we also deduce that $\cL(X_\cdot)\in L^\infty_T L^{\tilde r}_x$ for all $\tilde r< r$.

Since by construction $b^n\to b$ in $L^q_t L^p_x$, $\mu^n_t\rightharpoonup \mu_t$ and $\{b^n\ast \mu^n\}_n$ is bounded in $L^q_T L^{\tilde p}_x$ for some $(q,\tilde p)$ satisfying \eqref{eq:parameters-qp-strong}, we can apply Corollary \ref{cor:IntegrableConvolve} from Appendix \ref{app:UsefulLemmas} to deduce that (up to further relabelling $\tilde p-\eps$ into $\tilde p$) $b^n\ast \mu^n\to b\ast \mu$ in $L^q_T L^{\tilde p}_x$.
In light of the embedding $L^q_T L^{\tilde p}_x \hookrightarrow L^q_T B^{-d/\tilde{p}}_{\infty,\infty}$ and Corollary \ref{cor:RndDataWellPosed}, this implies the conclusion.
\end{proof}

We are now ready to prove the main result of this subsection.

\begin{thm}\label{thm:main_sec5}
Let $H \leq 1/2$, $(p,r,q) \in [1,\infty)\times (1,\infty) \times (2,\infty]$ satisfy \eqref{eq:GirsanovBoodstrap}.
Then for any $b \in L^q_T L^p_x$ and $\mu_0 \in L^r_x$, then there exists a strong solution $X$ to \eqref{eq:ConvMckeanSDE}, which satisfies $\mathcal{L}(X_\cdot)\in L^\infty_T L^{\tilde r}_x$ for any $\tilde r<r$; pathwise uniqueness and uniqueness in law hold in the class of solutions satisfying this condition.
\end{thm}

\begin{proof}
The proof is based on a (non-trivial) combination of Propositions \ref{prop:RefinedAPrioriUniqueness} and \ref{prop:GirsanovLqLpExist}.

Under our assumptions, the existence of a strong solution such that $\cL(X_\cdot)\in L^\infty_T L^{\tilde r}_x$ for any $\tilde r<r$ is granted; in particular if $r>p'$, then we can choose $\tilde r=p'$ and then assumptions of Proposition \ref{prop:DivStrongExist} in this case are satisfied thanks to the embedding $L^q_T L^p_x \hookrightarrow L^q_T B^{-\eps}_{p,p}$ for any $\eps>0$, giving the uniqueness part of the statement.
Up to technicalities the borderline case $r=p'$ can be treated similarly, exploiting the embedding $L^q_T L^p_x \hookrightarrow L^q_T B^{-\eps}_{\tilde p,\tilde p}$ for some $\tilde p=\tilde p(\eps)>p$ chosen so that $1/\tilde r + 1/\tilde p =1$.

Thus it remains to study the regime $r< p'$, equivalently $r'> p$; in this case we can choose $\tilde r<r$ such that $\tilde r'>p$ as well.
By Besov embedding it then holds
\begin{align*}
    L^q_T L^p_x \hookrightarrow B^\alpha_{\tilde r',\tilde r'} \quad \text{for}\quad
    \alpha:= -d\bigg( \frac{1}{p}-\frac{1}{\tilde r'}\bigg)
    = -d\bigg( \frac{1}{p}+\frac{1}{\tilde r}-1\bigg);
\end{align*}
to verify that $b$ satisfies the assumptions of Proposition \ref{prop:RefinedAPrioriUniqueness}, it then suffices to check that
\begin{align*}
    \frac{1}{q} + Hd \bigg( \frac{1}{p}+\frac{1}{\tilde r}-1\bigg) < \frac{1}{2}-H;
\end{align*}
since $\tilde r$ can be taken arbitrarily close to $r$, this follows from the first strict inequality in \eqref{eq:GirsanovBoodstrap}.
\end{proof}

\begin{rem}
For $r>d/(d-1)$, condition \eqref{eq:parameters-qp-weak} implies condition \eqref{eq:GirsanovBoodstrap}.
Therefore Theorem \ref{thm:main_thm4} is a particular subcase of Theorem \ref{thm:main_sec5}.
\end{rem}

\appendix
\section{Some useful lemmas}\label{app:UsefulLemmas}

We collect in this appendix several technical lemmas of analytic nature that have been used throughout the paper. We start with useful facts on the identification of elements of dual spaces.

\begin{lem}\label{lem:duality1}
Suppose $f\in L^1_{loc}([0,T]\times\RR^d)$, $q\in [1,\infty)$ and there exists a constant $C$ such that
\[
\int_{[0,T]\times\RR^d} f(t,x)\varphi(t,x)\,\dd t\dd x \leq C \| \varphi\|_{L^{q'}_T L^\infty_x} 
\]
for all compactly supported $\varphi\in L^\infty([0,T]\times\RR^d)$. Then $f\in L^q_T L^1_x$ and $\| f\|_{L^q_T L^1_x} \leq C$.
\end{lem}

\begin{proof}
Fix $M>0$ and set
\[
g_t := \int_{\RR^d} |f(t,x)| \ind_{|f(t,\cdot)|\leq M} (x)\, \ind_{|x|\leq M} \dd x,\quad h(t,x):= g_t^{q-1} \text{sgn}(f(t,x)) \ind_{|f(t,\cdot)|\leq M} (x)\, \ind_{|x|\leq M};
\]
then by assumption
\[
\int_{[0,T]} |g_t|^q \dd t = \int_{[0,T]\times\RR^d} f(t,x) h(t,x)\,\dd t\dd x \leq C \| h\|_{L^{q'}_t L^\infty_x} = C \bigg(\int_{[0,T]} |g_t|^q \dd t\bigg)^{1-1/q},
\]
namely
\[
\Bigg( \int_{[0,T]} \bigg( \int_{\RR^d} |f(t,x)| \ind_{|f(t,\cdot)|\leq M} (x)\, \ind_{|x|\leq M} \dd x\bigg)^q \dd t\Bigg)^{1/q} \leq C.
\]
Taking the limit $M\to\infty$ gives the conclusion.
\end{proof}

\begin{lem}\label{lem:duality2}
Suppose $f\in L^1_{loc}([0,T]\times\RR^d)$, $q\in (1,\infty)$ and there exists a constant $C$ such that
\[
\int_{[0,T]\times\RR^d} f(t,x)\varphi(t,x)\,\dd t\dd x \leq C \| \varphi\|_{L^{q'}_T L^1_x} 
\]
for all compactly supported $\varphi\in L^\infty([0,T]\times\RR^d)$. Then $f\in L^q_T L^\infty_x$ and $\| f\|_{L^q_T L^\infty_x} \leq C$.
\end{lem}

\begin{proof}
Fix $R>0$ and set $D=\{x\in \RR^d:|x|\leq R\}$; we can identify $L^p_D:=L^p(D)$ as the subset of $L^p(\RR^d)$ made of functions supported on $D$, similarly $L^q_T L^p_D$ as a subset of $L^q_T L^p_x$.
Observe that for $p>1$, $L^{q'}_T L^p_D \hookrightarrow L^{q'}_T L^1_D$ with $\| f\|_{L^{q'}_T L^1_D} \leq (c_d R^d)^{1/p'} \| f\|_{L^{q'}_T L^p_D}$.
By hypothesis $f$ defines an element of the dual of $L^{q'}_T L^p_D$, thus by duality
\[
\| f\|_{L^q_T L^{p'}_D} \leq C (c_d R^d)^{1/p'} \quad \forall p'\in (1,\infty).
\]
Since $D$ has finite measure, $\| g\|_{L^\infty_D} = \lim_{p\to\infty} \| g\|_{L^p_D}$ for all measurable $g$; taking $p'\to\infty$ in the above by Fatou's lemma we obtain
\[
\| f\|_{L^q_T L^\infty_D} \leq C;
\]
as the reasoning holds for any $R>0$, the conclusion follows talking $R\to\infty$.
\end{proof}

The next statements concern the compactness properties of convolutions, specifically how weak convergence of measures is enhanced to strong convergence of associated functionals $b\ast \mu$.

\begin{lem}\label{lem:MeasureConvolveConverge}
Let $b\in L^p_x$ with $p\in [1,\infty)$ and $\{\mu^n\}_n\subset \cP(\RR^d)$ such that $\mu^n\rightharpoonup \mu$ weakly, then
\begin{align*}
    \lim_{n\to\infty}\|b\ast \mu^n-b\ast \mu\|_{L^p_x}=0.
\end{align*}
If moreover $\{\mu^n\}_n$ is bounded in $L^r_x$ with $r>1$, then $b\ast \mu^i\to b\ast \mu$ in $L^{\tilde p}_x$ for any $\tilde p\in [p,\infty)$ s.t.
\begin{equation}\label{eq:parameters-measure-convolve}
    1+\frac{1}{\tilde p} >\frac{1}{p}+\frac{1}{r}.
\end{equation}
\end{lem}

\begin{proof}
Given $h\in \RR^d$, define the translation operator $\tau_h:f\mapsto f(\cdot+h)$ acting on $L^p_x$. Recall that any given $b\in L^p_x$ is equicontinuous, in the sense that $\tau_{h^n} b\to \tau_h b$ in $L^p_x$ for $h^n\to h$ in $\RR^d$.

Since $\mu^n\rightharpoonup\mu$, by Skorokhod's representation theorem we can construct a probability space $(\Omega,\cF,\PP)$ and a family of r.v.s $\{X^n\}_n,\,X$ on it such that $\cL_\PP(X^n)=\mu^n,\,\cL_\PP(X)=\mu$ and $X^n\to X$ $\PP$-a.s.; it then holds
\begin{align*}
    \| b\ast \mu^n-b\ast \mu\|_{L^p_x}^p
    & = \int_{\RR^d} |\EE[b(x-X^n)-b(x-X)]|^p\, \dd x\\
    & \leq \EE\bigg[ \int_{\RR^d} |b(x-X^n)-b(x-X)|^p\, \dd x\bigg]
    = \EE\Big[ \| \tau_{X^n} b-\tau_{X} b\|_{L^p_x}^p\Big]
\end{align*}
where in the second passage we used Jensen's inequality. By the aforementioned equicontinuity it holds $\| \tau_{X^n} b-\tau_{X} b\|_{L^p_x}\to 0$ $\PP$-a.s. and we have the uniform bound $\| \tau_{X^n} b-\tau_{X} b\|_{L^p_x}\leq 2 \| b\|_{L^p_x}$, thus the first claim follows from dominated convergence.

Regarding the second claim, Young's inequality gives a uniform bound for $\{b\ast \mu^n\}$ in $L^{\bar p}_x$ for $1+1/\bar{p}=1/p+1/r$; combined with convergence in $L^p_x$ and interpolation estimates, we deduce convergence in $L^{\tilde p}_x$ for any $\tilde p\in [p,\bar{p})$.
\end{proof}

\begin{rem}
In the borderline case of $b\in L^p_x$ and $\{\mu^n\}_n$  bounded in $L^{p'}_x$, then $\{b\ast \mu^n\}_n$ is a bounded sequence in  $C_0(\RR^d)$, which denotes the Banach space of continuous functions vanishing at infinity, endowed with the supremum norm. In this case it can be shown that $\{b\ast \mu^n\}_n$ is also equicontinuous, so by Ascoli--Arzel\`{a} it converges to $b\ast \mu$ uniformly on compact sets.
\end{rem}

\begin{cor}\label{cor:BesovConvolve}
Let $b\in B^\alpha_{p,p}$ with $p\in [1,\infty)$,  $\alpha\in \RR$, and $\{\mu^n\}_n\subset \cP(\RR^d)$ such that $\mu^n \rightharpoonup \mu$ weakly, then $b\ast \mu^n \to b\ast \mu$ in $B^\alpha_{p,p}$;
if moreover $\{\mu^n\}_n$ is bounded in $L^r_x$ with $r>1$, then $b\ast \mu^n\to b\ast \mu$ in $B^\alpha_{\tilde p,\tilde p}$ for any $\tilde p\in [p,\infty)$ satisfying \eqref{eq:parameters-measure-convolve}.
Finally, if
\begin{align*}
    \frac{1}{p}+\frac{1}{r}<1,
\end{align*}
then $b\ast \mu^n\to b\ast \mu$ in $B^{\tilde \alpha}_{\infty,\infty}$ for any $\tilde \alpha<\alpha$.
\end{cor}

\begin{proof}
Observe that $\Delta_j (b\ast \mu^n-b\ast \mu)=(\Delta_j b)\ast (\mu^n-\mu)$, so by Lemma \ref{lem:MeasureConvolveConverge}  $\|\Delta_j (b\ast \mu^n-b\ast \mu) \|_{L^p_x}\to 0$ for any fixed $j\in\NN$. Moreover
\begin{equation*}
\| (\Delta_j b)\ast (\mu^n- \mu) \|_{L^p_x}^p\lesssim_p \| \Delta_j b \|_{L^p_x}^p
\quad \text{with} \quad
\sum_j 2^{\alpha j p} \| \Delta_j b \|_{L^p_x}^p<\infty
\end{equation*}
therefore by dominated convergence
\begin{equation*}
\lim_i \| b\ast \mu^n-b\ast \mu\|_{B^\alpha_{p,p}}^p
= \lim_i \sum_j 2^{\alpha j p} \| (\Delta_j b)\ast (\mu^n- \mu) \|_{L^p_x}^p
= 0.
\end{equation*}
The second statement follows as before by interpolation between convergence in $B^\alpha_{p,p}$ and boundedness in $B^{\alpha}_{\bar{p},\bar{p}}$ with $\bar{p}$ satisfying the equality in \eqref{eq:parameters-measure-convolve}.

Finally, if $1>1/p+1/r$, then condition \eqref{eq:parameters-measure-convolve} is satisfied for all $\tilde p$ big enough; the conclusion then follows from convergence in $B^\alpha_{\tilde p,\tilde p}$ and the Besov embedding $B^\alpha_{\tilde p,\tilde p}\hookrightarrow B^{\alpha-d/\tilde p}_{\infty,\infty}$.
\end{proof}

\begin{cor}\label{cor:IntegrableConvolve}
Let $p,q\in [1,\infty)$, $\{b^n\}_n\subset L^q_T L^p_x$ be a sequence such that $b^n\to b$ in $L^q_T L^p_x$; moreover let $\{\mu^n,\mu\}\subset C_T \cP(\RR^d)$ be such that $\mu^n_t\rightharpoonup \mu_t$ weakly for every $t\in [0,T]$ and such that
\begin{align*}
    \sup_{n\in\NN} \sup_{t\in [0,T]}\|\mu^n\|_{L^r_x} + \sup_{t\in [0,T]} \|\mu\|_{L^r_x} <\infty
\end{align*}
for some $r\in [1,\infty]$. Then $b^n\ast \mu^n\to b\ast \mu$ in $L^q_T L^{\tilde p}_x$ for all $\tilde{p}\in [p,\infty)$ satisfying \eqref{eq:parameters-measure-convolve}.
\end{cor}

\begin{proof}
It suffices to show that $b^n\ast \mu^n\to b\ast\mu$ in $L^q_T L^p_x$; once this is done, convergence in $L^q_T L^{\tilde p}_x$ follows as usual by interpolation and boundedness in $L^q_T L^{\bar{p}}_x$, which comes from the assumptions and Young's inequality. It holds
\begin{align*}
    \| b^n\ast \mu^n\|_{L^q_T L^p_x} \leq \| (b^n-b)\ast\mu^n\|_{L^q_T L^p_x} + \| b\ast \mu^n-b\ast \mu\|_{L^q_T L^p_x}
\end{align*}
where we can estimate the first term by $\| (b^n-b)\ast\mu^n\|_{L^q_T L^p_x} \leq \| b^n-b\|_{L^q_T L^p_x} \to 0$.
For the second term, by Lemma \ref{lem:MeasureConvolveConverge} and the assumptions it holds $\| b_t\ast (\mu^n_t-\mu_t)\|_{L^p_x}\to 0$ for Lebesgue a.e. $t\in [0,T]$, as well as $\| b_t\ast (\mu^n_t-\mu_t)\|_{L^p_x}\leq 2\| b_t\|_{L^p_x}$; thus by dominated convergence we infer $\| b\ast \mu^n-b\ast\mu\|_{L^q_T L^p_x}\to 0$ as well.
\end{proof}

\begin{lem}\label{lem:basic-Besov-convolve}
For any $p\in [1,\infty)$ and $\alpha\in \RR$ there exists a constant $C=C(\alpha)$ such that
\begin{align*}
    \| b\ast (\mu-\nu)\|_{B^{\alpha-1}_{\infty,\infty}} \leq C \, \| b\|_{B^\alpha_{\infty,\infty}}\, d_p(\mu,\nu)
\end{align*}
for all $b\in B^\alpha_{\infty,\infty}$ and $\mu,\nu\in \cP(\RR^d)$.
\end{lem}

\begin{proof}
It's enough give the proof for $p=1$, as the general case follows from $d_1(\mu,\nu)\leq d_p(\mu,\nu)$; we can assume $d_1(\mu,\nu)<\infty$, otherwise the inequality is trivial.
By Bernstein estimates, reasoning on Littlewood-Paley blocks, we have
\begin{align*}
    \Vert b \ast (\mu-\nu)\Vert_{B^{\alpha-1}_{\infty,\infty}}
    & = \sup_n \big\{ 2^{n(\alpha-1)} \Vert (\Delta_n b)\ast (\mu-\nu)\Vert_{L^\infty_x} \big\}\\
    & \leq \sup_n  \big\{ 2^{n(\alpha-1)} \Vert \Delta_n b\Vert_{W^{1,\infty}_x} \big\}\, d_1(\mu,\nu)\\
    & \lesssim \sup_n  \big\{ 2^{n\alpha} \Vert \Delta_n b\Vert_{L^\infty_x} \big\}\, d_1(\mu,\nu)
    = \Vert b\Vert_{B^\alpha_{\infty,\infty}}\, d_1(\mu,\nu)
\end{align*}
which gives the claim.
\end{proof}

The last statements we are going to provide concern the continuity of the map $\mu\mapsto b\ast \mu$ in suitable topologies. Their proof require the use of maximal functions and their basic properties, which we recall first; we refer the interested reader to \cite{stein1970singular} for their proofs.

Given $b\in L^p(\RR^d)$, $p\in [1,\infty]$, its maximal function $M b$ is defined by
\[
M b(x) := \sup_{r>0} \frac{1}{\lambda_d\, r^d} \int_{B(x,r)} |b(y)|\,\dd y
\]
where $\lambda_d$ stands for the Lebesgue measure of $B(0,1)$ in $\RR^d$. It is well known that if $p\in (1,\infty]$, then $M f\in L^p(\RR^d)$ and
\[
\|Mb \|_p\leq c_{d,p} \|b\|_p
\]
for some constant $c_{d,p}>0$; similar definitions and properties hold in the case of vector-valued drifts $b\in L^p(\RR^d;\RR^m)$ (in which case $c=c_{d,p,m}$).

If $b\in W^{1,p}(\RR^d;\RR^d)$, then there exists a Lebesgue-negligible set $N\subset\RR^d$ and a constant $c_d>0$ such that the Hajlasz inequality holds:
\begin{equation}\label{eq: maximal function1}
|b(x)-b(y)|\leq c_d\, |x-y|\,( M Db(x) + MDb(y))\quad \forall\, x,y\in \RR^d\setminus N.
\end{equation}

The above and similar inequalities (see \cite{caravenna2021directional, brue2021positive} for recent asymmetric extensions) allow to control the map $\mu\mapsto b\ast \mu$ in Wasserstein spaces.

\begin{lem}\label{lem: maximal function1}
Let $(p,q,r,s)\in (1,\infty)^4$ be such that $r \leq p\wedge s$ and
\begin{equation}\label{eq:param-maximal-ineq}
    1+\frac{r}{s} \leq \frac{r}{p}+\frac{1}{q}.
\end{equation}
Then there exists a constant $C$, depending on $d$ and the above parameters, such that for any $b\in W^{1,p}_x$ and $\mu,\nu \in L^q_x$ it holds
\begin{equation*}
\| b\ast(\mu-\nu) \|_{L^s_x} \leq C \| b\|_{W^{1,p}_x}\, \big(\| \mu\|_{L^q_x}^{1/r}+\|\nu\|_{L^q_x}^{1/r}\big)\, d_{r'}(\mu,\nu)
\end{equation*}
\end{lem}

\begin{proof}
If $d_{r'}(\mu,\nu)=+\infty$ the inequality is trivially true, so we can assume $d_{r'}(\mu,\nu)<\infty$; moreover since $\mu,\nu\in L^{\tilde{q}}_x$ for any $\tilde{q}\in [1,q]$, w.l.o.g. we can assume equality holds in \eqref{eq:param-maximal-ineq}. 

Let $m\in \Pi(\mu,\nu)$ be an optimal coupling of $(\mu,\nu)$ for $d_{r'}(\mu,\nu)$ and let $N\subset\RR^d$ be as in \eqref{eq: maximal function1}; since $\mu,\nu$ are absolutely continuous w.r.t. Lebesgue, it holds that $m((N\times\RR^d) \cup (\RR^d\times N))=0$.
Therefore we can apply \eqref{eq: maximal function1} for any fixed $x\in \RR^d$ to find
\begin{align*}
|b\ast(\mu-\nu)(x)|
& = \bigg| \int_{\RR^{2d}} [b(x-y)-b(x-z)]\, m(\dd y,\dd z)]\bigg|\\
& \lesssim \int_{\RR^{2d}} |y-z|\, (M Db(x-y)+MDb(x-z))\, m(\dd y,\dd z) \\
& \leq \bigg(\int_{\RR^{2d}} | y-z|^{r'} \,m(\dd y,\dd z) \bigg)^{1/r'}\, \bigg(\int_{\RR^{2d}} | M Db(x-y) + M Db(x-z)|^r \,m(\dd y,\dd z)\bigg)^{1/r}\\
& \leq d_{r'}(\mu,\nu)\, \bigg[ \Big(\int_{\RR^d} |M Db(x-y)|^r \mu(\dd y)\Big)^{1/r} + \Big(\int_{\RR^d} |M Db(x-z)|^r \nu(\dd z)\Big)^{1/r}\bigg].
\end{align*}
Observe that
\begin{align*}
    \Big(\int_{\RR^d} |M Db(\cdot-y)|^r \mu(\dd y)\Big)^{1/r} = (|M Db|^r\ast \mu)^{1/r}(\cdot)
\end{align*}
and by assumption \eqref{eq:param-maximal-ineq} and Young's inequality it holds
\begin{align*}
    \|(|M Db|^r \ast \mu)^{1/r} \|_{L^s_x}
    = \| |M Db|^r \ast \mu\|_{L^{s/r}_x}^{1/r}
    \lesssim \||M Db|^r \|_{L^{p/r}_x}^{1/r} \, \| \mu\|_{L^q_x}^{1/r}
    = \| M D b\|_{L^p_x}\,\| \mu\|_{L^q_x}^{1/r}.
\end{align*}
Together with a similar estimate for $|M Db|\ast\nu$ and the property $\| M Db\|_{L^p_x} \lesssim \| b\|_{W^{1,p}_x}$, combining everything we arrive at
\begin{equation*}
    \| b\ast (\mu-\nu)\|_{L^s_x} \lesssim \| b\|_{W^{1,p}_x} (\| \mu\|_{L^q_x}^{1/r} + \| \nu\|_{L^q_x}^{1/r})\, d_{r'}(\mu,\nu)
\end{equation*}
which is the claim.
\end{proof}

\begin{cor}\label{cor: maximal function2}
Let $\alpha\in \RR$, $(p,q,r,s)$ as in Lemma \ref{lem: maximal function1}. Then there exists a constant $C$, depending on $d$ and the above parameters, such that for any $b\in B^\alpha_{p,p}$ and $\mu,\nu\in L^q_x$
\begin{equation*}
\| b\ast(\mu-\nu) \|_{B^{\alpha-1}_{s,s}} \leq C \| b\|_{B^\alpha_{p,p}}\, \big(\| \mu\|_{L^q_x}^{1/r}+\|\nu\|_{L^q_x}^{1/r}\big)\, d_{r'}(\mu,\nu)
\end{equation*}
\end{cor}

\begin{proof}
Let $\Delta_j b$ denote the Littlewood--Paley blocks of $b$, then $\Delta_j( b\ast(\mu-\nu))= (\Delta_j b)\ast (\mu-\nu)$. By the previous lemma and Bernstein estimates, one has that
\begin{align*}
\|(\Delta_j b)\ast (\mu-\nu) \|_{L^s_x}
& \lesssim \| \Delta_j b\|_{W^{1,p}_x} \big(\| \mu\|_{L^q_x}^{1/r}+\|\nu\|_{L^q_x}^{1/r}\big)\, d_{r'}(\mu,\nu)\\
& \lesssim 2^j\, \| \Delta_j b\|_p \big(\| \mu\|_{L^q_x}^{1/r}+\|\nu\|_{L^q_x}^{1/r}\big)\, d_{r'}(\mu,\nu);
\end{align*}
moreover assumption \eqref{eq:param-maximal-ineq} implies $s\geq p$. Therefore we have
\begin{align*}
    \| b\ast (\mu-\nu)\|_{B^\alpha_{s,s}}
    & = \bigg(\sum_j \Big[ 2^{(\alpha-1) j} \|(\Delta_j b)\ast (\mu-\nu) \|_{L^s_x}\Big]^s\bigg)^{1/s}\\
    & \lesssim \bigg(\sum_j \Big[ 2^{\alpha j} \|\Delta_j b\|_{L^s_x}\Big]^s\bigg)^{1/s} \big(\| \mu\|_{L^q_x}^{1/r}+\|\nu\|_{L^q_x}^{1/r}\big)\, d_{r'}(\mu,\nu) \\
    & \lesssim \bigg(\sum_j \Big[ 2^{\alpha j} \|\Delta_j b\|_{L^s_x}\Big]^p\bigg)^{1/p} \big(\| \mu\|_{L^q_x}^{1/r}+\|\nu\|_{L^q_x}^{1/r}\big)\, d_{r'}(\mu,\nu)
\end{align*}
which gives the conclusion.
\end{proof}

\bibliography{all}{}

\begin{thebibliography}{10}

\bibitem{amine2017c}
Oussama Amine, David Ba{\~n}os, and Frank Proske.
\newblock C-infinity-regularization by {N}oise of {S}ingular {ODE}s.
\newblock {\em arXiv preprint arXiv:1710.05760}, 2017.

\bibitem{BahCheDan}
Hajer Bahouri, Jean-Yves Chemin, and Rapha\"{e}l Danchin.
\newblock {\em Fourier analysis and nonlinear partial differential equations},
  volume 343 of {\em Grundlehren der Mathematischen Wissenschaften [Fundamental
  Principles of Mathematical Sciences]}.
\newblock Springer, Heidelberg, 2011.

\bibitem{banos2019strong}
David Ba{\~n}os, Torstein Nilssen, and Frank Proske.
\newblock Strong existence and higher order {F}r{\'e}chet differentiability of
  stochastic flows of fractional {B}rownian motion driven {SDE}s with singular
  drift.
\newblock {\em Journal of Dynamics and Differential Equations}, pages 1--48,
  2019.

\bibitem{bauer2019mckean}
Martin Bauer and Thilo Meyer-Brandis.
\newblock Mc{K}ean--{V}lasov equations on infinite-dimensional {H}ilbert spaces
  with irregular drift and additive fractional noise.
\newblock {\em arXiv preprint arXiv:1912.07427}, 2019.

\bibitem{BauerMBrandisProske2018}
Martin Bauer, Thilo Meyer-Brandis, and Frank Proske.
\newblock {Strong solutions of mean-field stochastic differential equations
  with irregular drift}.
\newblock {\em Electronic Journal of Probability}, 23(none):1 -- 35, 2018.

\bibitem{bresch2019mean}
Didier Bresch, Pierre-Emmanuel Jabin, and Zhenfu Wang.
\newblock On mean-field limits and quantitative estimates with a large class of
  singular kernels: {A}pplication to the {P}atlak--{K}eller--{S}egel model.
\newblock {\em Comptes Rendus Mathematique}, 357(9):708--720, 2019.

\bibitem{brue2021positive}
Elia Bru{\`e}, Maria Colombo, and Camillo De~Lellis.
\newblock Positive solutions of transport equations and classical nonuniqueness
  of characteristic curves.
\newblock {\em Archive for Rational Mechanics and Analysis}, pages 1--36, 2021.

\bibitem{caravenna2021directional}
Laura Caravenna and Gianluca Crippa.
\newblock A directional {L}ipschitz extension lemma, with applications to
  uniqueness and {L}agrangianity for the continuity equation.
\newblock {\em Communications in Partial Differential Equations}, pages 1--33,
  2021.

\bibitem{Catellier2016}
R\'emi Catellier and Massimiliano Gubinelli.
\newblock Averaging along irregular curves and regularisation of {ODE}s.
\newblock {\em Stochastic Process. Appl.}, 126(8):2323--2366, 2016.

\bibitem{coghi2020pathwise}
Michele Coghi, Jean-Dominique Deuschel, Peter~K. Friz, and Mario Maurelli.
\newblock Pathwise {McKean--Vlasov} theory with additive noise.
\newblock {\em Annals of Applied Probability}, 30(5):2355--2392, 2020.

\bibitem{chaudru2020strong}
P.E.~Chaudru de~Raynal.
\newblock Strong well posedness of {M}c{K}ean--{V}lasov stochastic differential
  equations with {H}{\"o}lder drift.
\newblock {\em Stochastic Processes and their Applications}, 130(1):79--107,
  2020.

\bibitem{diperna1989ordinary}
Ronald~J. DiPerna and Pierre-Louis Lions.
\newblock Ordinary differential equations, transport theory and {S}obolev
  spaces.
\newblock {\em Inventiones mathematicae}, 98(3):511--547, 1989.

\bibitem{Flandoli2011}
Franco Flandoli.
\newblock {\em Random Perturbation of {PDE}s and Fluid Dynamic Models:
  {\'E}cole d'{\'e}t{\'e} de Probabilit{\'e}s de {S}aint-{F}lour {XL}--2010},
  volume 2015.
\newblock Springer Science \& Business Media, 2011.

\bibitem{fournier2014propagation}
Nicolas Fournier, Maxime Hauray, and St{\'e}phane Mischler.
\newblock Propagation of chaos for the 2{D} viscous vortex model.
\newblock {\em Journal of the European Mathematical Society}, 16(7):1423--1466,
  2014.

\bibitem{GalHarMay_21benchmark}
L.~Galeati, F.~Harang, and A.~Mayorcas.
\newblock Distribution dependent {SDE}s driven by additive continuous noise,
  2021.
\newblock Work in progress.

\bibitem{galeati2020nonlinear}
Lucio Galeati.
\newblock Nonlinear {Y}oung differential equations: a review.
\newblock {\em J. Dyn. Diff. Equat., \rm{available at}
  https://doi.org/10.1007/s10884-021-09952-w}, 2021.

\bibitem{galeati2020noiseless}
Lucio Galeati and Massimiliano Gubinelli.
\newblock Noiseless regularisation by noise.
\newblock {\em arXiv preprint arXiv:2003.14264, {\rm to appear in} Revista
  Matem\'atica Iberoamericana}, 2020.

\bibitem{galeati2020prevalence}
Lucio Galeati and Massimiliano Gubinelli.
\newblock Prevalence of $\rho$-irregularity and related properties.
\newblock {\em arXiv preprint arXiv:2004.00872}, 2020.

\bibitem{galeati2020regularization}
Lucio Galeati and Fabian~A. Harang.
\newblock Regularization of multiplicative {SDE}s through additive noise.
\newblock {\em arXiv preprint arXiv:2008.02335}, 2020.

\bibitem{Gartner1988}
J\"{u}rgen G\"{a}rtner.
\newblock On the {M}c{K}ean-{V}lasov limit for interacting diffusions.
\newblock {\em Math. Nachr.}, 137:197--248, 1988.

\bibitem{harang2020pathwise}
Fabian Harang and Avi Mayorcas.
\newblock Pathwise {R}egularisation of {S}ingular {I}nteracting {P}article
  {S}ystems and their {M}ean {F}ield {L}imits.
\newblock {\em arXiv preprint arXiv:2010.15517}, 2020.

\bibitem{harang2020cinfinity}
Fabian~A. Harang and Nicolas Perkowski.
\newblock ${C}^\infty$-regularization of {ODE}s perturbed by noise.
\newblock {\em Stochastics and Dynamics}, 2021.

\bibitem{hoeksema2020large}
Jasper Hoeksema, Thomas Holding, Mario Maurelli, and Oliver Tse.
\newblock Large deviations for singularly interacting diffusions.
\newblock {\em arXiv preprint arXiv:2002.01295}, 2020.

\bibitem{huang2019distribution}
Xing Huang and Feng-Yu Wang.
\newblock Distribution dependent {SDE}s with singular coefficients.
\newblock {\em Stochastic Processes and their Applications},
  129(11):4747--4770, 2019.

\bibitem{huang2020mckean}
Xing Huang and Feng-Yu Wang.
\newblock Mc{K}ean--{V}lasov {SDE}s with drifts discontinuous under
  {W}asserstein distance.
\newblock {\em arXiv preprint arXiv:2002.06877}, 2020.

\bibitem{jabin2017mean}
Pierre-Emmanuel Jabin and Zhenfu Wang.
\newblock Mean field limit for stochastic particle systems.
\newblock In {\em Active Particles, Volume 1}, pages 379--402. Springer, 2017.

\bibitem{jabin2018quantitative}
Pierre-Emmanuel Jabin and Zhenfu Wang.
\newblock Quantitative estimates of propagation of chaos for stochastic systems
  with ${W}^{-1,\infty}$ kernels.
\newblock {\em Inventiones mathematicae}, 214(1):523--591, 2018.

\bibitem{jabir2019rate}
Jean-Francois Jabir.
\newblock Rate of propagation of chaos for diffusive stochastic particle
  systems via {G}irsanov transformation.
\newblock {\em arXiv preprint arXiv:1907.09096}, 2019.

\bibitem{lacker2018strong}
Daniel Lacker.
\newblock On a strong form of propagation of chaos for {M}c{K}ean--{V}lasov
  equations.
\newblock {\em Electronic Communications in Probability}, 23, 2018.

\bibitem{lasry2007mean}
Jean-Michel Lasry and Pierre-Louis Lions.
\newblock Mean field games.
\newblock {\em Japanese journal of mathematics}, 2(1):229--260, 2007.

\bibitem{le2020stochastic}
Khoa L{\^e}.
\newblock A stochastic sewing lemma and applications.
\newblock {\em Electronic Journal of Probability}, 25, 2020.

\bibitem{mckean1966class}
H.~P.~Jr. McKean.
\newblock A class of {M}arkov processes associated with nonlinear parabolic
  equations.
\newblock {\em Proceedings of the National Academy of Sciences of the United
  States of America}, 56(6):1907--1911, 1966.

\bibitem{mishura2016existence}
Yuliya~S. Mishura and Alexander~Yu Veretennikov.
\newblock Existence and uniqueness theorems for solutions of
  {M}c{K}ean--{V}lasov stochastic equations.
\newblock {\em arXiv preprint arXiv:1603.02212}, 2016.

\bibitem{nualart2006malliavin}
David Nualart.
\newblock {\em The {M}alliavin calculus and related topics}, volume 1995.
\newblock Springer, 2006.

\bibitem{nualart2002regularization}
David Nualart and Youssef Ouknine.
\newblock Regularization of differential equations by fractional noise.
\newblock {\em Stochastic Processes and their Applications}, 102(1):103--116,
  2002.

\bibitem{nualart03stochastic}
David Nualart and Youssef Ouknine.
\newblock Stochastic differential equations with additive fractional noise and
  locally unbounded drift.
\newblock In Evariste Gin{\'e}, Christian Houdr{\'e}, and David Nualart,
  editors, {\em Stochastic Inequalities and Applications}, pages 353--365,
  Basel, 2003. Birkh{\"a}user Basel.

\bibitem{picard2011representation}
Jean Picard.
\newblock Representation formulae for the fractional {B}rownian motion.
\newblock In {\em S{\'e}minaire de Probabilit{\'e}s XLIII}, pages 3--70.
  Springer, 2011.

\bibitem{rockner2018well}
Michael R{\"o}ckner and Xicheng Zhang.
\newblock Well-posedness of distribution dependent {SDE}s with singular drifts.
\newblock {\em arXiv preprint arXiv:1809.02216}, 2018.

\bibitem{serfaty2017mean}
Sylvia Serfaty.
\newblock Mean field limits of the {G}ross-{P}itaevskii and parabolic
  {G}inzburg-{L}andau equations.
\newblock {\em J. Amer. Math. Soc}, 30(3):713--768, 2017.

\bibitem{serfaty2020mean}
Sylvia Serfaty.
\newblock Mean field limit for {C}oulomb-type flows.
\newblock {\em Duke Mathematical Journal}, 169(15):2887--2935, 2020.

\bibitem{stein1970singular}
Elias~M. Stein.
\newblock {\em Singular integrals and differentiability properties of
  functions}, volume~2.
\newblock Princeton university press, 1970.

\bibitem{sznitman1991topics}
Alain-Sol Sznitman.
\newblock Topics in propagation of chaos.
\newblock In {\em Ecole d'{\'e}t{\'e} de probabilit{\'e}s de Saint-Flour
  XIX-1989}, pages 165--251. Springer, 1991.

\bibitem{tanaka1984Limits}
Hiroshi Tanaka.
\newblock Limit theorems for certain diffusion processes with interaction.
\newblock In Kiyosi It\^o, editor, {\em {Stochastic Analysis}}, volume~32 of
  {\em North-Holland Mathematical Library}, pages 469 -- 488. Elsevier, 1984.

\bibitem{tomasevic2020propagation}
Milica Tomasevic.
\newblock Propagation of chaos for stochastic particle systems with singular
  mean-field interaction of ${L}^q-{L}^p$ type.
\newblock {\em HAL-03086253 available at https://hal.inria.fr/hal-03086253},
  2020.

\bibitem{Veretennikov1981}
Alexander~Yu.. Veretennikov.
\newblock On strong solution and explicit formulas for solutions of stochastic
  integral equations.
\newblock {\em Math. USSR Sb.}, 39:387--403, 1981.

\bibitem{villani2008optimal}
C{\'e}dric Villani.
\newblock {\em Optimal transport: old and new}, volume 338.
\newblock Springer Science \& Business Media, 2008.

\bibitem{Zvonkin1974}
A.~K. Zvonkin.
\newblock A transformation of the phase space of a diffusion process that will
  remove the drift.
\newblock {\em Mat. Sb. (N.S.)}, 93(135):129--149, 152, 1974.

\end{thebibliography}
\bibliographystyle{plain}

\end{document}